\documentclass[reqno,10pt]{amsart}
\usepackage{amscd,amssymb,amsmath,amscd,misha,stmaryrd}
\usepackage{latexsym}
\usepackage[all]{xy}
\usepackage[colorlinks=true]{hyperref}

\usepackage[usenames,dvipsnames]{color}

\newcommand{\Log}{{\mathbf{Log}}}
\def\hatX{{\widehat{X}}}
\def\CC{{\mathbb C}}

\def\int{{\rm int}}
\def\cln{{\rm cln}}
\def\res{{\rm res}}
\def\nor{{\rm nor}}
\def\lcm{{\rm lcm}}
\def\sat{{\rm sat}}
\def\ket{{\rm ket}}
\def\cProj{{\mathcal Proj}}
\def\Gm{{\bbG_m}}
\def\word{{\rm w\mbox{-}ord}}
\def\logord{{\rm logord}}
\def\wlogord{{\rm wlogord}}
\def\ocM{{\overline\calM}}
\def\cD{\calD}
\def\cC{\calC}
\def\cR{\calR}
\def\hatcO{{\widehat\cO}}
\def\hatcI{{\widehat\cI}}

\def\hatcN{{\widehat\cN}}
\def\ocJ{{\overline{\cJ}}}
\def\cO{\calO}
\def\cN{\calN}
\def\cF{\calF}
\def\cA{\calA}
\def\cJ{\calJ}
\def\cX{\calX}
\def\cP{\calP}
\def\cZ{\calZ}
\def\cI{\calI}
\def\cM{\calM}
\def\cB{\calB}
\def\cs{{\rm cs}}
\def\rmlog{{\rm log}}
\def\..{,\dots,}
\def\:{{\colon}}
\def\dashto{{\dashrightarrow}}
\def\QQ{{\bbQ}}

\begin{document}

\author{Michael Temkin}
\title{Relative and logarithmic resolution of singularities}
\address{\tiny{Einstein Institute of Mathematics\\
               The Hebrew University of Jerusalem\\
                Edmond J. Safra Campus, Giv'at Ram, Jerusalem, 91904, Israel}}
\email{\scriptsize{michael.temkin@mail.huji.ac.il}}
\thanks{This research is supported by  BSF grants 2014365 and 2018193, ERC Consolidator Grant 770922 - BirNonArchGeom.}

\setcounter{tocdepth}{1}

\maketitle
\tableofcontents

\section{Introduction}

These notes will constitute a chapter in a book on recent advances in resolution of singularities based on a series of minicourses given at an Oberwolfach seminar. My original plan was to provide an expanded version of lecture notes of a mini-course on logarithmic resolution and semistable reduction, but while working on this I decided to widen the perspective and discuss the non-logarithmic methods too -- both the classical method and the weighted (or dream) algorithm. This produces a certain intersection with the material of the other chapters, but I think that the profit is larger than this inconvenience. So, what does one gain? Primarily, when working on the exposition of the logarithmic methods I found a new perspective, which allows to relatively uniformly present and compare all four methods -- the classical one, the logarithmic one, the non-logarithmic weighted method and the logarithmic weighted method. This perspective is slightly new even in the case of the classical algorithm -- marked ideals $(\cI,d)$ are not used, but we blow up $d$-multiple smooth centers instead -- so we will start with a reinterpretation of the classical method (assuming a basic familiarity with it), and then add a logarithmic layer, a stack-theoretic layer, and a weighted layer, each time either obtaining a new algorithm or studying where an attempt fails and preparing to add one more layer to fix this failure. Not only does this presentation follows the order of discovery of these methods and layers, but I hope it also makes the exposition more accessible, since we try to introduce new tools one after another rather than all at once. Second, in any case the weighted logarithmic method of Quek should be covered by this chapter, so it makes sense to start with the non-logarithmic dream algorithm, and only then add the logarithmic layer, indicating the needed logarithmic adjustments.

Our goal is to introduce all relevant notions, constructions and techniques, and to formulate all main results, including all important intermediate results. There are no proofs in the notes. Some easier results are given as exercises and provided with hints, more difficult theorems are provided with references to the literature and a short discussion of main ideas of the proof. So these notes can be viewed as a light guide or a companion for reading research papers, where the new methods were constructed: \cite{ATW-principalization}, \cite{ATW-relative}, \cite{ATW-weighted} (see also \cite{McQuillan}, though it uses a different language) and \cite{Quek}.

\subsection{History and motivation}
Until a few years ago there was known an essentially unique basic functorial method for principalization of ideals and resolution of varieties in characteristic zero, to which we refer in the sequel as ``the classical method''. It was distilled during decades from Hironaka's original method from \cite{Hironaka}, and in this joint and very long effort took part Hironaka himself (idealistic exponents), Giraud (maximal contact), Villamayor and Bierstone-Milman (canonicity), W{\l}odarczyk (smooth functoriality), and others. In particular, until 2017 it was not clear if there exist other algorithms, especially the ones which are simpler, faster or possess better functorial properties. These questions, fundamental by themselves, are especially important in view of the fact that a similar method fails in two other classical desingularization problems: resolution in positive characteristic and resolution of vector fields (in characteristic zero). So, enriching the pool of ideas and techniques can be critical in order to achieve a substantial progress in these questions.

The first advance beyond the classical settings was done when a logarithmic analogue of the classical method was constructed for varieties in \cite{ATW-principalization} and then extended to morphisms (or semistable reduction theorems) in \cite{ATW-relative}. The original motivation for this project was twofold: 1) we wanted to obtain a functorial semistable reduction theorem, which will extend, in particular, to any valuation ring, not necessarily discrete, 2) we wanted to clarify the role of log structures in the classical algorithm, where it was visible but a bit opaque. Both lines pointed in the same direction:

1) Over a general valuation ring the best one can hope for is the semistable reduction in the sense of Abramovich-Karu (see \cite{AK} and \cite{semistable}), and this indicates that one has to work with arbitrary fs log structures and not only those with free monoids $\oM_x$. In addition, the notions of smoothness, derivations, smooth functoriality, etc. should be replaced by log smoothness, log derivations, log smooth functoriality, etc.

2) In this context it was natural to seek for principalization on general log smooth varieties as opposed to smooth varieties with an snc boundary (or the induced log structure).

All in all, one is led to replace the classical setup by the logarithmic one, and the rest, to our surprise, was rather imposed upon us (though not easy to discover). In particular, it turned out that the only sufficiently general class of logarithmic centers that preserve log smoothness are intersections of subvarieties with monomial centers. A relevant log algorithm is even simpler than the classical one because no induction on the log stratification is needed anymore, but it got stuck at one place and insisted that we also blow up centers generated by roots of monomials. Such a blowing up may be not log smooth, and to resolve this we had to refine such blowings up to stacks -- adding the stack-theoretic layer is the solution which seemed to us technical but unavoidable. In fact, we just introduced a stack-theoretic refinement of the classical weighted blowing up (with a certain pattern of weights). Note that analogous obstacles were discovered earlier in resolution of vector fields, and the solution also was to consider non-representable weighted blowings up (or any equivalent tool playing the same role), see \cite{Panazzolo} and \cite{McQuillan-Panazzolo}.

Once a new pool of blowings up that preserve smoothness (in the setting to stacks) was discovered, the next natural question was to study which improvement to the classical algorithm can be obtained using them all. It was independently studied in \cite{McQuillan} (following \cite{McQuillan-Panazzolo}) and in \cite{ATW-weighted} (following \cite{ATW-principalization}) and, as in the classical case, led to the same algorithm, despite different description and justification. Quite to our surprise, the natural principalization algorithm in this setting does not involve any log structure, does not have any memory, uses a simple (in fact, classical) multi-order invariant and improves it after each single weighted blowing up. Moreover, even the resulting resolution algorithm has the same properties, and it is really a non-embedded method. Finally, such an algorithm was believed not to exist (and it dose not exist in the classical setting, see \S~\ref{nodream}). For all these reasons we sometimes call it a dream algorithm, despite the fact that (unfortunately for us) some experts consider it as a variation of old ideas, which does not contain anything essentially new...

Finally, one may ask if a dream algorithm also exists in the logarithmic situation, in particular, leading to a fast and simple resolution of schemes with divisors and semistable reduction theorem. The answer is yes. The absolute case was worked out in \cite{Quek}, and it seems certain that the similar method will also apply in the relative setting.

\subsection{An overview}
Now let us outline the content of these notes.

\subsubsection{The general principles}
In Section \ref{principles} we formulate the general principles that apply to all four methods known so far. In fact, this is an excellent time for such a classification -- we already have a few known functorial methods (unlike the past decades, when only the essentially unique classical method was known), but still quite a few methods, so that such a generalization is possible... In brief, quite surprisingly each algorithm seems to be quite determined by what we call the framework of the method: the class of geometric objects one works with, the relevant notions of smoothness and derivations, the admissible centers one can blow up without destroying the smoothness, and a primary classification of admissible centers by a (partially or totally) ordered set whose elements are called orders, log orders, weighted orders, etc. Resolution of $Z$ is always deduced from principalization of $I_Z\subset\cO_X$ on a manifold $X$ in which $Z$ is (locally) embedded, and the principalization of $\cI\subseteq\cO_X$ is deduced from appropriate order reduction of $\cI$, in which one iteratively blows up an admissible center $\cJ$ which contains $\cI$ and has maximal possible order, and then factors out the pullback of $\cJ$ from the pullback of $\cI$.

It turns out that what is usually viewed as the main machinery, including the heavy one -- maximal contacts, coefficient ideals, homogenization, independence of the embedding, etc. -- generalizes quite easily to any setting, once an appropriate framework is chosen. And what we viewed as technical aspects in \cite{ATW-principalization} and \cite{ATW-weighted} -- extending varieties to DM stacks, introducing an appropriate formalism of new centers, such as $(x^2,y^{3/2})$, etc. -- seems to be the main choices which one has to carefully design. A wrong or insufficient choice often leads to an ''almost'' algorithm which gets stuck at an unexpected innocently, or technically looking, point.

We finish Section \ref{principles} with an illustration of these general principles on the case of the classical algorithms -- first we show how far one can go without boundaries, and then add this additional layer to the framework. In this case, one gets stuck because the order can jump on the maximal contact, so one has to consider order reductions of non-maximal order, and Hironaka's insight was that this can be done once the excess of the exceptional divisor is well controlled by the boundary. If one uses precise weighted centers, the exceptional divisor is cleared off in a more precise way, and this explains why the basic weighted algorithm uses no boundary (or log structure) at all; this is the only method known so far which does not use log structures.

A familiarity with the classical methods is assumed in our exposition, so we refer to chapter \cite{classical} and to the usual sources, such as \cite{Kollar} and \cite{Wlodarczyk}. Also, a very short and clear exposition (with some proofs omitted) can be found \cite{Bierstone-Milman-funct}.

\subsubsection{The logarithmic methods}
In the next two sections we explain in detail the logarithmic analogue of the classical method, which was constructed in \cite{ATW-principalization}. In Section~\ref{firstsec} we proceed as much as one can within the framework of log varieties (as the reader can imagine, this was precisely our first line of research when working on \cite{ATW-principalization}). The slogan of this part is to add ``log'' everywhere: log varieties, log smoothness, log smooth functoriality, log derivations and the associated notions of log order, log maximal contact and log coefficient ideal. The main novelty is that one can consider centers, which are powers of the centers of the form $(t_1\..t_n,u^{p_1}\..u^{p_r})$, where $t_i$ are regular parameters and $p_i$ are arbitrary monomials.

It turns out that in this fashion one almost obtains a perfect algorithm, which fails only at one point -- this time the failure also happens in the log order reduction of non-maximal log order $d$, but, in addition, the problem only pops up when the log order of $\cI$ is infinite. The solution this time is to allow blowings up of Kummer centers defined also by Kummer monomials like $u^{p/d}$. This requires to extend both the formalism of such ideals and of their blowings up. The first task is solved in the Kummer \'etale topology and the second one forces one to consider stacks and non-representable modifications. The theory of Kummer centers and blowings up is developed in the beginning of Section~\ref{sec4}, and then the same logarithmic algorithm constructed earlier works perfectly well.

Section \ref{morsec} is devoted to extending the absolute logarithmic methods to morphisms, following \cite{ATW-relative}. In fact, precisely the same algorithm works once one replaces absolute log derivations by relative ones. The only serious novelty is that one should take base changes into account. On the positive side, the algorithm is compatible with arbitrary base changes with a log regular source -- a new type of functoriality. However, there is no free lunch, and another new feature is that the algorithm can fail, and in order for it to succeed one has to modify the base first. Non-surprisingly, once again the failure can happen at the ``simple'' monomial stage. Much more surprisingly is that we could not find a simple way to prove a monomialization theorem, which guarantees that the monomial step succeeds after a large enough base change, see \S\ref{monomsec}. The existing proof is non-canonical, and we expect further progress to be possible.

\subsubsection{The weighted methods}
Finally, in Section \ref{dreamsec} we construct weighted algorithms which blow up arbitrary $\bbQ$-regular (or weighted) centers both in the non-logarithmic and logarithmic settings. We follow \cite{ATW-weighted} in the non-logarithmic case and we follow \cite{Quek} in the logarithmic ones. Note also that the same algorithm was constructed by McQuillan in \cite{McQuillan}, but the presentation uses a much more coordinate dependant language, so it is rather far from ours. The main idea is that we would like to blow up centers like $\gamma=(t_1^{q_1}\..t_n^{q_n})$ which might lead to a singular output, but this can be resolved by blowing up an appropriate root $\gamma^{1/n}$, obtaining a smooth stack-theoretic refinement of $Bl_\gamma(X)$. In order to make sense of things like $\gamma^{1/n}$ one has to introduce a new formalism of generalized ideals. In fact there exist a few ways to deal with this -- valuative $\bbQ$-ideals, $\bbQ$-ideals (which are equivalent to Hironaka's characteristic exponents) and Rees algebras, and we discuss them and relations between them in the first three subsections of \S\ref{dreamsec}.

The paradigm of blowing up general $\bbQ$-regular ideals leads to what we call dream algorithms which introduce a simple invariant -- the weighted (log) order, and improve it by a single blowing up along an $\cI$-admissible $\bbQ$-regular center of maximal possible weighted (log) order. This results in what we call dream algorithms which require no history and simply repeat the same basic operation of weighted (log) order reduction. In addition, one obtains a non-embedded resolution which acts in the absolutely same manner -- one simply blows up the unique maximal $\bbQ$-regular center contained in the scheme, and this blowing up improves the invariant. The logarithmic weighted algorithm is constructed very similarly but using the logarithmic setting. The main difference is that one also has to add a monomial part to the $\bbQ$-regular center, and one should take such a part as small as possible.

\subsubsection{Conventiones}
Unless stated otherwise, we will always work over a ground field $k$ of characteristic zero.

By a {\em blowing up} we always mean a morphism $X'=Bl_\calI(X)\to X$ with the ideal $\calI\subseteq\calO_X$ being part of the datum. Thus, the same morphism can underly different blowings up. By a slight abuse of language, saying that a morphism $f$ is a blowing up without specifying the center we always mean that $f$ underlies a blowing up along some center (in particular, it is projective).

All log schemes are fs, see \cite{lognotes}.

\section{General principles}\label{principles}
In this section we will discuss principles and features shared by all known functorial resolution algorithms. In particular, we choose a presentation which might look a bit strange to the reader familiar only with the classical algorithm, but it extends naturally to other settings. We will end the section with a short description and re-interpretation of the classical algorithm in the new framework.

\subsection{Frameworks}

\subsubsection{Modifications}
By a {\em modification} we mean a proper morphism $X'\to X$ which establishes an isomorphism $U'= U$ of dense open subschemes (subspaces, substacks, etc.) This definition applies to non-reduced objects as well, though we will usually work with generically reduced ones. For example, a blowing up is a modification if and only if its center is nowhere dense.

\subsubsection{Basic choices of an algorithm}\label{basicsec}
Each algorithm makes a few basic choices that we list below and call the {\em framework} of the algorithm.

\begin{itemize}
\item[(0)] The category $\gtC$ of geometric objects the algorithm deals with, that will be called {\em spaces}, and the corresponding topology. For example, varieties over a fixed or varying fields, schemes with enough derivations, analytic spaces, stacks, log schemes, etc. The topology can be Zariski, \'etale, analytic, etc.
\item[(1)] The class of {\em regular} spaces that will be called {\em manifolds}, and the class of {\em regular} morphisms. For example, smooth varieties, regular schemes, log smooth log varieties, etc., and smooth morphisms of varieties over a fixed field, regular morphisms between varieties over varying fields, log smooth morphisms of log varieties, etc.
\item[(2)] The class of {\em admissible modifications} $f\:X'\to X$ with $X$ and $X'$ manifolds. It will always be a variant of a blowing up along an {\em admissible center} $\cJ$ (or simply a {\em center}), so we will use the notation $X'=Bl_\cJ(X)$. In particular, the pullback $\cJ\cO_{X'}$ is always an honest invertible ideal which defines the {\em exceptional divisor} $E_f\subset X'$. The center itself is an ideal in an appropriate topology, which can be rather fancy. For example, Kummer \'etale topology or $h$-topology.
\item[(3)] A {\em primary invariant} which takes values in a totally or partially ordered set and more or less classifies different types of admissible centers. We call it the {\em order} of the center and extend to arbitrary ideals as follows: the order $\ord_X(\cI)$ of $\cI$ on $X$ is the maximal order of an {\em $\cI$-admissible} center, that is, an admissible center $\cJ$ such that $\cI\subseteq\cJ$. Examples include the order of ideal, the log order and the weighted order $(d_1\..d_n)$ of a weighted center $(t_1^{d_1}\..t_n^{d_n})$.
\item[(4)] A theory of derivations on manifolds. This amounts to choosing large enough sheaves of derivations one works with. For example, $k$-derivations or absolute derivations over $\bbQ$.
\end{itemize}

\begin{rem}
(i) Choices (0)--(2) will be called the {\em basic framework} of the method. Choices (3) and (4) seem to be dictated by the basic framework, at least to a large extent.

(ii) Concerning the choice of admissible blowings up, the general principle is that one should try to choose as large a class as possible with the restriction that the centers $\cJ$ should possess a simple explicit description. In the cases we know, it seems that the framework essentially dictates a unique natural algorithm corresponding to it, and the larger the class of admissible blowings up is, the better algorithm one obtains. Even in the classical setting it is beneficial to consider centers of the form $\cJ^d$, where $\cJ=I_V$ defines a smooth subvariety $V$. Since $Bl_{\cJ^d}(X)=Bl_\cJ(X)$ this might look as a simple bookkeeping of the order $d$ inside the center, but we will argue that in this form the description of the algorithm becomes both more natural and more similar to the logarithmic and weighted algorithms, e.g. see Remark~\ref{dcenterclassical}.
\end{rem}

\subsubsection{Functoriality}
All methods we will consider are functorial in the following strong sense: they are compatible with surjective regular morphisms. In particular, this implies that the method can be constructed \'etale locally (or even smooth locally), and once this is done the method globalizes by \'etale descent. This fact simplifies arguments tremendously as they become essentially of a local nature (see also \S\ref{synchr}). Compatibility with non-surjective regular morphisms holds on the level of a morphism, but not a finer structure of the principalization sequence of blowings up, and we will touch on this delicate issue later.

\begin{rem}
(i) Historically, the first {\em canonical} algorithms (i.e. compatible with automorphisms) were constructed by Bierstone-Milman in \cite{Bierstone-Milman} and Villamayor in \cite{Villamayor}, and W{\l}odarczyk was the first to emphasize on smooth functoriality and use it in an essential way in constructing the algorithm, see \cite{Wlodarczyk}.

(ii) Smooth functoriality implies that the algorithm is equivariant with respect to any group scheme action, as any group scheme in characteristic zero is smooth.
\end{rem}

\subsection{Principalization and resolution}
The main result of each desingularization method is an appropriate principalization theorem, and as a consequence one obtains a non-embedded desingularization theorem.

\subsubsection{Functorial resolution}
Let $\bfP$ be a class of regular morphisms (usually all regular morphisms in $\gtC$). By a {\em $\bfP$-functorial resolution} on $\gtC$ we mean a rule $\cR$ which associates to each object $X$ of $\gtC$ a modification $\cR(X)\:X_\res\to X$ with a regular source in such a way that $Y_\res=Y\times_XX_\res$ for any $\bfP$-morphism $Y\to X$ in $\gtC$. The main resolution theorem for a given framework asserts that such a resolution exists. In addition, the desingularization morphisms $\cR(X)$ are projective (or non-representable global quotients of projective morphisms). In fact, they are naturally equipped with a structure of a composition of explicit blowings up, but this factorization is only compatible with surjective regular morphisms.

\subsubsection{Functorial principalization}
Let $X$ be a manifold and $\cI$ an ideal on $X$. An admissible blowing up $X'=Bl_\cJ(X)\to X$ is called {\em $\cI$-admissible} if $\cI\subseteq\cJ$. In such a case, $\cI\cO_{X'}$ is contained in the invertible ideal $I_E=\cJ\cO_{X'}$, hence the {\em transform} $\cI'=\cI\cO_{X'}I_E^{-1}$ is defined. An {\em $\calI$-admissible sequence} $f_\bullet\:X_n\dashto X_0=X$ is a sequence of $\calI_i$-admissible blowings up $f_i\:X_{i+1}\to X_i$ such that $\calI_0=\calI$ and $\calI_{i+1}$ is the transform of $\calI_i$. Such a sequence is called a {\em principalization} of $\calI$ if $\calI_n=\calO_{X_n}$ is trivial.

By a {\em $\bfP$-functorial principalization} on $\gtC$ we mean a rule which associates to any ideal $\cI$ on a manifold $X$ in $\gtC$ a principalization $\cP(\cI)\:X_n\dashto X_0=X$ of $\cI$ in such a way that for any $\bfP$-morphism $Y\to X$ in $\gtC$ and $\cI'=\cI\cO_Y$ the sequence $\cP(\cI')$ is obtained from the pullback $\cP(\cI)\times_XY$ of $\cP(\cI)$ by omitting all trivial blowings up. The main principalization theorem asserts that such a principalization exists.

\begin{rem}
If $Y\to X$ is surjective, then $\cP(\cI')=\cP(\cI)\times_XY$, but in general the blowings up along centers whose image in $X$ is disjoint from the image of $Y$ are pulled back to trivial blowings up and hence ignored.
\end{rem}

\subsubsection{Synchronization}\label{synchr}
The above remark indicates that the principalization algorithm is not of local nature in the strict sense. For example, if $X=\cup_{i=1}^nX_i$ is an open covering, one cannot reconstruct $\cP(\cI)$ from $\cP(\cI|_{X_1})\..\cP(\cI|_{X_n})$ without an additional synchronization data -- what are the trivial blowings up we removed after the restriction. However, if $X'=\coprod_{i=1}^nX_i$ and $\cI'=\cI\cO_{X'}$, then all these blowings up are kept in the sequence $\cP(\cI')$, and as we remarked earlier $\cP(\cI)$ is easily reconstructed from $\cP(\cI')$. Informally speaking, when principalizing $\cI'$ the method has to compare the singularities of $\cI|_{X_i}$ and decide which one is blown up earlier (or simultaneously), entering trivial blowings up at the other places. This is precisely the needed synchronization datum.

\begin{rem}
(i) The above argument shows that it is important to consider simultaneous principalization on disconnected manifolds, and the method is only ``local up to disjoint unions'' or quasi-local accordingly to the terminology of \cite{semistable}.

(ii) Another way to establish a synchronization of local constructions is by use of an explicit invariant, for example, as \cite{Bierstone-Milman} do. In fact, using the trick with disjoint unions is equivalent to the use of an abstract invariant, see \cite[Remark~2.3.4]{Temkin-qe}.
\end{rem}

\subsubsection{The re-embedding principle}
There is one more important functoriality property satisfied by all known methods  called the re-embedding principle. We say that a principalization method $\calP$ on $\gtC$ satisfies the {\em re-embedding principle} if for any closed immersion of manifolds {\em of constant codimension} $X'\into X$ an ideal $\cI'$ on $X'$ and its preimage $\cI\subseteq\cO_X$ the blowing up sequence $f_\bullet=\calP(\calI)$ is obtained by pushing forward the blowing up sequence $f'_\bullet=\calP(\calI')$, that is, the centers $\cJ_i$ of $f_\bullet$ are the preimages of the centers $\cJ'_i$ of $f'_\bullet$ and (by induction on the length) each $X'_i\into X_i$ is the strict transform of $X'\into X$.

\subsubsection{Reduction to principalization}
In all settings the appropriate desingularization theorem is a relatively easy corollary of the principalization theorem. Loosely speaking the general principle can be formulated as follows:

\begin{principle}
If there exists a $\bfP$-functorial principalization on $\gtC$ which satisfies the re-embedding principle, then there exists a $\bfP$-functorial desingularization $\calR$ on the class of locally equidimensional generically reduced spaces from $\gtC$ which locally possess a closed immersion into a manifold.
\end{principle}

The embeddability assumption is automatic for varieties, formal varieties or analytic spaces, and is only relevant for general excellent schemes. The local equidimensionality condition is used to construct an embedding of constant codimension. The argument in all settings is essentially the same: to resolve a space $Z$, one locally embeds it into a manifold $X$ and constructs the resolution of $Z$ from the principalization $\calI_Z$. Loosely speaking, before blowing up a generic point $\eta\in Z$ the principalization has to guarantee that it is a generic point of an admissible center, and in all methods this amounts to resolving the Zariski closure of $\eta$. Moreover, because of the codimension assumption, all generic points of $Z$ are blown up simultaneously at some blowing up $X_{i+1}\to X_i$ and its center contains a component which is the strict transform of $Z$. In particular, $Z_i\to Z$ is the induced desingularization of $Z$. Independence of the embedding in all methods is proved by use of the re-embedding principle and a simple computation showing that an embedding of minimal possible codimension is unique \'etale locally (or formally locally).

\begin{rem}\label{strong}
(i) The desingularization morphism $Z_i\to Z$ is naturally a composition of blowings up $Z_{j+1}\to Z_j$, $0\le j\le i-1$ with centers $\cJ_j\cO_{Z_j}$. In the classical resolution each center $V_j=V(\cJ_j)$ is smooth, but the intersection $V_j\times_{X_j}Z_j$ can be singular. In particular, the factorization of $Z_i\to Z$ into a composition of blowings up is not too informative.

(ii) So-called {\em strong resolution} methods construct a resolution which is a composition $Z_i\dashto Z_0=Z$ of blowings up along smooth centers. Perhaps the main advantage of this is that for any closed immersion $Z\into Y$ into a manifold the resolution automatically extends to a modification of manifolds $Y_i\to Y$ with $Z_i$ a closed subscheme in $Y_i$: just consider the pushout $Y_i\dashto Y$ of the sequence $Z_i\to Z$.

(iii) The only known method to construct strong resolution is to force the condition that $V_j\subseteq Z_j$ in the principalization sequence, and hence the whole sequence $X_i\dashto X$ is the pushforward of the sequence $Z_i\dashto Z$. In the classical case this is achieved by serious additional work building on the usual principalization (the so-called presentation of the Hilbert-Samuel function in \cite{Bierstone-Milman}). We will see that in the weighted desingularization methods strong factorization is achieved just as a by-product.
\end{rem}

\subsubsection{The miracle}
The reduction of resolution to a seemingly very different principalization problem is usually viewed as a brilliant trick if not a miracle. Nevertheless, we claim that this is not so surprising. An equivalent formulation of existence of resolution is that manifolds are cofinal among the set of all modifications of a generically reduced space $X$. When one studies modifications of a manifold $X$, it is hard (if not impossible) to explicitly describe all modifications $X'\to X$ with $X'$ a manifold, so it is natural to only consider basic explicit modifications of this form -- admissible blowings up and their composition. The principalization theorem asserts that for any ideal $\cI$ there is an admissible sequence $X'\dashto X$ which principalizes $\cI$, and by the universal property of blowings up this happens if and only if the morphism $X'\to X$ factors through $Bl_\cI(X)$. By Chow's lemma blowings up form a cofinal family of modifications of $X$, hence the principalization just asserts that admissible sequences form a cofinal family of modifications of a manifold. In this form, it is rather natural to expect that the theorems are related and the principalization theorem is finer.

\begin{rem}
One may also wonder if the following weak factorization conjecture holds: any modification of manifolds can be factored into a composition of admissible blowings up and blowings down. This conjecture provides the next level of depth. In the classical case the only known argument deduces it with a large amount of work from $\Gm$-equivariant principalization in the next dimension. In other settings this is still open, though we expect that an analogous approach with birational cobordisms and $\Gm$-equivariant principalization should work there too. It would be especially interesting to check this for semistable models and morphisms.
\end{rem}

\subsubsection{Order reduction}
In first approximation, the principalization is achieved by successive order reduction procedures: reduce the order $d=\ord_X(\cI)$ of $\cI$ by blowing up centers of order $d$. In the non-weighted algorithms one reduces this problem to an order reduction on a maximal contact hypersurface. However, the order can jump under this reduction, so for inductive reasons one also has to solve the problem of reducing the order of $\cI$ below $e$ only by blowings up $e$-centers for any fixed value $e\le d$ of the invariant. This results in the accumulation of exceptional divisors in the transform, and one has to use a log structure to control this -- guarantee that the exceptional divisor is monomial and deal with it mainly by combinatorial methods. In fact, this is the only place in the algorithms, where some flexibility can take place.

In weighted algorithms the order reduction is done by a single weighted blowing up, so they are what we call {\em dream algorithms} -- no history is needed, each blowing up is independent of the rest and reduces the invariant further. However, the argument that a unique maximal $\cI$-admissible center exists is rather complicated and, again, uses the theory of maximal contact and homogenization. In particular, it completely fails in positive characteristic.

\subsection{The classical algorithm: a first attempt}\label{classicalfirst}
The classical method was already discussed in detail in chapter~\cite{classical}, so we assume that the reader is familiar with the main ideas and constructions, and our goal is to briefly re-interpet it within the general paradigm we described earlier in this section. Later we will develop the logarithmic algorithm pretty much in the same spirit. In \S\ref{classicalfirst} we will discuss what can be done without the boundary and where this attempt fails. However, all constructions we are going to describe are relevant, and in the next subsection, we will just add the boundary as an additional layer of the framework. For simplicity, we work with $k$-varieties.

\subsubsection{The framework}
One considers the category $\gtC$ of varieties over $k$, manifolds are smooth varieties and the algorithms will be smooth functorial. Admissible centers are of the form $\cJ=I_V^d$, where $I_V\subset\cO_X$ is the ideal of a submanifold $V\subset X$ and $d\ge 1$. We call such a center a $d$-center. An admissible blowing up is the usual blowing up of the center. The derivation theory is given by the sheaves $\cD_X=\Der_{X/k}$ of $k$-derivations and the sheaves $\cD_X^{(\le d)}$ of differential $k$-operators of order at most $d$. The primary invariant of the center is just the {\em multiplicity} $d$ of $\cI^d$.\footnote{We use the notion of the multiplicity of a center instead of the order to avoid confusion with the general order of ideals it is used to define. This is justified because the classical multiplicity of $V(\cI^d)$ at any its point is $d$.} The order of an arbitrary ideal $\cI$ at $x$ is the maximal $d$ such that $\cI_x\subseteq\cI_V^d$. Clearly, it suffices to take the center $V=\{x\}$, and then we obtain the usual definition of the order.

\begin{rem}\label{dcenterclassical}
Classically one only considers 1-centers, works with marked ideals $(\cI,d)$ and uses a $d$-transform after blowings up along a smooth center $V$ that lies inside the locus of points where the order is at least $d$. This is equivalent to our admissibility condition $\cI\subseteq \cJ=I_V^d$ and using the usual transform with respect to the blowing up along $\cJ$. So, we just provide a slightly different interpretation of the classical constructions.
\end{rem}

\subsubsection{Derivations}
Derivations of ideals provide a convenient way to describe all basic ingredients of the algorithm (except the boundary):

\begin{itemize}
\item[(1)] The maximal order $\ord_X(\cI)$ of $\cI$ on $X$ is the minimal number $d$ such that $\cD_X^{(\le d)}(\cI)=\cO_X$. The order of $\cI$ at a point $x\in X$ is computed similarly.
\item[(2)] A maximal contact to $\cI$ at $x$ is any closed smooth subscheme $H\into X$ which locally at $x$ is of the form $V(t)$ with $t\in\cD_X^{(\le d-1)}(\cI_x)$.
\item[(3)] The homogenized coefficient ideal is the homogenized weighted sum of derivations, which are weighted by their orders: $$\cC(\cI)=\sum_{a\in\bbN^d\ :\ \ \sum_{i=0}^{d-1}a_i(d-i)\le d!}\ \ \prod_{i=0}^{d-1}\left(\cD_X^{(\le i)}(\cI)\right)^{a_i}.$$
\end{itemize}

\begin{rem}\label{Rem:coefficient-ideals}
Usual coefficient ideals are defined using only the corresponding powers of the derivations of $\cI$, but the homogenized version is integral over it, and hence can be used instead. The homogenized coefficient ideals were introduced by Koll\'ar, see \cite[\S3.54]{Kollar}. They subsume the homogenization procedure of W{\l}odarczyk, see \cite[\S2.9]{Wlodarczyk}.
\end{rem}

\subsubsection{Order reduction}
If $e\le d=\ord(\cI)$, then an {\em order $e$-reduction of $\cI$} is an $\cI$-admissible sequence of blowings up along $e$-centers $X_n\dashto X_0=X$ such that $\ord_{X_n}(\cI_n)<e$.

\begin{rem}
Usually one talks about order reduction of a marked ideal $(\cI,d)$ by blowing up smooth centers, and $d$ indicates which power of the exeptional divisor to factor out on each transform. The two languages are equivalent.
\end{rem}

\subsubsection{The maximal order case}
The main loop of classical principalization iteratively performs order reduction with $e=d$ -- the so-called maximal order case. In this case, the theory of maximal contact implies that for any maximal contact $H$ (which exists locally) pushing out from $H$ to $X$ establishes a one-to-one correspondence between order $d$-reductions of $\cI$ and order $d!$-reductions of $\cC(\cI)|_H$, so we can apply induction on dimension. Moreover, for any other maximal contact $H'$ the restrictions of $\cC(\cI)$ to $H$ and $H'$ can be taken one to another by an \'etale correspondence, hence the construction is independent of choices and globalizes.

\subsubsection{General order $e$-reduction}
It can happen that $\ord_H(\cC(\cI)|_H)>d!$ and so the induction forces one to also consider the non-maximal order case. The trick is to reduce this to the maximal order case by controlling the accumulated exceptional divisor, and for this job one has to add one more layer to the framework -- the boundary.

\subsection{The classical algorithm: the boundary}

\subsubsection{The framework}
In fact, instead of manifolds $X$ one works with pairs $(X,E)$, where the {\em boundary} (or the exceptional divisor) $E$ is an snc divisor. In some versions, one also orders components of $E$ by a history function. One restricts the set of the admissible centers $\cJ=I_V^d$ by requiring that $V$ has simple normal crossings with $E$, and then the boundary $E'$ on $X'=Bl_\cJ(X)$ is combined from the preimage of $E$ (the old boundary) and the exceptional divisor of $X'\to X$ (the new boundary). This guarantees that $E'$ is snc. However, one still has to struggle with two complications mainly caused by the fact that one uses all derivations rather than those that preserve $E$, so all constructions are not well-adapted to $E$ and one has to fix this essentially by hand. Fortunately, this can be done by two tricks.

\subsubsection{Removing the old boundary}
The first complication is that a maximal contact $H$ does not have to be transversal to $E$, so $E|_H$ does not have to be a boundary. This is resolved by separating $H$ and the old boundary by iterative order reduction of $\cI$ along the maximal multiplicity strata of the old boundary. This trick forces one to introduce a secondary invariant $s$ -- the number of the old boundary components remaining since a maximal contact was created. As a result, the (non-normalized, see below) total invariant is $(d_0,s_0;d_1,s_1,\dots)$ rather than just the string of orders $(d_0,d_1,\dots)$. 

\subsubsection{The normalized degrees}
In addition, one usually normalizes the orders by $q_i=d_i/\prod_{j<i}(d_j-1)!$ so that $q_0=d_0$, but the other degrees can be non-integral. In particular, this choice is made by Bierstone and Milman, see \cite{Bierstone-Milman}, and it is made in \cite{ATW-weighted}, but not in \cite{ATW-principalization}. It is more natural, for example, $(a_1,0;a_2,0;\dots;a_n,0;\infty)$ is the normalized invariant of $V(t_1^{a_1}+\dots t_n^{a_n})$. We will use the normalized choice also in the logarithmic and weighted algorithms. In the latter case, this is a ``no-brainer" choice.

\subsubsection{The companion ideal}
Order $e$-reduction of $\cI$ is done by splitting it into the product $\cI^\cln\cN$ of the maximal invertible monomial factor $\cN$ and the remaining non-monomial part, which will be called {\em clean}. Until $\ord_X(\cI^\cln)\ge e$ we simply apply maximal order reduction to $\cI^\cln$. To proceed further we should take into account $\cN$ and the fact that the order is affected by both $\cI^\cln$ and $\cN$. Fortunately, the contribution of $\cN$ is locally constant along the strata of $E$, so again one can design an order reduction by a careful work along the strata. Technically, this is done by a trick with the companion ideal. Finally, when $\cI^\cln=\cO_X$ one resolves $\cI=\cN$ by purely combinatorial methods.

\begin{rem}
(i) The classical algorithm has more complicated structure than the recently discovered ones. Probably, the main reason for this is that the boundary is not fully built into the framework -- it is not respected by the derivations and its categorical meaning is not so evident. In a sense, it is an additional layer added in an ad hoc manner, and various incompatibility problems are also solved ad hoc.

(ii) It is observed in \cite{Bierstone-Milman-funct} that the sheaf $\cD_{(X,E)}$ of logarithmic derivations fits various constructions, including the chain rule for transform of derivations, much better than $\cD_X$, but one still has to work with $\cD_X$ because it computes the order.
\end{rem}

\subsubsection{The log structure}
In fact, boundaries do have a categorical interpretation -- what one really uses in the classical principalization is the log structure $\cM(\log(E))$ defined by $E$, rather than a divisor. In particular, monomial ideals are defined using the log structure and admissible blowings ups $(X',E')\to(X,E)$ are morphisms of log schemes, but not of scheme-divisor pairs. Thus, manifolds in the classical principalization are in fact log smooth log schemes $X$ with free monoids $\ocM_x$ (equivalently, $X$ is smooth). Furthermore, the induced resolution $f\:Z_\res\to Z$ automatically satisfies the following condition: the exceptional divisor of $f$ is an snc divisor. Indeed, when the strict transform $Z_\res$ of $Z$ is blown up during the principalization of $I_Z$ on a manifold $X$, it has simple normal crossings with the boundary $E$ and is not contained in $E$, hence $E|_{Z_\res}$ is snc. Finally, starting principalization with a non-empty $E$ one can also resolve embedded log schemes $Z$ such that the monoids $\ocM_z$ are free.

We have illustrated that the classical algorithm possesses certain logarithmic aspects. However, it is not functorial for log smooth morphisms, it works only with log structures of a special form, and it does not use log derivations. The natural question whether one can remove the restrictions on log structures and work log smooth functorially was one of the main motivations for a research which led us to the discovery of the logarithmic algorithm, constructed in the next sections.

\section{Resolution of logarithmic schemes: a first attempt}\label{firstsec}
Logarithmic algorithms are obtained by switching to the logarithmic framework: log schemes, log smooth functoriality, log derivations, etc. This makes the framework more complicated, but has numerous advantages: the principalization algorithm becomes simpler and faster, the functoriality is stronger, the method extends to morphisms. The main technical complication is that one is forced to extend the category to stacks or, alternatively, consider cobordant blowings up \cite{weighted}, which increase the dimensions.

In this section we will describe all desingularization tools provided by log geometry of varieties but not involving stacks. This makes the presentation simpler and illustrates the algorithm, and it will be easy for the exposition to add a stack-theoretic layer separately in the next section. To get used to the \'etale topology, already in this section we do not assume that the log structure is Zariski. In the end of this section we will construct a potential algorithm and detect the only place where it fails without using stacks.

\subsection{The framework}
We start with describing the manifolds and the admissible blowings up of the logarithmic principalization. Again the reader is referred to Chapter \cite{lognotes} in this volume for an introduction and general notation.

\subsubsection{Log schemes}
The basic geometric category we work with is the category of fs log schemes of finite type over a field $k$ of characteristic zero, in particular, all fiber products are saturated.

\subsubsection{Log smooth functoriality}
We will tacitly check (or at least mention) that all basic ingredients of our method are compatible with log smooth morphisms. This will guarantee that the desingularization and principalization we will construct are functorial with respect to all log smooth morphisms.

\subsubsection{Manifolds}
By a {\em log manifold} we mean a log smooth log variety over $k$. Recall that these are the same as the classical toroidal varieties and \'etale-locally they are of the form $X=\Spec(k[P][t_1\..t_n])$ with the log structure given by a toric monoid $P$. By a {\em submanifold} we mean any strict closed immersion $V\into X$ with $V$ a log manifold. Basic properties of logarithmic smoothness imply that these notions are log smooth functorial:

\begin{lem}
Let $Y\into X$ be a strict closed immersion of log varieties, $f\:X'\to X$ a log smooth morphism and $Y'=Y\times_XX'$. If $X$ is a log manifold and $Y$ is a log submanifold, then $X'$ is a log manifold and $Y'$ is log submanifold. If $f$ is surjective, then the converse is also true.
\end{lem}

\subsubsection{Parameters}
Any log manifold $X$ is log regular, so the description from \cite[\S3.4.6]{lognotes} applies. Recall that the log strata are regular (and they are well-defined even if the log structure is not Zariski, see \cite[Exercise ~3.2.2(ii)]{lognotes}). Let $C$ be the log stratum containing a point $x$. We say that $t\in\calO_{X,x}$ is a {\em regular parameter} if its image in $\calO_{C,x}$ is a regular parameter. Similarly, $t_1\..t_n\in\calO_{X,x}$ form a {\em regular family of parameters} at $x$ if their images in $\calO_{C,x}$ do. By a full family of parameters $(t_1\..t_n,u)$ at $x$ we mean a regular family of parameters $t_1\..t_n\in\calO_{X,x}$ and a monomial chart $u\:P=\ocalM_\ox\to\cO_x$; the image of $p\in P$ is denoted $u^p$. The latter exists if and only if the log structure is Zariski at $x$, so in general a full family of parameters exists only \'etale locally. Any element of $u^P$ is viewed as a monomial parameter regardless of its arithmetic properties in $P$.

\begin{exer}
Let $X$ be a log manifold with a closed point $x$ and $l=k(x)$. Also, let $u\:P=\ocalM_\ox\to\calO_x$ be a monomial chart.

(i) Show that $t_1\..t_n\in\calO_{X,x}$ is a regular family of parameters if and only if $\hatcalO_{X,x}=l\llbracket P\rrbracket\llbracket t_1\..t_n\rrbracket$ if and only if the morphism $X\to\Spec(k[P][t_1\..t_n])$ is \'etale at $x$.

(ii) Show that $V\into X$ is a submanifold of codimension $r$ at $x$ if and only if there exist a regular family of parameters $t_1\..t_n\in\calO_{X,x}$ such that $V=V(t_1\..t_r)$ at $x$.
\end{exer}

\subsubsection{Admissible centers}
For an ideal $\cI$ let $\cI^{(d)}$ denote the integral closure $(\cI^d)^\nor$ of the $d$-th power. An {\em admissible center} or simply a {\em center} is an ideal which is locally of the form $\cJ=(t_1\..t_r,u^{p_1}\..u^{p_s})^{(d)}$, where $t_1\..t_r$ is a partial family of regular parameters, $u^{p_1}\..u^{p_s}$ are monomials and $d\ge 1$. Thus, as in the classical algorithm admissible centers are the ideals defined by parameters, but we have much more flexibility with the monomial generators.

\begin{rem}
(i) Admissible centers with $d=1$ are also called {\em submonomial ideals} because they correspond to monomial ideals on submanifolds. They can be locally described as $\calI+\calN$, where $\calI=(t_1\..t_r)$ is the ideal of a submanifold and $\calN=(u^{p_1}\..u^{p_r})$ is a monomial ideal.

(ii) The monomial ideal $\calN$ in the above presentation is unique, while the choice of $\calI$ and the corresponding submanifold $V=V(\calI)$ is not. For example, $\cJ=(t,u^p)$ can be also presented as $(t+fu^p,u^p)$ for any element $f$.
\end{rem}

\subsubsection{The multiplicity}
Amy admissible center which possesses a presentation as above with the power $d$ is called a {\em $d$-center}. If $r\ge 1$, then the center is not monomial, $d$ is uniquely determined and we define the {\em multiplicity} of $\cJ$ to be equal to $d$. If $r=0$, then the center is monomial, it is necessarily a 1-center, but can also be a $d$-center for finitely many other powers, and we define its multiplicity to be infinite. Multiplicity provides a reasonable classification of centers, and it will be the primary invariant of the method.

\begin{rem}
We will later see that log principalization can blindly insist to treat a monomial center as a $d$-center with no relation between the center and $d$. This request fails on log manifolds, and to remedy the problem we will have to switch to stacks.
\end{rem}

\subsubsection{Admissible blowings up}\label{admiseqsec}
Let $X$ be a log manifold and let $D$ be the toroidal divisor, i.e. $X\setminus D$ is the triviality locus of the log structure. Recall that the log structure of $X$ is divisorial, that is, $\calM_X=\cM(\rmlog D)$. Let $\cI\subseteq\cO_X$ be an ideal. A center $\cJ$ is {\em $\cI$-admissible} if $\cI\subseteq\cJ$, and an {\em $\cI$-admissible blowing up} is the normalized blowing up $f\:X'=Bl^\nor_\cJ(X)\to X$ along an $\cI$-admissible center $\cJ$ with the log structure $\cM_{X'}=\cM(\rmlog D')$ induced by the divisor $D'=f^{-1}(D\cup V(\cJ))$. The transform of the ideal is defined  as usual by $\cI'=(\cI\cO_{X'})(\cJ\cO_{X'})^{-1}$. Note that, similarly to the boundary in the classical desingularization framework, the new log structure is induced by the old one and the exceptional divisor.

\begin{rem}\label{admrem}
(i) The normalized blowings up along $\cJ$ and $\cJ^{(d)}$ coincide by Corollary~\ref{intcor}. What differs are the notions of admissibility and transform. Furthermore, the transforms with respect to the blowings up along $\cJ^d$ and $\cJ^{(d)}$ coincide by Lemma~\ref{intlem}, but we obtain the right version of admissibility only working with $\cJ^{(d)}$. For example, principalization of $\cJ^{(d)}$ can be done by blowing up either $\cJ^d$ or $\cJ^{(d)}$, but unless $\cJ^d$ is integrally closed, $\cJ^{(d)}$ is not $\cJ^d$-admissible.

(ii) A more classical but equivalent approach is to work with marked ideals $(\cI,d)$ instead of ideals and only 1-centers $\cJ$. In such a case, $\cJ$ is defined to be $(\cI,d)$-admissible if $\cI\subseteq\cJ^{(d)}$ and the transform is defined by $\cI\cO_{X'}(\cJ\cO_{X'})^{-d}$. This is the language used in \cite{ATW-principalization} and \cite{ATW-relative}, but once the general weighted centers were introduced in \cite{ATW-weighted}, our definition seems more natural.
\end{rem}

Our choice of admissible centers is natural, but one has to do some computation to check its correctness:

\begin{lem}\label{admblowlem}
Let $X$ be a log manifold and $X'\to X$ an admissible blowing up. Then $X'$ is a log manifold too.
\end{lem}

The proof is straightforward -- a model case is dealt with by an essentially toric computation and the general case follows by an appropriate choice of a covering. With enough care this method applies to the case when $X$ is an arbitrary log regular log scheme, see \cite[Lemma~5.2.3]{ATW-destackification}, but we stick to the case of varieties for simplicity.

\begin{exer}\label{admblowex}
Complete details in the following sketch of a proof of Lemma~\ref{admblowlem}

 (i) The question reduces to the model case when $X=\bfA^n_P=\Spec(k[P][t_1\..t_n])$, the log structure is given by a toric monoid $P$ and the center is of the form $\cJ=(t_1\..t_m,u^{p_1}\..u^{p_r})$. Indeed, the blowing up along $\cJ^{(d)}$ is the same as along $\cJ$, so we can assume that $\cJ$ is a 1-center, and then \'etale locally we can just choose parameters $t_1\..t_n,u\:P\to\calO_X$ such that $\cJ$ is generated by vanishing of a few parameters. Consider the induced \'etale morphism $\phi$ to the model log scheme $\bfA^n_P$. Then \'etale descent and compatibility of blowings up and normalizations with the \'etale morphism $\phi$ imply that it suffices to study the blowing up of the source along the ideal generated by the same parameters.

(ii) In the model case the blowing up is described as follows:

(a) If $s=t_i$, then the $s$-chart is $$X_s=\Spec(k[P_s][t'_1\..t'_{i-1},t'_{i+1}\..t'_n]),$$ where $t'_j=t_j/t_i$ for $j\le m$, $t'_j=t_j$ for $m<j\le n$ and $P_s$ is the saturation of the submonoid of $P^\gp\oplus \bbZ q$ generated by $P$, $q$ and the elements $p_i-q$. The log structure is extended by $u^q=t_i$.

(b) If $s=u^{p_j}$, then the $s$-chart is $X_s=\Spec(k[P_s][t'_1\..t'_n])$, where $t'_j=t_j/s$ for $j\le m$, $t'_j=t_j$ for $m<j\le n$ and $P_s$ is the saturation of the submonoid of $P^\gp$ generated over $P$ by the elements $p_1-p_j\..p_r-p_j$.
\end{exer}

The lemma can also be proved by another standard approach, which is beneficial in some situations. We will not need this, so just outline the idea in a side remark.

\begin{rem}
We will describe $Bl_\cJ(X)^\nor$ in terms of a log blowing up. First, working locally one increases the log structure by the direct summand $\oplus_{i=1}^m\bbN\log(t_i)$. This produces a new log scheme $Y$, which is still a log manifold. Then the center $\cJ$ becomes monomial and one considers the saturated log blowing up $Y'=LogBl_\cJ(Y)^\sat$. The morphism $Y'\to Y$ is log smooth, hence $Y'$ is also a log manifold. In particular, it is normal and hence $Y'$ is also the normalized blowing up of $Y$. In particular, $\uX'=\uY'$ and the divisor defining the log structure of $X'$ is obtained from that of $Y'$ by removing the strict transform of divisors corresponding to $V(t_i)$ -- on the $s$-chart one has that $D_{X'}=D_{Y'}\setminus\cup_{i=1}^mV(t_i/s)$. A simple computation shows that the obtained log scheme $X'$ is indeed log smooth because locally the monoids split as $\ocM_{Y',y}=\ocM_{X',y}\oplus\bbN^r$, where the free summand is generated by the elements $\log(t_i/s)$ such that $y\in V(t_i/s)$.
\end{rem}

\subsubsection{Functoriality of admissible blowings up}
Let $X$ be a log manifold. A log smooth morphism $f\:X'\to X$ induces a smooth morphism between the log strata, and hence a family of regular parameters at $x\in X$ pulls back to a partial family of regular parameters at any point $x'\in f^{-1}(x)$.  This implies that the pullback of an admissible center is an admissible center. Log smooth morphisms do not have to be flat, so their compatibility with blowings up is not automatic. However, using charts of log smooth morphisms and the above explicit description of admissible blowings up one easily obtains the following

\begin{lem}\label{functblowup}
If $X'\to X$ is a log smooth morphism of log manifolds and $\cJ$ is an admissible center on $X$, then $\cJ'=\cJ\cO_{X'}$ is an admissible on $X'$ and $Bl_{\cJ'}(X')=Bl_\cJ(X)\times_XX'$ (with the fiber product taken in the saturated category).
\end{lem}

\subsubsection{$\cI$-admissible sequences}
Given an ideal $\cI\subseteq\cO_X$, by an {\em $\cI$-admissible sequence of blowings up} or simply {\em an $\cI$-admissible sequence} we mean a sequence of admissible blowings up $f_\bullet\:X_n\dashto X_0=X$ with centers $\cJ_i\subseteq\cO_{X_i}$ such that each $\cJ_i$ is $\cI_i$-admissible, where $\cI_i$ is the transform of $\cI_{i-1}$ and $\cI_0=\cI$.

We will also need the following special case. Let $d\ge 1$ be a natural number. If each $\cJ_i$ is a $d$-center, then $f_\bullet$ is called an {\em $\cI$-admissible $d$-sequence}. By an {\em $\calI$-admissible $\infty$-sequence} we just mean any $\cI$-admissible sequence whose centers are monomial without any further restrictions. Lemma~\ref{functblowup} easily implies functoriality of $\cI$-admissible sequences:

\begin{lem}\label{functseq}
If $X'\to X$ is a log smooth morphism of log manifolds, $\cI$ is an ideal on $X$ with $\cI'=\cI\cO_{X'}$, and $f_\bullet\:X_n\dashto X_0=X$ is an $\cI$-admissible sequence (resp. $d$-sequence) with centers $\cJ_i\subseteq\cO_{X_i}$, then the pullback $f_\bullet\times_XX'\:X'_n\dashto X'_0=X'$ is an $\cI'$-admissible sequence (resp. $d$-sequence) with centers $\cJ_i\cO_{X'_i}$.
\end{lem}

\subsubsection{Log principalization}
A {\em principalization} of $\cI$ is an $\cI$-admissible sequence with $\cI_n=\cO_{X_n}$. The main theorem of logarithmic desingularization theory is that there exists log principalization of varieties, which is functorial with respect to all log smooth morphisms and satisfies the re-embedding principle. We will formulate it later, when the correct stack-theoretic framework will be established.

\subsubsection{Uniqueness of an ambient log manifold}
As in the classical case, log varieties can be embedded into log manifolds and the minimal embedding is unique up to \'etale covers and is essentially controlled by the tangent space. This is formulated and proved in detail in \cite[\S7.1]{ATW-principalization}, and here is a short summary and an outline of the arguments.

\begin{lem}\label{embedlem}
Let $Z$ be an equidimensional log variety, whose log structure is Zariski at a point $z\in Z$. Let $P=\ocM_z$, $r=\rk(P)$ and $n=\dim(m_{C_z,z}/m^2_{C_z,z})$ the dimension of the cotangent space of $C_z$ at $z$, where $C_z=V(P^+)$ is the log stratum at $z$.

(i) There exists a strict closed immersion of a neighborhood $U$ of $z$ into a log manifold $X$.

(ii) Locally at $z$ any strict closed immersion $Z\into X$ into a manifold factors through a closed submanifold of dimension $n+r$. In particular, $n+r$ is the minimal embedding dimension locally at $z$.

(iii) Any two embeddings $Z\into X_1$, $Z\into X_2$ into manifolds of dimension $n+r$ locally at $z$ are dominated by an embedding into a log manifold $X$ \'etale over $X_1$ and $X_2$.

(iv) Given an embedding $Z\into X$ any log smooth $Z$-scheme $Z'$ \'etale locally can be embedded into a log smooth $X$-scheme $X'$ so that $Z'=X'\times_XZ$.
\end{lem}

\begin{exer}
Prove the above lemma along the following lines:

(i) If $u\:P\to \cM_z$ is a monoidal chart and $t_1\..t_n\in\cO_{Z,z}$ are elements whose images span the cotangent space to $C_z$ at $z$. Then the induced morphism $Z\to X_0=\Spec(k[P][t_1\..t_n])$ is unramified at $z$ and hence locally at $z$ factors through a closed immersion into an \'etale $X_0$-scheme $X$.

(ii) The relative cotangent space to $Z\to\Spec(k[P])$ at $z$ is of dimension $n$, hence the minimal embedding dimension is $n+r$. If the dimension of $X$ is larger than $n+r$, then the map of the cotangent spaces at $z$ has a non-trivial kernel and therefore a non-zero element of this kernel lifts to an element $t\in\cI_{Z,z}=\Ker(\cO_{X,z}\to\cO_{Z,z})$. This $t$ is a parameter, hence $V(t)$ is a submanifold containing $Z$.

(iii) Apply (ii) to the embedding $Z\into X_1\times X_2$ to obtain a minimal embedding $Z\into X$ which is \'etale over $X_1$ and $X_2$.

(iv) \'Etale locally we can factor $Z'\to Z$ into a composition of an \'etale morphism $Z'\to Z''=Z_P[Q]\times\bfA^m$ and the projection onto $Z$. The projection simply lifts to $X''=X_P[Q]\times\bfA^m\to X$ and then $Z'\to Z''$ is the pullback of an appropriate \'etale morphism $X'\to X''$.
\end{exer}

\subsubsection{Reduction to log principalization}\label{reductiontologprinc}
Lemma~\ref{embedlem} reduces desingularization of log varieties to principalization on log manifolds. This is similar to the classical argument and is worked out in detail in \cite[\S7.2]{ATW-principalization}, and we outline the argument here.

Assume that $Z$ is a locally eqidimensional log variety which is generically log smooth (that is, $Z$ is generically reduced and the log structure is generically trivial). Since \'etale locally any fs log structure is Zariski, by \ref{embedlem}(i) there exists an \'etale cover $Z'\to Z$ such that $Z'$ admits a strict closed immersion into a log manifold $X$. Multiplying some components of $X$ by $\bfA^m_k$ we can assume that $Z'$ is of a constant codimension in $X$. Then the principalization $X_n\dashto X_0=X$ of $\cI_{Z'}$ blows up all generic points of the strict transforms of $Z'$ simultaneously, say at the blowing up $X_{m+1}\to X_m$. At this stage the strict transform of $Z'$ is a component $Z'_m$ of the center $V(\cJ_m)$, and it is even a log manifold because the monomial part of the center is trivial at its generic points. Thus, $Z'_m\to Z'$ is a projective birational morphism of log varieties with a log smooth source.

This resolution is independent of the embedding and its codimension by \ref{embedlem}(ii) and (iii) and the re-imbedding principle satisfied by log principalization. Moreover, such a resolution is compatible with log smooth morphisms and hence the resolution of $Z''=Z'\times_ZZ'$ is compatible with both projections $Z''\to Z'$, and by \'etale descent the resolution $Z'_\res\to Z'$ descends to a projective birational morphism $Z_\res\to Z$. Since the log manifold $Z'_\res$ is an \'etale over of $Z_\res$, the latter is also a log manifold, and hence $Z_\res\to Z$ is a resolution. Finally, compatibility of the obtained resolution with arbitrary log smooth morphisms is deduced from log smooth functoriality of the log principalization by use of \ref{embedlem}(iv).

\subsection{Log derivations and log order}\label{logdersec}

\subsubsection{The sheaf of log derivations}
The last choice of the framework is as follows: we will use log $k$-derivations on log manifolds $X$ (see \cite[\S4.2.3]{lognotes}), so we will use the notation $\cD_X=\Der_{k}(\calO_X,\calO_X)$. Note that the sheaf of log $k$-derivations is locally free of rank $d=\dim(X)$, and it can be described explicitly in terms of parameters. Naturally, this description splits into a classical part -- derivations corresponding to regular parameters, and a logarithmic part -- log derivations corresponding to monomials.

\begin{exer}
Assume that $t_1\..t_n,u:P\to\calO_{X,x}$ is a family of parameters at $x$.

(i) Show that for $1\le i\le n$ there exists a unique log derivation $(\partial_i,\delta_i=0)\in\cD_{X,x}$ such that $\partial_i(t_i)=1$ and $\partial_i(t_j)=0$ for $j\neq i$. We will denote this derivation $\partial_{t_i}$ or just $\partial_i$. Show that for any additive homomorphism $\phi\:P^\gp\to\calO_{X,x}$ there exists a unique log derivation $(\partial_\phi,\phi)$ such that $\partial_\phi(t_i)=0$ for any $i$. In particular, $\partial(u^p)=\phi(p)u^p$ for monomials $p\in P$. This log derivation will be denoted $\partial_\phi$.

(ii) Furthermore, show that $\cD_{X,x}$ is a free module, and if $\phi_1\..\phi_r$ form a basis of $\Hom(P^\gp,k)$ (or even of $\Hom(P^\gp,\bbZ)$), then $\partial_1\..\partial_n,\partial_{\phi_1}\..\partial_{\phi_r}$ is a basis of $\cD_{X,x}$.
\end{exer}

\subsubsection{Logarithmic differential operators}
The algebra of log differential operators on $\cO_X$ is the $\cO_X$-algebra $\cD^\infty_X$ generated by the $\cO_X$-module $\cD_X$. It is naturally filtered by submodules $\cD^{\le i}_X$ of operators of order at most $i$. Locally its elements are polynomials in $\partial_1\..\partial_n,\partial_{\phi_1}\..\partial_{\phi_r}$ over $\cO_X$ of degree at most $i$. For an ideal $\cI$ we will often consider its $i$-th derivation ideal $\cD_X^{\le i}(\cI)$. Clearly, $\cD_X^{\le i}(\cD_X^{\le j}(\cI))=\cD_X^{\le i+j}(\cI)$. The compatibility of these notions with log smooth morphisms is as follows:

\begin{lem}\label{derfunct}
Let $f\:X'\to X$ be a log smooth morphism of log manifolds and $\cD_{X'/X}=\Der_{\cO_X}(\cO_{X'},\cO_{X'})$. Then,

(i) The sequence $$0\to\cD_{X'/X}\to\cD_{X'}\to f^*(\cD_X)\to 0$$ is an exact sequence of free $\calO_{X'}$-modules, and hence it splits locally.

(ii) If $\cI$ is an ideal on $X$ and $\cI'=\cI\cO_{X'}$, then $\cD_{X'}^{\le i}(\cI')=\cD_X^{\le i}(\cI)\cO_{X'}$ for any $i\le\infty$.
\end{lem}

The first claim follows from an explicit description of log derivations of log smooth morphisms (the description we gave for log manifolds extends to morphisms). The second claim follows because $\cD_{X'/X}(\cI)=0$ and hence $\cD_{X'/X}(\cI')\subseteq\cI'$ by Leibnitz rule.

\subsubsection{Log order}
Using the sheaf $\cD_X$ one can define a logarithmic analogue of the usual order. Namely, given an ideal $\cI\subseteq\cO_X$ and a point $x\in X$ the {\em log order} $\logord_x(\cI)$ of $\cI$ at $x$ is the minimal number $d$ such that $\cD_X^{\le d}(\cI)_x=\cO_{X,x}$. If no such $d$ exists, then we set $\logord_x(\cI)=\infty$. This happens if and only if the {\em differential saturation} $\cD_X^\infty(\cI)$ is non-trivial at $x$, that is, $\calD_X^\infty(\cI)_x\subsetneq\cO_x$.

\begin{exex}
(i) Each monomial $u$ has infinite logorder because $u|\partial(u)$ for any log derivation $\partial$.

(ii) Formally locally $f\in\cO_{X,x}$ can be presented as $f=\sum_{p\in P,i\in\bbN^n} c_{i,p}u^pt^i$ with $c_{i,p}\in k(y)$ and $\logord_x(f)$ is the minimal $d=i_1+\dots+i_n$ such that $c_{i,0}\neq 0$.

(iii) An element is a regular parameter (resp. a unit) at $x$ if and only if its log order at $x$ equals 1 (resp. 0).

(iv) $\logord_x(\cI)=\logord_x(\cI^\nor)$ and $d\cdot\logord_x(\cI)=\logord_x(\cI^d)$.
\end{exex}

\begin{rem}
(i) The above exercise implies that the logorder at $x$ is computed using only the classical derivations $\partial_1\..\partial_n$ (or it can be computed using the submodule $\calF_x\subseteq\cD_{X,x}$ generated by $\partial_1\..\partial_n$), while the derivations $\delta_\phi$ are irrelevant. However, $\calF_x$ depends on the choice of parameters, and working with the whole $\cD_X$ avoids choices.

(ii) In addition, we want to stress that modules $\calF_x$ behave badly under transforms with respect to admissible blowings up -- sometime they are too small. For example, if $X=\Spec(k[x,y])$ with the log structure $\bbN\log(x)\oplus\bbN\log(y)$, then there are no regular parameters at the origin $O$ (the log stratum is 0-dimensional) and $\calF_x=0$. However, the preimage of $O$ in the monomial blowing up $X'=Bl_O(X)$ contains points $z$ with the log structure generated by $\bbN\log(x-y)$, see \cite[Example~4.1.5]{lognotes}. Such a point $z$ also has a regular parameter and $\cF_z$ is one dimensional (and depends on the choice of the parameters).
\end{rem}

\subsubsection{Maximal log order}
The {\em maximal log order} of an ideal $\cI$ is $$\logord_X(\cI)=\max_{x\in X}\logord_x(X).$$ The definition of (maximal) log order we gave above follows \cite{ATW-principalization}. However, it agrees with the multiplicity of centers and the induced primary invariant of ideals in spirit of \S\ref{basicsec}(3).

\begin{exer}
(i) The multiplicity of an admissible center $\cJ$ equals $\logord_x(\cJ)$ for any $x\in V(\cJ)$.

(ii) For any ideal $\cI$ on $X$ one has that $\logord_X(\cI)$ equals the maximal value of $d\in\bbN\cup\{\infty\}$ such that $\cI$ is contained in a center of multiplicity $d$.
\end{exer}

\subsubsection{Relation to the classical order}
In fact, the log order has the following very simple geometric interpretation: the log order of an ideal is just the order of its restriction onto the log stratum.

\begin{lem}\label{logordlem}
Let $X$ be a log manifold, $C\into X$ a log stratum, $x\in C$ a point and $\cI$ an ideal on $X$. Then $\logord_x(\cI)=\ord_x(\cI\cO_C)$.
\end{lem}

\begin{exer}\label{logordex}
Observe that $\cD_X$ preserves monomial ideals, hence a natural restriction map $\cD_X\to\cD_C=\Der(\cO_C,\cO_C)$ arises. Show that it is surjective and deduce that $\cD^{\le d}(\cI)\cO_C=\cD_C^{\le d}(\cI\cO_C)$. In particular, $\cD^{\le d}(\cI)_x=\cO_{X,x}$ if and only if $\cD_C^{\le d}(\cI\cO_C)_x=\cO_{C,x}$, yielding the lemma.
\end{exer}

Since log smooth morphisms induce smooth morphisms of log strata, the lemma and the usual smooth functoriality of the order imply log smooth functoriality of the log order:

\begin{cor}\label{logordcor}
If $f\:X'\to X$ is a log smooth morphism of log manifolds, $\cI$ is an ideal on $X$ with the induced map $\logord(\cI)\:X\to\bbN\cup\{\infty\}$, and $\cI'=\cI\cO_{X'}$, then $\logord(\cI')=\logord(\cI)\circ f$.
\end{cor}

\subsection{Log order reduction}\label{ordredsec}
Similarly to the classical algorithm, log principalization works inductively and the primary loop decreases the maximal log order. The main difference is that there exist non-zero ideals of infinite log order. We will see that this can be remedied by a single monomial blowing up. Once the log order is finite, a direct analogue of the classical theory works, including log analogues of maximal contact and coefficient ideals.

\subsubsection{Log order reduction}
For any $d\le l=\logord_X(\cI)$ by a {\em log order $d$-reduction} we mean an $\cI$-admissible $d$-sequence $X_n\dashto X_0=X$ such that $\logord_{X_n}(\cI_n)<d$.

\begin{rem}\label{orderredrem}
(i) In \cite{ATW-principalization} one uses an equivalent definition in the style of the classical order reduction, which encodes $d$ in the notion of marked ideals $(\cI,d)$ (only with a finite $d$) and uses only centers of multiplicity 1 or infinity, $d$-admissibility and $d$-transforms.

(ii) Note that $\logord_{X_n}(\cI_n)<d$ if and only if there exist no $\cI_n$-admissible $d$-centers. Thus, a log order $d$-reduction is just a maximal (with respect to truncations) $\cI$-admissible $d$-sequence.

(iii) The most natural and basic case is when $d=l$, as in the classical algorithm we call it the {\em maximal order case}. This is the case when the maximal contact theory applies, but it is necessary to deal with the general case because of inductive reasons.
\end{rem}

\subsubsection{Monomial hull}
The {\em monomial hull} of an ideal $\cI$ on $X$ is the minimal monomial ideal $\cM(\cI)$ containing $\cI$. By definition, it is the minimal monomial $\cI$-admissible center.

\begin{lem}\label{hullem}
For any ideal $\cI$ on a log manifold $X$ the differential saturation coincides with the monomial hull: $\cD^\infty_X(\cI)=\cM(\cI)$.
\end{lem}

Since any monomial ideal is preserved by $\cD_X$ and hence is differentially saturated, one has an inclusion $\cD^\infty_X(\cI)\subseteq\cM(\cI)$. The fact that it is an equality can be checked formally-locally. One proof is given in \cite[Theorem~3.4.2]{ATW-principalization}, we outline another one close in spirit to the argument in the proof of the first part of \cite[Proposition~3.4.3]{ATW-relative}. The following exercise is rather difficult, so our hint is in fact a sketch of the whole argument.

\begin{exer}\label{operatorexer}
Let $P$ be a sharp toric monoid, $l$ a field and $A=l\llbracket P\rrbracket\llbracket t_1\..t_n\rrbracket$. Let $\cD^\infty$ be the algebra of differential $l$-operators $A\to A$ generated by $\partial_i=\partial_{t_i}$ and $\partial_\phi$ with $\phi\in\Hom(P^\gp,l)$. Then for any $f=\sum_{d\in\bbN^n,p\in P} c_{d,p}t^du^p$ the ideal $(\cD^\infty(f))$ coincides with the monomial ideal generated by the elements $c_{d,p}u^p$. (Hint: again, the inclusion is clear, so it suffices to show that $(\cD^\infty(f))$ contains each $c_{d,p}u^p$. Let $\cD^\infty_t$ be the subalgebra of operators which are also $t$-logarithmic, that is, the subalgebra generated by the elements $t_i\partial_i$ and $\partial_\phi$. Any monomial $c_{d,p}t^du^p$ is an eigenvector of any monomial operator $$\partial=\prod_{i=1}^n(t_i\partial_i)^{a_i}\cdot\prod_{j=1}^r\partial_{\phi_j}^{b_j}.$$ Show that if $f$ is a finite sum, then there exists a monomial operator $\partial$ which has different eigenvalues with respect to all non-zero monomials $c_{d,p}t^du^p$ of $f$, and deduce that already the $l$-linear span of the elements $\partial^m(f)$ contains all monomials $c_{d,p}t^du^p$. Then remove the finiteness assumption by continuity: $f=f_{<n}+f_{\ge n}$ with $f_{\ge n}\in m_A^n$ and the above argument works up to a correction term lying in $m_A^n$ because $\partial(m_A)\subseteq m_A$. Finally, applying the derivations $\partial_i$ one clears off the powers of $t_i$ and obtains that each $c_{d,p}u^p$ lies in $(\cD^\infty(f))$.)
\end{exer}

From Lemmas \ref{hullem} and \ref{derfunct}(ii) we obtain that the monomial hull is functorial.

\begin{cor}\label{hullcor}
If $X'\to X$ is a log smooth morphism of log manifolds, $\cI$ is an ideal on $X$ and $\cI'=\cI\cO_{X'}$, then $\cM(\cI')=\cM(\cI)\cO_{X'}$.
\end{cor}

\subsubsection{Making log order finite}
Non-triviality of the monomial hull detects whether the maximal log order is infinite, so it is natural to expect that blowing it up makes $\logord_X(\cI)$ finite.

\begin{lem}\label{finorderlem}
Let $X$ be a log manifold and $\cI\subseteq\cO_X$ an ideal with the monomial hull $\cN=\cM(\cI)$. Let $X'=Bl_\cN(X)\to X$ be the $\cI$-admissible blowing up along $\cN$ and let $\cI'$ be the transform of $\cI$. Then $\logord_{X'}(\cI')<\infty$.
\end{lem}

Indeed, the monomial blowing up $X'\to X$ is log \'etale, hence by Corollary~\ref{hullcor} $\cN'=\cM(\cI\cO_X')$ coincides with the invertible ideal $\cN\cO_{X'}$ and $\cI'=(\cI\cO_{X'})\cN'^{-1}$. Therefore, $\cI\subseteq\cM(\cI')\cN'$ and by the minimality of $\cN'$ the monomial hull of $\cI'$ is trivial: $\cM(\cI')=\cO_{X'}$. The lemma follows.

There is a more concrete way to prove the lemma. It is very instructive for understanding the mechanism of cleaning off the monomial part, so we outline it below.

\begin{exer}
Let $x\in X$ be the image of $x'\in X'$, let $f_1\..f_m$ be generators of $\cI_x$ and let $f_i=\sum_{p,j}c_{i,j,p}t^ju^p$ be their expansions in $\hatcO_{X,x}=l\llbracket P\rrbracket\llbracket t_1\..t_n\rrbracket$. Then by Exercise~\ref{operatorexer}, $\hatcN_x$ is generated by all monomials $s=c_{i,j,p}u^p$ with a non-zero $c_{i,j,p}$. Show that the transform $s^{-1}f_i$ of $f_i$ on the $s$-chart of the blowing up of $\Spec(\hatcO_x)$ along $\hatcN_x$  is of order bounded by $|j|=j_1+\dots+j_n$, and hence this blowing up achieves that the order of the transform of $\cI_x$ is finite.
\end{exer}

\subsubsection{Coefficient ideals}\label{coeffsec}
Once $d=\logord_X(\cI)$ is finite we can define the (logarithmic) {\em coefficient ideal} of $\cI$ to be the weighted sum $\cC(\cI)=\sum_{a=0}^{d-1}(\cD_X^{\le a}(\cI),d-a)$, where we use the following homogenized definition of the sum of ideals $\cI_1\..\cI_n$ with a tuple of weights $l=(l_1\..l_n)$:
$$\sum_{i=1}^n(\cI_i,l_i)=\sum_{a\in\bbN^n|\ (a,l)\le l_1\dots l_n} \cI_1^{a_1}\cdots\cI_n^{a_n}.$$

\begin{rem}
Homogenized sums are introduced in \cite[\S5.1.7]{ATW-relative}. In the usual definition, which is also used in \cite[\S5.2(1)]{ATW-principalization}, one considers the smaller weighted sum ideal $\sum_{i=1}^n\cI_i^{d/d_i}$, where $d=d_1\dots d_n$. By the same reason as was mentioned in Remark~\ref{Rem:coefficient-ideals}, working with the two versions is equivalent, though technically it might be more convenient to use the homogenized version.
\end{rem}

\subsubsection{Maximal contacts}
Assume that $\cI$ is an ideal on a log manifold $X$ and $d=\logord_X(\cI)<\infty$. A submanifold $V\into X$ is called a {\em maximal contact} to $\cI$ if locally at any $x\in X$ with $\logord_x(\cI)=d$ it is of the form $V=V(t)$ with $t\in\cD^{\le d-1}_X(\cI)_x$. Clearly, maximal contacts exist locally at any point $x\in X$ with $\logord_x(\cI)=d$: since $\cD_X^{\le d}(\cI_x)=\cO_x$, there exists $t\in\cD^{\le d-1}_X(\cI)_x$ such that $\cD_X(t)$ contains a unit of $\cO_x$. This element is of log order 1, hence it is a regular parameter at $x$ and $V(t)$ is a log manifold in a neighborhood of $x$.

\subsubsection{The key results}
The theory of maximal contact is summarized in the following two theorems. Here we only formulate them and discuss. The arguments are very close to their analogues in the classical algorithm and will be outlined later. The first theorem establishes induction on dimension when constructing log order reduction:

\begin{theor}\label{equivth}
Assume that $X$ is a log manifold, $\cI$ is an ideal on $X$ and $H$ is a maximal contact to $\cI$. Then an admissible $d$-sequence of blowings up $X_n\dashto X_0=X$ with centers $\cJ_i^{(d)}$ is $\cI$-admissible if and only if for any $0\le i\le n-1$ we have that $I_{H_i}\subset\cJ_i$, where $H_i\into X_i$ is the strict transform of $H$, and the induced admissible $d!$-sequence $H_n\dashto H_0=H$ with centers $(\cJ_i|_{H_i})^{(d!)}$ is $\cC(\cI)|_H$-admissible.
\end{theor}

In view of Remark \ref{orderredrem}(ii), log order $d$-reductions of $\cI$ are in a natural one-to-one correspondence with log order $d!$-reductions of $\cC(\cI)|_H$. This reduces the log order reduction on $X$ in the maximal order case to a log order reduction problem on a hypersurface, but the latter does not have to be maximal because the log order of $\cC(\cI)|_H$ can be strictly larger than $d!$. In fact, it can even be infinite.

The second theorem ensures that the obtained log order reduction on $X$ is independent of the local choice of a maximal contact and hence is functorial and globalizes. It is an \'etale refinement of the claim that for any pair of maximal contacts, there is an isomorphism of formal completions of $X$ along them which respects the coefficient ideal.

\begin{theor}\label{homogth}
Assume that $X$ is a log manifold, $\cI$ is an ideal on $X$ with $d=\logord_X(\cI)<\infty$ and $H_1,H_2$ are two maximal contacts to $\cI$. Then there exists a pair of \'etale covers $p_i\:X'\to X$, $i=1,2$ such that $p_1^{-1}\cC(\cI)=p_2^{-1}\cC(\cI)$, $p_1^*H_1=p_2^*H_2$ and for any $\cI$-admissible $d$-sequence $X_n\dashto X_0=X$ its pullbacks with respect to $p_1$ and $p_2$ coincide.
\end{theor}

\begin{rem}
In the classical principalization there are two ways to prove independence of maximal contact. W{\l}odarczyk's homogenization method, which is transferred to the log setting in the above theorem, and Hironaka's trick with a delicate definition of equivalence of marked ideals, see \cite{Bierstone-Milman-funct}. Hironaka's trick can be extended to the log setting too, as was worked out in \cite[\S7.2]{ATW-relative}.
\end{rem}

\subsubsection{The non-maximal order case}
The above section shows how one deals with the maximal order case. It remains to construct a general log order $d$-reduction using the maximal order case. If $\logord(\cI)=l\ge d$ this is done by splitting ideals into the product of a {\em clean part} of a finite log order and an invertible monomial ideal and is based on the following observation:

\begin{lem}\label{balancedlem}
Let $X$ be a log manifold with an ideal $\cI\cN$, where $\cI$ is of maximal log order $l$ and $\cN$ is invertible monomial, and let $1\le d\le l$. Assume that $X'=X_n\dashto X_0=X$ is an $\cI$-admissible $l$-sequence with centers $\cJ_i^{(l)}$, and let $\cI'$ denote the transform of $\cI$. Then the same sequence of morphisms underlies the $\cI\cN$-admissible $d$-sequence with centers $\cJ_i^{(d)}$ and its transform is of the form $(\cI\cN)'=\cI'\cN'$, where $\cN'$ is invertible monomial.
\end{lem}

Indeed, the claim reduces by induction to a sequence of length one, and then $\cN'=\cN\cI_E^{l-d}$, where $E$ is the exceptional divisor.

In view of the lemma, one simply applies the sequences produced by the maximal order case of the log order reduction of the clean part with $l,l-1\..d$ and replaces in this sequence the $l$-centers, $l-1$-centers, etc. by the associated $d$-centers. This collects an invertible monomial factor, which can be removed in the end by a single blowing up.

\subsection{The lacuna}
At first glance we already have all necessary ingredients to combine them into a working algorithm, but there is left just one more case which we did not discuss... Namely, what happens for the log order $d$-reduction if $\logord_X(\cI)$ is infinite? Our method in the infinite order case is to blow up the monomial hull $\cN=\cM(\cI)$, but here the center must be a $d$-center. And sometimes $\cN$ is a $d$-center, but sometimes it is not... At first glance one might hope that the obstacle is not serious, and another method should work, so we are going to start with simplest examples which illustrate that one cannot go ahead without modifying the basic framework. Also, we will use this oportunity to make some comparison with the classical algorithm.

\begin{exam}\label{lacunaex}
We will principalize ideals on $X=\Spec(k[t,u])$ with the log structure $\bbN\log(u)$. By $O$ we denote the origin.

(i) Let $\cI=(t+u^2)$. In this case, $\logord_O(\cI)=1$, $H=V(t)$ is a maximal contact, $\cC(\cI)=(t,u^2)$ and the ideal $\cC(\cI)|_H=(u^2)$ is monomial and hence it is principalized by blowing it up. By Theorem~\ref{equivth} $\cI$ is principalized by the single blowing up along $(t,u^2)$. For comparison notice that the classical algorithm without log structures would use the same maximal contact and coefficient ideal, but would ``blow up $u$\ twice", so the centers would be $(t,u)$ and $(t',u')$.

(ii) Let $\cI=(t^2+u^2)$. In this case, $\logord_O(\cI)=2$, $H=V(t)$ is a maximal contact, $\cC(\cI)=(t^2,tu,u^2)$ and the ideal $\cC(\cI)|_H=(u^2)$ is monomial and principalized by blowing up the $2$-center $(u^2)$. Hence $\cI$ is principalized by the single blowing up along $(t,u)$. The classical algorithm acts in the same way in this case. The only difference is that it can also take $V(u)$ as the maximal contact, but this is forbidden in the logarithmic case.

(iii) Finally, let $\cI=(t^2+u)$. Then $\logord_O(\cI)=2$, $H=V(t)$ is a maximal contact, $\cC(\cI)=(t^2,u)$ and the ideal $\cC(\cI)|_H=(u)$ is monomial and principalized by blowing up the center $(u)$, which is not a 2-center. We are stuck. The classical algorithm acts completely differently in this case, as the usual order is 1 and the classical maximal contact is $V(u)$ but not $V(t)$.
\end{exam}

\subsubsection{Consequences of log smoothness}
In fact, the examples (ii) and (iii) are related by a log-\'etale cover. Namely, consider $X$ and $\cI$ as in (iii), the Kummer covering $Y=\Spec(k[t,v=u^{1/2}])$ of $X$ and the ideal $\cI'=\cI\cO_{X'}=(t^2+v^2)$. By the log smooth functoriality, the principalizations of $\cI$ and $\cI'$ must be compatible, and the principalization of $\cI'$ is done by blowing up $(t,v)$. This leaves us no choice -- the principalization $X'\to X$ of $\cI$ should be the blowing up of the center $(t,u^{1/2})$, whatever it means, such that $Y'=X'\times_XY$ as fs log schemes. Fortunately, functoriality considerations not only show that the method is in a conundrum in the basic framework, it also indicates how the framework should be extended:

\begin{rem}\label{solrem}
(i) First, it is clear that $X'$ should be the quotient of $Y'$ by the action of $\mu_2$, and since the schematic quotient is not log smooth we must take a stack-theoretic quotient -- this is the only way to obtain a (non-representable) modification $X'\to X$ such that $Y'=X'\times_XY$ in the fs category.

(ii) Second, this also suggests a way to formalize the notion of centers of the form $(t_1\..t_n,\cN^{1/d})$, where $\cN$ is monomial. Similarly to the blowings up, they are well defined after adjoining enough roots of monomials, so we just should work Kummer \'etale locally (see \cite[\S4.4]{lognotes}) and then any monomial center becomes a $d$-center for any natural number $d$. For ideals this is even a simpler formalism, which does not require to introduce stacks (but makes sense for stacks as well).
\end{rem}

In the next section we will extend the logarithmic framework along these lines by switching to log orbifolds and introducing Kummer ideals and blowings up, and then all constructions of this section will immediately extend and give rise to a logarithmic principalization algorithm.

\subsubsection{Monomial democracy}
To conclude the section we deduce from the log smooth functoriality a principle which explains most of the differences between the classical and the logarithmic algorithms. If $P\into Q$ is an embedding of toric monoids, then the homomorphism $Y=X_P[Q]\to X$ is log smooth, and the functoriality implies that principalization of $\calI$ on $X$ should pull back to principalization of $\calI\calO_Y$ on $Y$. In particular, the information about arithmetic properties of a monomial $p\in P$, such as the knowledge whether $p$ is a square in $P$, cannot be used by a log smooth functorial algorithm. Informally speaking, this means that a monomial $p\in P$ ``does not know'' in which monoid it lives, and this can be also phrased as the following ``monomial democracy'' principle in the style of Orwell: all monomials are equal. In particular, the principle explains why, unlike the classical algorithm, we allow any set of monomial generators in the definition of admissible centers and why the log order of monomials has to be infinite. Loosely speaking, the algorithm deals with the regular parameters by the classical methods, and it does not distinguish the monomial parameters -- just forms the monomial hull and blows it up at once.

Note for completeness that other orders are possible if one only constructs a smooth functorial method. For example, one orders monomial parameters by $1_+$ in a weighted resolution with boundary, see \cite[Remark~2.3.14]{weighted}.

\section{The logarithmic algorithm}\label{sec4}
Our goal now is to extend the framework by including stacks and Kummer ideals and blowings up, and then to finally describe the log principalization algorithm.

\subsection{The framework}
To extend the framework we have to switch to stacks and develop a bit the technique of Kummer \'etale localization.

\subsubsection{Geometric objects}\label{Sec:log-orbifold}
We will work with the 2-category of fs logarithmic DM stacks $X$ of finite type over $k$, and such an $X$ will be called a {\em log orbifold} if it is log smooth. Thus log manifolds are representable log orbifolds, and log orbifolds will play the role of manifolds in the realm of log DM stacks. Log smooth substacks will be called {\em suborbifolds} and their ideals \emph{suborbifold ideals}.

When working with stacks the default topology is \'etale, so $\cO_X$ and its ideals, the sheaves $\cD_X$, etc., are \'etale sheaves. This is based on the fact that in the case of schemes $\cO_X,\cD_X,\cD^\infty_X$, etc. are, in fact, \'etale sheaves (even flat sheaves) and hence their formations extend to DM stacks.

\begin{rem}
(i) Recall that the usual definition of log structures is \'etale local hence extends to DM stacks. Alternatively, one can define it via a descent: if a stack $\cX$ has a smooth presentation $X_0\to\cX$ and $X_1=X_0\times_\cX X_0$, then the log structure on $\cX$ is the same as a log structure $\cM_0$ on $X_0$ and an isomorphism of pullbacks $p_0^*\cM_0\toisom p_1^*\cM_0$ satisfying the usual cocycle condition. The latter definition also applies to arbitrary Artin stacks.

(ii) One can consider wider or narrower categories, the main restriction is that they should be closed under Kummer blowings up defined below. In particular, one can work with arbitrary Artin stacks or one can restrict the setting to DM stacks with finite diagonalizable inertia, as one does in \cite{ATW-principalization}. Our choice is the same as in \cite{ATW-relative}, it is the most general case when one can still work with the \'etale topology and all details of the proofs are literally the same.
\end{rem}

\subsubsection{Kummer ideals}
Kummer ideals are just ideals of the structure sheaf in the Kummer \'etale topology. For any log stack $X$ by $X_\ket$ we denote the category of log stacks over $X$ such that the morphism $X'\to X$ is Kummer \'etale. Declaring surjective morphisms to be covers we obtain a Grothendieck topology on $X_\ket$. The structure presheaf $\calO_{X_\ket}$ defined in the usual way by $\calO_{X_\ket}(U)=\Gamma(\calO_U)$ is in fact a sheaf by Kummer descent results of Niziol, see \cite[Proposition~2.18]{Niziol-K-theory-of-log-schemes-I}. A {\em Kummer ideal} is an ideal $\cI\subseteq\calO_{X_\ket}$. Such ideals are subject to all usual operations: sums, products, etc. Also, the sheaves of differential operators are Kummer so one can derive ideals and we will use the same notation $\cD^{\le d}_X(\cI)$.

Any ideal $\calJ\subseteq\calO_X$ generates the Kummer ideal $\calJ_\ket=\calJ\calO_{X_\ket}$ and by Kummer descent, the restriction of $\calJ_\ket$ onto $\calO_X$ is $\calJ$. In this case we will say that the Kummer ideal $\calJ_\ket$ is an ordinary ideal and by an abuse of notation will denote it also by $\calJ$. A Kummer ideal is called {\em monomial} if it is generated by monomials Kummer-locally.

\begin{rem}
In \cite{ATW-principalization} one also assumes that the ideal is finitely generated, but this will not make any difference. We will only consider usual ideals and admissible centers, which are combined from finitely generated ideals by arithmetic operations and saturation. The output of saturation  is usually not a finitely generated Kummer ideal, but it is still easily controllable.
\end{rem}

\begin{exex}\label{kumex}
(i) If $p\in\Gamma(\calM_X)$ is a global monomial and $d\ge 1$ is invertible on $X$ (which is automatically the case when the characteristic is zero), then the ideal $(p)$ is a $d$-th power of the uniquely defined invertible Kummer monomial ideal which we denote $(p^{1/d})$. Indeed, Kummer-locally this ideal is generated by a $d$-th root of $p$.

(ii) Any finitely generated Kummer monomial ideal is \'etale-locally of the form $(p_1^{1/d}\..p_r^{1/d}):=(p_1^{1/d})+\dots+(p_r^{1/d})$ for monomials $p_i$.

(iii) More generally, for a monomial ideal $\calJ$ let $\calJ^{1/d}$ be the ideal generated by $d$-th roots of the monomial generators of $\calJ$. It is a monomial Kummer ideal whose $d$-th power is integral over $\cJ$.

(iv) Give an example of a finitely generated monomial ideal $\cJ$ whose saturation as a Kummer ideal is not finitely generated. (Hint: one can take $X=\Spec(k[u,v])$ with the log structure $\bbN\log(u)+\bbN\log(v)$ and $\cJ=(u,v)$. Then its saturation in the Kummer topology contains all Kumer ideals of the form $((u^av^b)^{1/(a+b)})$ with $a,b\ge 1$ and hence is not finitely generated.)
\end{exex}

\subsubsection{Admissible centers}\label{centerssec}
The definitions are the same as for log schemes. A 1-center is a Kummer ideal of the form $\calI+\calN^\sat$, where $\calI$ is a suborbifold ideal (see \ref{Sec:log-orbifold}) and $\cN$ is a finitely generated monomial Kummer ideal. A $d$-center is of the form $\cJ=\cJ_1^{(d)}=(\cJ_1^d)^\nor$, where $\cJ_1$ is a 1-center. Any non-monomial center is a $d$-center for a unique $d$, while any monomial center is a $d$-center for any $d\ge 1$ by Exercise~\ref{kumex}(iii). In fact, one can even describe $\cJ$ in terms of saturation and arithmetic operations on ideals, without the need to use the more complicated integral closure operation, see \cite[Lemma~4.3.2(1)]{ATW-principalization}.

\begin{exer}
Show that $(\calI+\cN)^{(d)}=\sum_{i=0}^d\cI^i(\cN^{d-i})^\sat$. (Hint: the claim is Kummer local, hence it suffices to consider the case when $X$ is a log manifold and $\cN$ is a monomial ideal. Then increasing the log structure by $\oplus_{i=1}^n\bbN\log(t_i)$, where $\cI=(t_1\..t_n)$, the claim reduces to the standard lemma that a monomial ideal is saturated if and only if it is integrally closed.)
\end{exer}

\subsubsection{Kummer blowings up: two approaches}\label{twoapproaches}
There are a few possible definitions. A first construction was introduced in \cite[\S5.4]{ATW-destackification} and used in the first logarithmic principalization algorithm in \cite{ATW-principalization}. This construction followed the logic of Remark~\ref{solrem}(1) -- one considers an appropriate Galois Kummer covering on which the center becomes a usual ideal, blows it up, and then divides by the Galois group. The subtle point is that in order to do this independently of the covering the quotient should be a mix of a stack-theoretic and schematic quotients (a partial coarsening of the stack-theoretic quotient). A simpler and more general definition that uses stack-theoretic Proj construction was suggested by Rydh soon after \cite{ATW-principalization} was written, and this is the definition we will work with below. It was used in \cite{ATW-relative} and it can be generalized to arbitrary weighted blowings up easily. Since this definition was studied in chapter \cite{stacksnotes}, we just quickly recall the main points.

\subsubsection{The stacky Proj construction}
Let $X$ be a scheme and let $\cA=\oplus_{d=0}^\infty\cA_d$ be a finitely generated graded $\cO_X$-algebra. Then $Y=\Spec_X(\cA)\to X$ is a $\Gm$-equivariant morphism with the action on $X$ being trivial. The vanishing locus $V=V(\cA_{\ge 1})$ of the irrelevant ideal $\cA_{\ge 1}=\oplus_{d=1}^\infty\cA_d$ is the set of fixed points. In particular, the action on $Y\setminus V$ has stabilizers of the form $\mu_d$ and the usual scheme-theoretic quotient exists (there is no need to use GIT). The usual Proj construction associates to $\cA$ the scheme-theoretic quotient $\Proj_X(\cA)=(Y\setminus V)/\Gm$ and forgets the information about the stabilizers. In the realm of stacks it is natural to consider the finer {\em stacky Proj construction} which outputs the stack-theoretic quotient $\cProj_X(\cA)=[(Y\setminus V)/\Gm]$. By definition, $\Proj_X(\cA)$ is the coarse moduli space of $\cProj_X(\cA)$.

\subsubsection{Charts and stabilizers}
The stacky Proj construction is  local on the base, and it is glued from charts which refine the charts of the usual Proj stack-theoretically. Details are in chapter \cite{stacksnotes}, here we only recall them briefly.

\begin{exer}\label{kummerchartsex}
Assume that $X=\Spec(B)$ and $A=\oplus_{d=0}^\infty A_d$ is a finitely generated graded $B$-algebra, and let $\cZ=\cProj_B(A)=[Y/\Gm]$, where $Y=\Spec(A)\setminus V$.

(o) For any homogeneous $s\in A_d$ the $s$-chart $\cZ_s=[Y_s/\Gm]$ is an open substack of $\cZ$ and $\cZ=\cup_{i=1}^n\cZ_{s_i}$ for homogeneous elements $s_1\..s_n\in A$ if and only if the radical of $(s_1\..s_n)$ coincides with the irrelevant ideal $A_{\ge 1}$.

(i) Consider the $\bbZ/d\bbZ$-graded ring $A_s/(s-1)$, whose $i$-th homogeneous component is isomorphic to $(A_s)_{i+jd}$ with any $j\in\bbZ$. Then $\cZ_s=[\Spec(A_s/(s-1))/\mu_d]$, in particular, its inertia is bounded by $\mu_d$. In addition, we see that the usual and the stacky Proj constructions coincide in the classical case when $\cA$ is generated in degree 1.
\end{exer}

\begin{exam}
(i) If $\cJ$ is an ideal on $X$ and $\cA=\oplus_{d=0}^\infty\cJ^d$, then $X'=\cProj_X(\cA)=\Proj_X(\cA)=Bl_\cJ(X)$ is the usual blowing up of $X$ along $\cJ$. Blowing up $\cJ^n$ we obtain the same morphism $X'\to X$, but with the exceptional divisor multiplied by $n$.

(ii) Using the stacky Proj we can also blow up an analogue of $\cJ^{1/n}$. Namely, consider $\cA=\oplus_{d=0}^\infty\cA_d$ with $\cA_{d}=\cJ^{\lceil d/n\rceil}$ where the brackets denote the upper rounding to an integer. Then $\cX'=\cProj_X(\cA)$ is the root stack obtained from $X'$ by extracting $n$-root from the exceptional divisor $\cJ\cO_{X'}$. This root ideal should be viewed as the exceptional divisor of the blowing up along $\cJ^{1/n}$ and, indeed, $\cX'\to X$ is the universal morphism such that the pullback of $\cJ$ is the $n$-th power of an invertible ideal, though this condition should be stated carefully, see \cite[Theorem~3.2.11 and Corollary~3.2.12]{Quek-Rydh}.
\end{exam}

\subsubsection{Extension to stacks}
An important feature of the stacky Proj construction is that it extends to stacks. The usual Proj construction extends by standard \'etale descent, but dealing with the $\Gm$-quotient requires an additional argument, see chapter \cite{stacksnotes}.

\subsubsection{Kummer blowings up}\label{kumblowsec}
Given a log DM stack $X$, let $\pi\:X_\ket\to X$ denote the restriction of sites. Thus, given a Kummer ideal $\cJ$ on $X$ by $\pi_*\cJ$ we denote its restriction on $X$. For example, if $X=\Spec(A)$ is affine, then $\pi_*\cJ$ is just the ideal $\cJ(X)$ of $A$. Now, we define Kummer blowing up by $X'=\cProj(\cR_\cJ)^\nor$, where $\cR_\cJ=\oplus_{d=0}^\infty\pi_*(\cJ^d)$ is the Rees algebra of $\cJ$. Clearly, when $\cJ$ is a usual ideal, this reduces to the usual definition of a normalized blowing up. The Kummer ideal $\cI_E=\cJ\cO_{X'_\ket}$ is an invertible (usual) ideal because, as in the classical case, its restriction onto the $s$-chart $X'_s$, where $s\in\cJ^d(X)$, is the invertible idea $(s)$. In particular, the exceptional divisor $E$ is defined as usual, and if the blowing up is {\em $\cI$-admissible}, that is, $\cI\subseteq\cJ$, then the transform $\cI=(\cI\cO_X)\cI_E^{-1}$ is defined.

Finally, as in \S\ref{admiseqsec} we provide $X'$ with the log structure $\cM_{X'}=\cM(\log(D'))$ induced by the divisor $D'=f^{-1}(D)\cup E$, where $f\:X'\to X$ is the blowing up and $D$ defines the log structure on $X$ (the construction of the log structure $\cM(\log(D'))$ from a divisor is compatible with \'etale morphisms and hence extends to DM stacks). The log DM stack $(\uX',\cM_{X'})$ is called the {\em blowing up} of $X$ along the Kummer ideal $\cJ$ and denoted $Bl_\cJ(X)$.

\subsubsection{Admissible blowings up}
We will only use the above construction when $\cJ$ is an admissible center. The main properties of this operation are that the source is log smooth and the blowing up is compatible with log smooth morphisms:

\begin{theor}\label{Kummeblowth}
Let $X$ be a log orbifold, let $\cJ$ be an admissible Kummer center on $X$ and let $X'=Bl_\cJ(X)$.

(i) The blowing up $X'\to X$ is a proper birational morphism with finite diagonalizable relative inertia, which is an isomorphism over the complement of $V(\cJ)$.

(ii) $X'$ is a log orbifold.

(iii) Log smooth functoriality: if $Y\to X$ is a log smooth morphism, then $Bl_{\cJ\cO_{Y_\ket}}(Y)=X'\times_XY$, where the product is taken in the fs category.

(iv) \'Etale descent: if $X\to Y$ is a strict \'etale cover and the pullbacks of $\cJ$ to $X\times_YX$ coincide, then $\cJ$ is the pullback of an admissible Kummer center $\cI$ on $Y$ and $X'\to X$ is the saturated pullback of the admissible blowing up $Bl_\cI(Y)\to Y$.
\end{theor}

Part (i) of the theorem holds for any Kummer blowing up. Parts (ii) and (iii) are proved similarly to Exercise~\ref{admblowex} -- using \'etale descent one reduces the claim to the case of a model scheme and admissible center and then performs an explicit (essentially toric) computation. In view of part (iii), it suffices to descend the center in part (iv). The claims is \'etale local on $Y$ hence we can assume that it possesses a chart $Y\to\bfA_P$. Then $X_P[\frac{1}{d}P]=Y_P[\frac{1}{d}P]\times_YX$ and taking $d$ large enough we achieve that $\cJ$ is an ideal on $X_P[\frac{1}{d}P]$. By \'etale descent it is the pullback of an ideal on $Y_P[\frac{1}{d}P]$, which is easily seen to be an admissible center, and therefore it gives rise to an admissible Kummer center on $Y$.

Now, let us illustrate how the functoriality from (iii) works in a concrete case, how the stacky structure disappears for Kummer covers and how this is related to saturated pullbacks. In fact, all this is seen very well already on the log manifolds and ideals from Example~\ref{lacunaex}.

\begin{exam}
(i) $Y=\Spec(k[t,v])$ with the log structure generated by $v$ and $\cI=(t,v)$. The $v$-chart $Y'_v$ of the blowing up $Y'=Bl_\cI(Y)$ is computed as usual: $Y'_v=\Spec(k[v,t'])$ with $t'=t/v$ and the log structure generated by $v$.

(ii) $X=\Spec(k[t,u])$ with the log structure generated by $u$ and $\cJ=(t,u^{1/2})$. By the chart description in Exercise~\ref{kummerchartsex} the $u^{1/2}$-chart $X'_{u^{1/2}}$ of $X'=Bl_\cJ(X)$ is described as follows. The Rees algebra is the graded subalgebra $\cR_\cJ=k[t,u,tz,uz,uz^2]$ of the graded algebra $\oplus_{d=0}^\infty k[t,u]z^d$ and the $u^{1/2}$-chart corresponds to the $uz^2$-chart. So, $X'_{u^{1/2}}$ is the $\mu_2$-quotient of $W=\Spec(\cR_\cJ/(uz^2-1))$, where $$\cR_\cJ/(uz^2-1)=k[t,u,x=tz,y=uz]/(y^2-u,xy-t)=k[x,y].$$ We see that the new coordinates are $y=u^{1/2}$ and $x=t/y$, quite analogously to coordinates of admissible blowings up, and $\mu_2$ acts on both $x$ and $y$ by negation because they are of grading 1. In particular, the coarse quotient $W/\mu_2=\Spec(k[y^2,xy,y^2])$ with the log structure $\bbN\log(y^2)$ is not log smooth, while the stack theoretic refinement $X'_{u^{1/2}}=[W/\mu_2]$ is log smooth (and even smooth).

(iii) Define a morphism $Y\to X$ by $v^2=u$. Then $\cI=\cJ\cO_Y$ and indeed $Y'$ is easily seen to be the saturation of $X'\times_XY$. At first glance it looks paradoxical that $X'\to X$ is a non-representable morphism but its base change $Y'\to Y$ is representable, as this could not happen for \'etale covers, but here $Y\to X$ is log \'etale and the base change is saturated. In fact, in this case the non-saturated base change $X'\times_XY$ has a non-trivial stacky structure (the root stack associated to the square root of $(v^2)$), but its saturation $Y'$ is a scheme.
\end{exam}

\subsection{The algorithm}\label{42}

\subsubsection{Results of \S\S\ref{logdersec}--\ref{ordredsec}}
Theorem~\ref{Kummeblowth} accomplishes the construction of the logarithmic principalization framework. Once this is done, almost all results and proofs of \S\ref{firstsec} extend almost verbatim because local proofs are \'etale-local. In particular, absolutely in the same way one defines $\cI$-admissible sequences (of Kummer blowings up), principalization, log order and log order $d$-reduction, maximal contact and coefficient ideal and establishes analogues of Theorems~\ref{equivth} and \ref{homogth}, Lemmas~\ref{hullem}, \ref{finorderlem}, \ref{logordlem}, etc.

\subsubsection{Log order reduction}
The following definitions are the same as in the classical case.

\begin{defin}\label{ordereddef}
(i) A {\em log order $d$-reduction} of an ideal $\cI$ on a log orbifold $X$ is an $\cI$-admissible $d$-sequence $f_\bullet\:X_n\dashto X_0=X$ such that $\logord_{X_n}(\cI_n)<d$. A log order reduction $\cF$ on the category of log orbifolds over $k$ associates to any $X,\cI,d$ as above a log order reduction $f_\bullet=\cF(\cI,d)$.

(ii) A log order reduction $\cF$ on the category of log orbifolds over $k$ is {\em log smooth functorial} if for any $X,\cI,d$ as above and a log smooth morphism $Y\to X$ the log order $d$-reduction of $\cI\cO_Y$ is the contracted pullback of the log order $d$-reduction of $\cI$, that is, $\cF(\cI\cO_Y,d)$ is obtained from $\cF(\cI,d)\times_XY$ by omitting all trivial blowings up.

(iii) $\cF$ satisfies the {\em re-embedding principle} if for any closed immersion of log orbifolds $Y\into X$ of constant codimension and an ideal $\cI$ on $X$ such that $I_Y\subseteq\cI$ the sequence $\cF(\cI,1)$ is obtained by pushing forward $\cF(\cI\cO_Y,1)$ (that is, one blows up the same Kummer centers on the strict transforms of $Y$).
\end{defin}

Now we can formulate the log order reduction theorem. It is the main result of the theory, which implies the log principalization and log desingularization theorems.

\begin{theor}\label{logorderredth}
On the category of log orbifolds over $k$ there exists a log order reduction method $\cF$ which is log smooth functorial and satisfies the re-embedding principle.
\end{theor}

The {\em log principalization} is defined similarly, but it depends only on $X$ and $\cI$ (but not on $d$), one is allowed to use arbitrary $\cI$-admissible centers and obtains $\logord_{X_n}(\cI_n)=0$ in the end (that is, $\cI_n=\cO_{X_n}$). The log smooth functoriality and the re-mebedding principle are literally the same. The log principalization theorem is an immediate corollary of Theorem~\ref{logorderredth} obtained by taking $d=1$.

\begin{theor}\label{logprincth}
On the category of log orbifolds over $k$ there exists a logarithmic principalization method $\cP$, which is log smooth functorial and satisfies the re-embedding principle.
\end{theor}

Finally, the arguments from \S\ref{reductiontologprinc} imply the following logarithmic desingularization result.

\begin{theor}\label{legresth}
There is a logarithmic desingularization method $\cR$ which to any generically smooth and locally equidimensional log DM stack $X$ of finite type over $k$ associates a birational morphism $\cR(X)\:X_\res\to X$ with a log smooth $X_\res$, and for any log smooth morphism $Y\to X$ one has that $Y_\res=X_\res\times_XY$ in the fs category.
\end{theor}

\begin{rem}
(i) The morphism $X_\res\to X$ almost always is not representable even when $X$ is a variety, in particular, it is not projective. However, it belongs to the class of morphisms which is a natural non-representable extension of the class of projective morphisms -- the class of global quotients of projective morphisms. In particular, the relative coarse space $(X_\res)_{\cs/X}$ is projective over $X$. However, $(X_\res)_{\cs/X}$ does not have to be log smooth and can have quotient singularities.

(ii) One can canonically resolve $(X_\res)_{\cs/X}$, and even do this on the level of stacks by the so-called destackification procedure. Namely, there is a smooth functorial way to construct a further modification $X'\to X_\res$ such that the relative coarse space $X'_{\cs/X}$ is log smooth and the morphism $X'_{\cs/X}\to X$ is projective, see \cite{stacksnotes}. This step cannot be done log smooth functorially.
\end{rem}

\subsubsection{The log order reduction algorithm}\label{algsec}
Our proof of Theorem~\ref{logorderredth} consists of constructing a required algorithm by induction on dimension and checking that it satisfies all required properties. So, assume that the algorithm is already constructed when $\dim(X)\le n-1$ and let us construct it when $\dim(X)=n$. The input data consists of an ideal $\cI$ and a natural number $d\ge 1$. The output will be an $\cI$-admissible $d$-sequence $$X'\to X_n\dashto X_{n-1}\dashto\ldots\dashto X_1\to X,$$ where the first and the last arrows are single monomial blowings up, and the arrows $X_{i+1}\dashto X_i$ encode a whole $\cI_i$-admissible $d$-sequence of non-monomial blowings up and the following conditions hold: the final transform satisfies $\logord_{X'}(\cI')<d$, each intermediate transform $\cI_i\subseteq\cO_{X_i}$ splits as $\cI_i=\cI_i^\cln\cN_i$, where $\cN_i$ is invertible monomial and $d_i=\logord_{X_i}(\cI_i^\cln)$ is finite, and one has that $d_1>d_2>\dots>d_{n-1}\ge d>d_n$.

{\em The initial cleaning step.} This step makes the log order finite and consists of the single blowing up along the $d$-center $\cJ^{(d)}$, where $\cJ=\cM(\cI)^{1/d}$. Clearly, $\cJ^{(d)}$ is the integral closure of $\cM(\cI)$, hence the blowing up has the same effect and the transform $\cI_1$ is of finite log order by Lemma~\ref{finorderlem}. In particular, $\cN_1$ is trivial.

{\em The regular step.} This step accepts $\cI_i=\cI_i^{\cln}\cN_i$ with $d_i\ge d$ as an input and outputs a $d$-sequence $f_\bullet^d\:X_{i+1}=Y_m\dashto Y_0=X_i$, which is associated in the sense of Lemma~\ref{balancedlem} to the $d_i$-sequence $f_\bullet^{d_i}\:Y_m\dashto Y_0=X_i$ which reduces the log order of $\cI^\cln_i$. By Lemma~\ref{balancedlem} the sequence outputs $\cI_{i+1}=\cI_{i+1}^\cln\cN_{i+1}$ with $d_{i+1}<d_i$.

The $d_i$-sequence $f_\bullet^{d_i}$ is constructed as follows. Working \'etale locally we can assume that there exists a maximal contact $H$ to $\cI$ and by induction assumption the ideal $\cC(\cI)|_H$ possesses a log order $d_i!$-reduction $g_\bullet\:H_m\dashto H_0=H$ with centers $\ocJ_j^{(d!)}$. We define $f_\bullet^{d_i}$ to be the sequence with centers $\cJ_j^{(d)}$, where $\cJ_j$ is the preimage of $\ocJ_j$ under the surjection $\cO_{Y_j}\onto\cO_{H_j}$. Thus, if locally on $Y_j$ we have that $H_j=V(t)$ and the center of $H_{j+1}\to H_j$ is $(t_1\..t_r,u^{p_1}\..u^{p_s})^{(d_j!)}$, then the center of $Y_{j+1}\to Y_j$ is $\cJ_j=(t,t_1\..t_r,u^{p_1}\..u^{p_s})^{(d_j)}$ and the corresponding center of $f^d_\bullet$ is $(t,t_1\..t_r,u^{p_1}\..u^{p_s})^{(d)}$.

By Theorem~\ref{equivth} the sequence $f_\bullet^{d}$ is indeed a log order $d$-reduction of $\cI_i$, and by Theorem~\ref{homogth} it is independent of choices and descends from the \'etale local construction.

{\em The final cleaning step.} This is trivial -- we just repeat the cleaning step, but this time $\cM(\cI_n)=\cN_n$, so we just blow up the $d$-center $\cN_n=(\cN_n^{1/d})^{(d)}$. Thus, $X'=X_n$ and $\cI_{n+1}=\cI_n^\cln$ is as required.

\subsubsection{Justification of the algorithm}
It is clear from the construction that the obtained sequence $X'\to X$ is a log order $d$-reduction for $\cI$ -- it first makes the log order finite, then reduces the log order of the clean part below $d$, and then removes the invertible monomial part. Compatibility of this algorithm with log smooth morphisms $Y\to X$ follows from the fact that all basic ingredients of our construction are log smooth functorial, as was stated in Corollaries~\ref{hullcor} and \ref{logordcor}, Lemma~\ref{derfunct}(ii), etc. Finally, it suffices to check the re-embedding principle \'etale-locally, so we can assume that $X\into Y$ is a closed immersion of log orbifolds of pure codimension one and $\cI'$ is the preimage of $\cI$ in $\cO_Y$. Then $\logord_Y(\cI')=1$ and $X$ is a maximal contact to $\cI'$, so by the regular step of the algorithm $\cF(\cI',1)$ is the pushforward of $\cF(\cI,1)$, which is precisely what is claimed by the re-embedding principle.

\begin{rem}
We already noted that the log principalization is obtained by applying the log order reduction with $d=1$. Similarly to the proof of log order reduction, one can principalize just a bit faster by iteratively applying maximal order reductions. This might look more natural as there is no need to do the final cleaning step, but the drawback is that such an algorithm does not satisfy the re-embedding principle.
\end{rem}

\subsubsection{The invariant}
We constructed the log order reduction $\cF(\cI,d)$ as a composition of blowing up sequences $X_{i+1}\dashto X_i$, which can be parameterized by the strictly decreasing sequence $\logord(\cI_i^\cln)$. In particular, the first stage is numbered by $\infty$, the last stage -- by a number which is strictly smaller than $d$ and will be replaced by 0 for convenience, and all regular steps are numbered by a finite number $d_i\ge d$. The first and last steps are single monomial blowings up, and each regular step corresponds to the log order $d_i!$-reduction of $\cC(\cI_i^\cln)|_{H_i}$. We define the invariant of a separate blowing up of $\cF(\cI,d)$ by induction on $\dim(X)$ as follows: the blowing up of the first step (if non-trivial) has invariant $(\infty)$, the blowing up of the last step (if non-trivial) has invariant $(0)$, the blowings up of the $i$-th regular sequence have invariant $(d_i,\inv_{H_i})$, where $\inv_{H_i}$ denotes the invariant of the corresponding log order $d_i!$-reduction sequence of $\cC(\cI_i^\cln)|_{H_i}$.

Each non-normalized invariant is of the form $(d_0,d_1\..d_n,*)$, where $d_0\ge d$, $d_{i+1}\ge d_i!$ and $*\in\{\infty, 0\}$, and by induction on the length of the invariant and the fact that $\logord_{X_i}(\cI^\cln_i)$ strictly decreases, we obtain that each separate blowing up in the sequence indeed decreases the value of the invariant. Hence the same is also true for the normalized invariant $(q_0,q_1\..q_n,*)$, where $q_i=d_i/\prod_{j<i}(d_j-1)!$.

\begin{rem}\label{invrem}
(i) The algorithm consists of nested loops: the outer loop reduces the value of the logorder $d_0$, the next loop reduces the order $d_1$ of the coefficient ideal (and its transforms) on the first maximal contact $H_1$ (and its transforms), etc. The invariant just parameterizes our location in this sequence of loops.

(ii) The non-normalized invariant of the first blowing up is always of the form $(d_0\..d_n,d_{n+1}=\infty)$, where $d_i$ is the log order of the $i$-iterated coefficient ideal restricted to the $i$-th iterated maximal contact.
\end{rem}

\subsubsection{Comparison with the classical algorithm}
The classical principalization algorithm is more complicated and slow, but uses simpler basic operations. Here are main differences between the algorithms:

0) The combinatorial step of the classical algorithm is rather heavy and uses the order of exceptional components. In our case its analogue is the single blowing up at the final cleaning step.

1) The classical algorithm has to separate the boundary each time it passes to a maximal contact. This happens because the maximal contact may have a bad intersection with the boundary. This separation is done by a separate loop on the number of boundary components, so the actual normalized invariant of the regular steps of the algorithm is of the form $(q_0,s_0;q_1,s_1;\dots)$. Instead of this the logarithmic algorithm chooses maximal contact adopted to the log structure.

2) The classical algorithm uses both regular and exceptional parameters to compute the order. Because of this the clean part can be already resolved, while the order of the whole $\cI$ is still larger than $d$. This forces one to work with the companion ideal rather than just the clean part. In the logarithmic algorithm, one can safely ignore the monomial factor.

$\infty$) The classical algorithm never starts with the initial cleaning step and its invariant is never $(\infty)$ (unless $\cI=0$), but an analogue of this step is applied when the restriction of $\cC(\cI)$ onto the maximal contact $H$ vanishes and hence its order is infinite. In the classical algorithm, this situation is also encoded by the infinite degree finishing the invariant string. This situation occurs if and only if $\cI=(I_H)^d$ and then the single blowing up of $H$ principalizes $\cI$.

\subsection{Other logarithmic frameworks}
In brief, the logarithmic resolution theory works in any context of logarithmic spaces with large enough sheaves of derivations -- sufficient to distinguish regular and logarithmic parameters, and the obtained methods are functorial with respect to log regular morphisms. This includes various categories of analytic spaces with log structures, and schemes or formal schemes with enough log derivations.

\section{Resolution of morphisms}\label{morsec}
The goal of this section is to outline the results of \cite{ATW-relative}, in which logarithmic desingularization methods were generalized to the relative setting and a functorial resolution of morphisms between log schemes was obtained. As earlier, an important part of this work is to set up the framework. In two words, one studies principalization on relative log orbifolds $f\:X\to B$, with a log regular $B$, and one uses the sheaf of relative log derivations $\cD_{X/B}=\Der_{\calO_B}(\calO_X,\calO_X)$ to define the relative log order, maximal contacts and coefficient ideals. Using this language one generalizes the principalization algorithm straightforwardly. So, our exposition will be very fast and we will often just mention which argument from \S3--4 should be adjusted in the relative case.

The only really new feature of the relative situation is that base changes should be incorporated into the algorithm. On the one hand, the algorithm is functorial with respect to any base changes, so it is indeed an honest relative algorithm. On the other hand, the algorithm can fail over a given $B$, and it succeeds only after a large enough base change $B'\to B$. The proof of this is based on a new tool -- the monomialization theorem, see \cite[Theorem~3.6.13]{ATW-relative}. As earlier, for simplicity we will consider the case when $f$ is of finite type, and only in the end make a few comments about more general settings.

\subsection{Framework}
Constructing the relative framework occupies Sections 2 and 4 in \cite{ATW-relative} and is rather heavy. To a large extent this is caused by considering log regular morphisms not necessarily of finite type. In such a case one has to require of the sheaf $\cD_{X/B}$ to be suitably ``large". In our case of DM stacks of finite type over a field, the morphisms are automatically log smooth and the sheaf $\cD_{X/B}$ is locally free of the expected rank, so the situation simplifies. Still, we prefer to briefly mention all needed ingredients and just provide references to \cite{ATW-relative}.

\subsubsection{Geometric objects}
Naturally, a basic object this time is a morphism $f\:X\to B$ of fs log DM stacks of finite type over $k$. In principle, we aim to construct a purely relative algorithm, which should apply to any target, but for technical reasons, we have to restrict to the case when $B$ is log smooth, see \cite[Remarks~1.2.9 and 3.1.11]{ATW-relative}. Since in any case the algorithms succeed only after a large enough base change, essentially this just restricts us to the case of a generically log smooth target, see also \cite[Remark~1.2.9]{ATW-relative}. For simplicity of exposition, we assume that $B$ is a log manifold (i.e. it is also a scheme), though using \'etale descent on the base one can easily extend everything to the case when $B$ is a stack.

There are two classes of morphisms between morphisms we will consider: 1) just $B$-morphisms (the same target), 2) base change morphisms (or diagrams) $f'\to f$, where $f'\:X'=X\times_BB'\to B'$ and the product is taken in the fs category. In case (2) we will say that $X'\to X$ is a {\em base change morphism}.

We say that $f$ is a {\em relative orbifold} if it is log smooth. If, in addition, $f$ is representable we call it a {\em relative manifold}. As usual, resolution of an arbitrary morphism $Z\to B$ will be obtained by embedding it \'etale-locally into a relative manifold $f\:X\to B$ and principalizing $I_Z$ on the relative $B$-manifold $X$. We will use the sheaf of relative log derivations $\cD_{X/B}=\Der_{\calO_B}(\calO_X,\calO_X)$.

\subsubsection{Sharp morphisms}\label{sharpsec}
Some arguments in \cite{ATW-relative} apply only to log smooth morphisms of a special form that we will introduce now. We say that a morphism $f\:X\to B$ is {\em sharp} at a point $x\to X$ if for a geometric point $\ox$ over $x$ and $\ob=f(\ox)$ the homomorphism $\oM_\ox\to\oM_\ob$ is injective. A typical example of a non-sharp log smooth morphism is a log blowing up, see \cite[Example~4.1.5]{lognotes}. So, the following fact is not so surprising.

\begin{exer}\label{sharpex}
Let $f\:X\to B$ be a morphism of log varieties over $k$ with a log smooth target. Show that there exists a log blowing up $B'\to B$ such that the base change $f'\:X'\to B'$ is a sharp morphism. (Hint: first, solve such a problem for a chart $\bfA_Q\to\bfA_P$ of $f$ with an injective $P\into Q$ (note that the monoids do not have to be sharp). Then use quasi-compactness of $X$ and the facts that log blowings up form a filtered family and, in case of log smooth varieties, can be extended from an open subvariety because the closure of a monomial subscheme is monomial.)
\end{exer}

\begin{rem}\label{sharprem}
In fact, changing the base via a log blow up one can achieve much more: one can achieve that $f$ is {\em integral} in the sense that for each $\ox\to X$ with $\ob=f(\ox)$, $P=\oM_\ob$ and $Q=\oM_\ox$ the homomorphism $P\to Q$ is {\em integral}. The latter notions have a few equivalent formulations for which we refer to \cite[Proposition~4.1]{Kato-log}. In particular, it implies that $f$ is flat (assuming $f$ is log smooth, see \cite[4.5]{Kato-log}) and set theoretically $Q$ splits as $P\times Q/P^\gp$, where $Q/P^\gp$ denotes the image of $Q$ in $Q^\gp/P^\gp$. Existence of such a log blowing up is a combinatorial analogue of the flattening theorem of Raynaud-Gruson, though its proof is rather elementary, see \cite[Proposition~3.6.11(i)]{ATW-relative} and references there. We will not use the integralization theorem in these notes.
\end{rem}

\subsubsection{Parameters}
Recall that we have defined in \cite[\S5.2.10]{lognotes} the notion of log fibers of a morphism $f\:X\to B$. The general definition uses the stacks $\Log_B$, but for a sharp $f$ the log fibers are nothing else but the log strata of the fibers of $f$, see \cite[Exercise~5.2.11]{lognotes}.

Let $f\:X\to B$ be a log smooth morphism of log varieties, and assume that the log structures are Zariski at $x\in X$ and $b=f(x)$. By a family of {\em relative parameters} of $f$ at $x$ we mean a family $(t,q)$, where $t=(t_1\..t_n)\in\cO_{X,x}$ maps to a regular family of parameters at $x$ of the log fiber (which is smooth by the log smoothness assumption), and $q_1\..q_r$ map to a basis of $\Coker(P^\gp\to Q^\gp)\otimes\bbQ$. We call $t_i$ and $q_j$ the {\em regular} and the {\em monomial parameters}, respectively.

\subsubsection{Formal description}
Sometimes sharp log smooth morphisms are more convenient to work with, in particular, because they possess a simple formal description.

\begin{lem}\label{formallem}
Assume that $f\:X\to B$ is a log smooth morphism of log varieties, $f$ is sharp at a point $x\in X$, and the log structures of $X$ and $B$ are Zariski at $x$ and $b=f(x)$. Fix compatible monomial charts $P=\oM_b\to\cO_b$ and $Q=\oM_x\to\cO_x$, fix fields of definitions $k(x)\into \hatcO_x$ and $k(b)\into\hatcO_b$, and fix a family $t_1\..t_n\in\cO_x$ of regular parameters. Then $(\hatcO_b)_P\llbracket Q\rrbracket\llbracket t_1\..t_n\rrbracket\widehat\otimes_{k(b)}k(x)=\hatcO_x$.
\end{lem}

As in many similar statements, the natural homomorphism induced by the choices of parameters and fields of definition is easily seen to be surjective (it induces surjective homomorphisms of residue fields and cotangent spaces), and the injectivity follows by comparing relative dimensions. See the proof of \cite[Lemma~2.3.14]{ATW-relative} for details.

\subsubsection{Suborbifolds}
A {\em relative suborbifold} of $f\:X\to B$ is a strict closed substack $Y\into X$ which is log smooth over $B$. As in the absolute case, \'etale locally suborbifolds are given by vanishing of regular parameters. The argument is essentially the same with log fibers used instead of log stratas. The following exercise is an advertisement of stacks $\Log_B$. For sharp morphisms one can solve it straightforwardly using only the ``low tech'' of log schemes and charts, but the general solution which uses the ``high tech'' of stacks $\Log_B$ is simpler.

\begin{exer}
Show that if $f\:X\to B$ is a log smooth morphism of log smooth varieties and $Y\into X$ is a relative $B$-submanifold, then for any point $y\in $Y there exists regular relative parameters $t_1\..t_n\in\cO_{X,y}$ of $X/B$ such that $Y=V(t_1\..t_r)$ locally at $y$. (Hint: use that $X$ and $Y$ are log smooth over $\Log_B$ and their log fibers over $B$ are the fibers over $\Log_B$.)
\end{exer}


\subsubsection{Kummer centers and blowings up}
A {\em relative $d$-center} is an ideal of the form $\cJ=(\cI+\cN)^{(d)}$, where $\cI$ is a {\em suborbifold ideal} (i.e. $Y=V(\cI)$ is a relative suborbifold of $X$) and $\cN$ is a monomial Kummer ideal. Since $B$ is a log manifold, $X$ is a log orbifold over $k$ and $Y$ is its log suborbifold. In particular, $\cJ$ is a usual $d$-center in the sense of \S\ref{centerssec} and the Kummer blowing up along the relative center $\cJ$ to be the Kummer blowing up along $\cJ$ as defined in \S\ref{kumblowsec}. This trick saves us some work in the relative case, but one still has to use that the center is in a special position with respect to $B$ in order to prove the following

\begin{lem}
Assume that $X\to B$ is a relative orbifold with a log smooth $B$, and $\cJ$ is a relative $d$-center on $X$. Then $X'=\Bl_\cJ(X)\to B$ is a log orbifold too.
\end{lem}

Using \'etale descent the proof reduces to an explicit chart computation analogously to the argument outlined in Exercise~\ref{admblowex}. The necessary refinement is that the model case now is of the form $\Spec(A_P[Q][t_1\..t_n])\to\Spec(A)$, where $P$ and $Q$ give rise to the log structures and $t_1\..t_n$ is a family of relative parameters whose subfamily is used to define the center.

\subsubsection{Relative log order}
The {\em relative log order} of an ideal $\cI$ on a relative $B$-orbifold $X$ at a point $x$ is defined as follows: let $S=S_x$ be the log fiber containing $x$, then $\logord_{X/B,x}(\cI):=\ord_{S,x}(\cI|_S)$. As in the absolute case (see Exercise~\ref{logordex}), the relative log order can be computed using relative log derivations. As earlier, we denote the maximal relative logorder of $\cI$ on $X$ by $\logord_{X/B}(\cI)$ and view log order of $\cI$ as a function $\logord(\cI)\:X\to\bbN\cup\{\infty\}$.

\begin{exer}
(i) Show that $\cD_{X/B}$ restricts to the sheaf of derivations on $S$ and deduce that $\logord_{X/B,x}(\cI)$ is the minimal number $d$ such that $\cD_{X/B}^{(\le d)}(\cI_x)=\cO_x$. (Hint: use that log fibers and relative log derivations of the log smooth morphism $X\to B$ are usual fibers and relative derivations of the smooth morphism $X\to\Log_B$.)

(ii) Assume that $X$ is a variety. Show that $t\in\cO_{X,x}$ is a relative parameter at $x$ if and only if $\logord_x(t)=1$.
\end{exer}

\subsubsection{Base change functoriality}
All basic constructions we have discussed, including relative log derivations, centers, Kummer blowings up, and relative log order are compatible with arbitrary base changes $g\:B'\to B$, where $B'$ is smooth (due to our general restrictions), but the morphism does not have to be smooth (and  can even contract $B'$ to a point).

\begin{lem}\label{basechange}
Let $B'\to B$ be a morphism of log $k$-manifolds, let $f\:X\to B$ be a log orbifold with the pullback $f'\:X'\to B'$ and the base change morphism $g\:X'\to X$, let $\cI$ be an ideal on $X$ with pullback $\cI'=\cI\cO_{X'}$, and let $\cJ$ be a $d$-center on $X$ with pullback $\cJ'=(\cJ\cO_{X'_\ket})^\nor$. Then $\cD_{X'/B'}=g^*\cD_{X/B}$, $\logord(\cI')=\logord(\cI)\circ g$, $\cJ'$ is a $d$-center and $\Bl_\cJ'(X')=\Bl_\cJ(X)\times_XX'$.
\end{lem}

In fact, all these properties easily follow from explicit local descriptions (of log smooth morphisms, charts of blowings up, etc.). In addition, one easily sees that if $X$ is a variety and $(t,q)$ is a family of regular and monomial relative parameters at $x\in X$, then their pullbacks $(t',q')$ form a family of regular and monomial relative parameters at any $x'\in X'$ over $x$. We refer to \cite[2.4.9, 2.7.2, 2.8.14(i) 4.2.20(iv)]{ATW-relative}.

\subsubsection{Log smooth functoriality}
Similar compatibilities hold for log smooth morphisms between $B$-orbifolds.

\begin{lem}\label{logsmooothfunct}
Let $B$ be a log $k$-manifold, let $f\:X\to B$ and $f'\:X'\to B$ be log orbifolds, let $g\:X'\to X$ be a log smooth $B$-morphism, let $\cI$ be an ideal on $X$ with pullback $\cI'=\cI\cO_{X'}$, and let $\cJ$ be a $d$-center on $X$ with the pullback $\cJ'=(\cJ\cO_{X'_\ket})^\nor$. Then the natural homomorphism $\cD_{X'/B'}\to g^*\cD_{X/B}$ is surjective and splits locally, $\logord(\cI')=\logord(\cI)\circ g$, $\cJ'$ is a $d$-center and $\Bl_{\cJ'}(X')=\Bl_\cJ(X)\times_XX'$.
\end{lem}

We refer to \cite[2.4.11, 2.7.6, 2.8.15(i), 4.2.20(v)]{ATW-relative}. Compatibility with Kummer blowings up follows from its analog for logarithmic $k$-manifolds, but the situation with relative parameters and log order is a bit subtle and requires care: if $g$ is not sharp at $x'\in X'$ with $x=g(x')$, then pullbacks of some log parameters at $x$ become linearly dependent at $x'$, and hence some log parameters should be removed and some ''additional'' regular parameters at $x'$ should be added. Also, the relation between the log fibers of $f$ and $f'$ is not that simple in this case. Nevertheless, a family of regular relative parameters at $x$ pullbacks to a subfamily of a family of regular relative parameters at $x'$, and this suffices to prove that $g$ induces regular morphisms between the log fibers, see \cite[2.7.6 and 2.7.7]{ATW-relative}. Compatibility of $g$ with relative log orders follows.

\subsection{The algorithm}
Once the framework is established the resolution, principalization and order reduction algorithms are constructed precisely in the same way as in Section~\ref{42} but using relative notions instead of the absolute ones. We will recall all these in the autopilot mode, indicate the only point, where a new phenomenon happens, and proceed to formulations of the main theorems.

\subsubsection{Reduction to principalization}\label{redpr}
Let $g\:Z\to B$ be a morphism of log varieties whose log structure is Zariski. Then for any point $z\in Z$ there exists a neighborhood $U$ which possesses a strict closed immersion into a log smooth $B$-scheme $X$. Indeed, working locally we can assume $f$ is modeled on a chart $\bfA_Q\to\bfA_P$, where $Q=\oM_z$ and $P=\oM_b$ for $b=g(z)$. Let $t_1\..t_n\in\cO_z$ be a family mapping to a basis of the cotangent space of the fiber of $Z\to B_P[Q]$ at $z$, then the induced morphism $Z\to X_0=\bfA^n\times B_P[Q]$ is strict and unramified at $z$, hence on a neighborhood $U$ of $z$ this morphism factors into a composition of a strict closed immersion $U\into X$ and a strict \'etale morphism $X\to X_0$. Clearly, the relative dimension of $X$ over $B$ is minimal possible, and with a bit more care one can prove that up to an \'etale correspondence this minimal strict closed embedding of $U$ into a relative log $B$-manifold is unique, see \cite[\S8.2]{ATW-relative}.

Now, the same argument as in the classical or absolute logarithmic cases shows that functorial principalization of ideals on $B$-manifolds implies functorial resolution of log $B$-schemes (or stacks) which are locally equidimensional and generically log smooth over $B$.

\subsubsection{Maximal contacts and coefficient ideals}
Let $X\to B$ be a relative log manifold, $\cI\subseteq\cO_X$ an ideal and $x\in X$ a point such that $d=\logord_x(\cI)<\infty$. Then the definitions of the maximal contact and coefficient ideals at $x$ are the same as in \S\ref{ordredsec}: a {\em maximal contact} to $\cI$ at $x$ is any $H=V(t)$, where $t\in\cD_{X/B}^{(\le d-1)}$ is a regular relative parameter at $x$; the {\em coefficient ideal} is the homogenized sum $\cC(\cI)=\sum_{a=0}^{d-1}(\cD_{X/B}^{\le a}(\cI),d-a)$.

Furthermore, in the maximal order case $d=\logord_{X/B}(\cI)$ we have precise analogs of Theorems \ref{equivth} and \ref{homogth} on equivalence of log order $d$-reductions of $\cI$ and log order $d!$-reductions of $\cC(\cI)|_H$, and on uniqueness of $H$ and $\cC(\cI)|_H$ up to an \'etale correspondence. The arguments are also precisely the same.

Finally, the non-maximal order case is reduced to the maximal order case by a precise analogue of Lemma~\ref{balancedlem}. To summarize, in the case of a finite relative log order the algorithm and its justification are precisely the same as in the absolute case discussed in Sections 3 and 4.

\subsubsection{A complication}
It remains to consider what should be the simplest case -- the case when the log order is infinite, and once again this case turns out to be not so simple. In the absolute case we had to introduce Kummer blowings up to deal with it, and now it turns out that the algorithm just can fail at this step...

So, assume that $\logord_x(\cI)=\infty$ and we are seeking for a relative log order $d$-reduction. In such a case locally at $x$ any $\cI$-admissible center $\cJ^{(d)}$ is of infinite relative log order, that is, it is monomial. So, the only thing we can do within our framework is to blow up the whole monomial hull $\cN=\cM(\cI)$ (or its saturation) viewed as the $d$-center $\cJ^{(d)}$ with $\cJ=\cN^{1/d}$ -- it is easy to see that blowing up any larger monomial ideal will not make the relative log order finite. If blowing up this ideal does not make the relative log order finite, then the algorithm fails.

On the other hand, infiniteness of the log order is detected (and controlled) by the differential saturation $\cD^\infty_{X/B}(\cI)$, which is contained in $\cN$. In the absolute case, Lemma~\ref{hullem} guarantees that $\cD^\infty_{X}(\cI)=\cN$ and it follows easily that the log order of the transform of $\cI$ drops after blowing up $\cN$, see Lemma~\ref{finorderlem}. In the relative case we use only the submodule of relative derivations $\cD_{X/B}\subseteq\cD_X$, that is, we do not have a non-trivial way to derive pullbacks of functions on $B$. In particular, a typical example when the algorithm fails is when $\cI$ is the pullback of a non-monomial ideal $\cI_0$ on $B$.

\begin{exam}\label{monomex}
Choose any $B$ with a proper non-zero ideal $\cI_0$ whose monomial hull is trivial: $\cM(\cI_0)=\cO_B$. For example, take $B$ with the trivial log structure and any proper ideal. Consider $X=\bfA^n_B=\Spec_B(\cO_B[t_1\..t_n])$ with the log structure induced from $B$.

(i) Then the principalization fails for the ideal $\cI=\cI_0\cO_X$ because $\logord_{X/B}(\cI)=\infty$ but $\cM(\cI)=\cO_X$.

(ii) More generally, the principalization fails for the ideal $\cI'=\cI+(t_1\..t_m)$ with any $0\le m\le n$. This time the first $m$ steps of the algorithm are just restrictions to the iterative maximal contacts, and then the algorithms reaches the situation described in (i) and fails. For example, if $B=\Spec(k[b])$ and $X=\Spec(k[t])$, then the relative principalization fails for $\cI'=(t,b)$. This can be also seen directly: $\cI$ is not contained in any relative center: the absolute algorithm would blow up the center $(t,b)$, but it is not a relative center (even worse, blowing up $(t,b)$ produces a log scheme which is not log smooth over $B$).
\end{exam}

\subsubsection{Base change}
The above example also indicates a way to solve the problem: if $\cI_0$ is monomial, then $\cI'$ is a relative center and we can freely blow it up (in particular, the result stays log smooth over the base). Therefore it is natural to expect that enlarging the log structure on $B$ we can improve the situation. Certainly such an operation can make $B$ not log smooth, so in general one should also modify the underlying scheme. It turns out that indeed, one can monomialize $\cD_{X/B}$-saturated ideals just by modifying the base $B$. We formulate this monomialization theorem now and postpone a discussion about its proof until \S\ref{monomsec}. We say that a morphism $f\:B'\to B$ between log manifolds is a {\em blowing up along $\cJ\subseteq\cO_B$} if scheme-theoretically one has that $\uB'=Bl_\cJ(\uB)$ and the log structure on $B'$ is induced over the log structure of $B$ by the exceptional divisor $E_f$, that is, the divisors $D\subset B$ and $D'\subset B'$ defining the log structures are related by $D'=f^{-1}(D)\cup E_f$.

\begin{theor}\label{monomth}
Let $B$ be a log manifold over $k$, let $f\:X\to B$ be a log orbifold, and let $\cI$ be an ideal on $X$ such that $\cD_{X/B}(\cI)\subseteq\cI$. Then there exists a morphism of log manifolds $g\:B'=Bl_\cJ(B)\to B$ such that $g$ is a blowing up and the pullback $\cI'=\cI\cO_{X'}$ is a monomial ideal, where $X'=X\times_BB'$. Moreover, one can achieve that the center $\cJ$ of $g$ is monomial outside of the schematic image of the support $V(\cI)$ in $B$.
\end{theor}

\begin{rem}
If either $B=\Spec(k)$ or $B$ is a curve and the log structure is non-trivial at any point of $f(V(\cI))$, then any blow up $B'\to B$ whose center is monomial outside of $f(V(\cI))$ is trivial (because the only way to modify $B$ is to increase the log structure). Thus, in this case the theorem just claims that $\cI$ is monomial and when $B$ is a point we recover Lemma~\ref{hullem}. In particular, for such a base $B$ the main results we will formulate below hold without any modification of $B$.
\end{rem}

\subsubsection{Relative log order reduction}
The first main result concerns relative log order $d$-reduction, whose definition copies Definition~\ref{ordereddef} with $\cI$-admissible centers replaced by $\cI$-admissible $B$-relative centers.

\begin{theor}\label{relativeorderred}
There exists a method $\calF$ which accepts as an input a log orbifold $f\:X\to B$, where $B$ is a log manifold, an ideal $\cI\subseteq\cO_X$ and a number $d\ge 1$, and either fails or outputs a relative log order $d$-reduction $\cF(f,\cI,d)\:X'\longto X$ of $\cI$ such that the following conditions are satisfied.

(i) Existence: There exists a blowing up $B'\to B$ whose center outside of $\overline{f(V\cI))}$ is monomial such that $B'$ is a log manifold and $\cF$ does not fail on $f'\:X'=X\times_BB'\to B'$, $\cI'=\cI\cO_{X'}$ and $d$.

(ii) Base change functoriality: if $\cF$ does not fail on $f,\cI,d$, then for any morphism of log manifolds $B'\to B$ with fs base change $f'\:X'=X\times_BB'\to B'$ and $\cI'=\cI\cO_{X'}$ we have that $\cF(f',\cI',d)$ is obtained from $\cF(f,\cI,d)\times_XX'$ by omitting all trivial blowings up (in particular, it does not fail).

(iii) Log smooth functoriality: if $\cF$ does not fail on $f,\cI,d$, then for any log smooth $g\:X'\to X$ with $\cI'=\cI\cO_{X'}$ we have that $\cF(f\circ g,\cI',d)$ is obtained from $\cF(f,\cI,d)\times_XX'$ by omitting all trivial blowings up.

(iv) For closed immersions of constant codimension $X\into X'$ of $B$-orbifolds the re-embedding principle is satisfied.
\end{theor}

In the more classical language of marked ideals $(\cI,d)$ this is \cite[Theorem~7.1.1]{ATW-relative}. The algorithm is constructed precisely as its absolute analogue (and particular case when $B=\Spec(k)$) in \S\ref{algsec}. The new claims in the relative algorithm are related to base changes -- (i) and (ii). The proof of (ii) is straightforward because all ingredients of the framework satisfy the base change functoriality, see Lemma~\ref{basechange}. Existence is proved inductively as follows. Recall that the algorithm in \S\ref{algsec} produces an $\cI$-admissible $d$-sequence $$Y\to X_n\dashto X_{n-1}\dashto\ldots\dashto X_1\to X,$$ where the first blowing up makes the log order finite, and the sequences $X_{i+1}\dashto X_i$ are constructed inductively (using maximal contacts) and reduce the log order of the clean part until it drops below $d$.

In the relative situation each step can fail. However, the first step certainly succeeds if $\cD^\infty_{X/B}(\cI)$ is monomial, and by the monomialization theorem \ref{monomth} this condition is achieved after an appropriate blowing up $B'\to B$. Thus, setting $X'=X\times_BB'$ and $\cI'=\cI\cO_{X'}$ we have that $\cD^{\infty}_{X'/B'}(\cI')=\cD^\infty_{X/B}(\cI)\cO_{X'}$ is monomial and hence the first step of $\cF(f',\cI',d)$ does not fail. Moreover, this is also the case for any further base change $B''\to B'$. By the induction assumption, the second step of $\cF(f',\cI',d)$ succeeds after an appropriate blowing up $B''\to B'$ and then it also succeeds after any further blowing up by claim (ii) of the theorem, and so on. In the end we find a sequence of blowings up $B^{(n)}\dashto B$ (which can be represented as a single blowing up) such that all steps except the final cleaning succeed after the base change $B^{(n)}\to B$. It remains to note that the final cleaning blows up an exceptional divisor, hence it succeeds automatically.

\begin{rem}
Using the monomialization theorem as a black box, the required base change $B'\to B$ is dictated by $\cF$ in a canonical (in fact, simple algorithmic) way. However, the current proof of the monomialization theorem is existential and difficult, see the discussion in \S\ref{monomsec} below.
\end{rem}

\subsubsection{Relative principalization}
Taking $d=1$ in the previous theorem one obtains the relative principalization theorem \cite[Theorem~1.2.6]{ATW-relative}.

\begin{theor}\label{rellogprincth}
On the category of relative log orbifolds whose targets are log manifolds over $k$ there exists a relative logarithmic principalization method $\cP$, which is base change functorial and log smooth functorial, satisfies the re-embedding principle, and succeeds on each $(f\:X\to B,\cI)$ after a large enough blowing up of $B$ whose center is monomial outside of $\overline{f(V(\cI))}$.
\end{theor}

\subsubsection{Relative desingularization}
Finally, by the usual methods discussed in \S\ref{redpr} the above theorem implies the following functorial semistable reduction theorem, see also \cite[Theorem~1.2.12]{ATW-relative}.

\begin{theor}\label{reldesing}
There exists a relative desingularization method $\cR$ which accepts as an input a generically log smooth morphism $g\:Z\to B$ of $k$-varieties with a locally equidimensional $Z$ and a log smooth $B$, and either fails or outputs a stack-theoretic modification $Z'\to Z$ such that $Z'\to B$ is log smooth and the following condition are satisfied.

(i) Existence: There exists a blowing up $B'=Bl_\cJ(B)\to B$ such that $B'$ is a log manifold, $\cJ|_U$ is monomial for any open log subscheme $U\subseteq B$ such that the restriction $Z\times_BU\to U$ is log smooth, and $\cR$ does not fail on the base change $g'\:Z'=Z\times_BB'\to B'$.

(ii) Base change functoriality: if $\cR$ does not fail on $g$, then for any morphism of log manifolds $B'\to B$ with fs base change $g'\:Z'=Z\times_BB'\to B'$ we have that $\cR(g')=\cR(g)\times_BB'$.

(iii) Log smooth functoriality: if $\cR$ does not fail on $g$, then for any log smooth $h\:X'\to X$ we have that $\cR(g\circ h)=\cR(g)\times_XX'$.
\end{theor}

\begin{rem}
(i) The above theorem can be viewed as a weak form of a semistable reduction theorem over an arbitrary base. By complicated but purely combinatorial methods one can improve the log smooth morphism $Z'\to B$ by a log blowing up $Z''\to Z'$ so that $Z''\to B$ is a so-called semistable morphism, see \cite{semistable}. When $B$ is one-dimensional this is precisely the semistable reduction of \cite{KKMS}.

(ii) So far Theorem~\ref{reldesing} is the only known resolution of morphisms (or semistable reduction) compatible with base changes, even in the case when the dimension of $B$ is $1$. In particular, using noetherian approximation on the base it implies semistable reduction theorem over any valuation ring, not necessarily discretely valued. Proving the latter was one of the main motivations for developing the logarithmic resolution methods.
\end{rem}

\subsubsection{Destackification}
As in the absolute case, the relative principalization and desingularization methods output a stack-theoretic modification. Using a canonical relative destackification procedure one can obtain a further modification $X''\to X'$ (resp. $Z''\to Z'$) such that the relative coarse space $X''_{\cs/X}$ (resp. $Z''_{\cs/Z}$) is log smooth over $B$. This provides a representable principalization and desingularization methods which are only smooth functorial, see \cite[Theorem~1.2.14]{ATW-relative}.

\subsection{The monomialization theorem}\label{monomsec}
Finally let us discuss the proof of the monomialization theorem. The argument in \cite[\S3]{ATW-relative} is surprisingly involved and most probably some improvements will be found in the future, so we will only outline the main ideas. Possibly the concept of a separate monomialization theorem is suboptimal, and a more natural monomialization should be intertwined with principalization. We will discuss some arguments in favor of this in the end of the section.

\subsubsection{The case of $\dim(B)\le 1$}
We start with the simple but already very useful case, when $\dim(B)\le 1$. Recall that the case when $B$ is a point was already established in Lemma~\ref{hullem} (and Exercise~\ref{operatorexer}). In fact, the case of a curve is similar. Working locally on $B$ it suffices to consider the case when $B=\Spec(R)$, where $R$ is a field or a dvr with a uniformizer $\pi$ and the log structure is generated by $\bbN\log(\pi)$. Furthermore, monomiality satisfies formal descent: $\cI$ is monomial at $x\in X$ if and only if $\hatcI$ is monomial at $\hatcO_x$ (we use that the completion homomorphism is flat and hence the ideal $\cI_x$ is determined uniquely by its formal completion), hence we should only prove that $\hatcI$ is monomial. Finally, the claim is \'etale-local on $X$, hence we can assume that the log structure is Zariski. Now, the claim reduces to the following formal computation. This is a difficult exercise and we outline the main lines of the argument.

\begin{exer}
Assume that $R=k\llbracket\pi\rrbracket$ is a complete DVR provided with the log structure generated by $P=\bbN\log(\pi)$ and $$A=R_P\llbracket Q\rrbracket\llbracket t_1\..t_n\rrbracket\wtimes_kl=l\llbracket Q\rrbracket\llbracket t_1\..t_n\rrbracket$$ where $P\into Q$ is an embedding of sharp monoids and $l/k$ is a field extension. Consider  the module $\cD\subseteq\Der_R(A,A)$ generated by log derivations of two types: (i) $\partial_i=\partial_{t_i}$ vanishes on $l$, $Q$, each $t_j$ with $j\neq i$ and satisfies $\partial_i(t_i)=1$, (ii) for any $\phi\:Q^\gp\to\bbZ$ with $\phi(P)=0$ the derivation $\partial_\phi$ vanishes on $l$, each $t_i$ and restricts to $\phi$ on $Q$. Prove that any $\cD$-stable ideal $I\subseteq A$ (i.e. $\cD(I)\subseteq I$) is monomial. (Hint: One should adopt the proof from Exercise~\ref{operatorexer} to this situation. First, use the same argument as there to show that $I$ is generated by elements of the form $c=\sum_i c_iq_i$, where $c_i\in l^\times$ and $q_i$ have the same image in $Q^\gp/P^\gp\otimes\bbQ$ (unlike Exercise~\ref{operatorexer} we cannot distinguish elements from $P$ by $R$-derivations). Choose $i_0$ such that $q_i-q_{i_0}$ is a non-negative element in $P^\gp$ for any $i$ (informally $q_i=q_{i_0}+r\pi$ with $r\in\bbQ_{\ge 0}$), and show that $q_i-q_{i_0}\in Q$ for any $i$. Deduce that $c=uq_{i_0}$ for a unit $u\in A^\times$, and hence $I$ is monomial.)
\end{exer}

\subsubsection{The general case}
If $B$ is arbitrary one starts with the same approach but has to bypass various technical complications. A minor issue is that if $x$ is not closed in the fiber and $l=k(x)$ is not algebraic over $k(b)$, one should also consider derivations in ``constant'' directions. A more serious obstruction is that derivations do not allow to control algebraic extensions $l/k$ and the torsion of $Q^\gp/P^\gp$. In the formal computation one assumes that these obstructions are trivial, and the general case is reduced to this by a Kummer descent and a cofinality argument, see \cite[Lemmas~3.5.8, 3.6.3]{ATW-relative}. Furthermore, in the formal case one only shows that $I$ is generated by an ideal $\hatJ\subseteq\hatcO_b$ and this is the maximum one can get from derivations -- by continuity $\cO_b$-derivations vanish on $\hatcO_b$.

Nevertheless, if $\hatJ$ is open, then it is the completion of an ideal $J\subseteq\cO_b$, and blowing up $J$ on $B$ we achieve monomialization. This is automatically the case when $f(V(I))$ is a finite union of closed points of $B$, and the general case is achieved by induction on the dimension of $\overline{f(V)}$ -- first one constructs a blowing up which monomializes $\cI$ over the generic points of $\overline{f(V)}$, then over the generic points of the remaining bad locus, etc.

The following remark is rather technical, and can be safely skipped by the reader.

\begin{rem}
The framework of the relative principalization is constructed for arbitrary log smooth morphisms $f\:X\to B$, including morphisms which are not sharp. Sometimes this made constructions more complicated, though usually it sufficed just to work with the morphism $X\to\Log_B$. In contrast, the proof of the monomialization theorem only works when $f$ is sharp (and for simplicity it is even assumed to be integral at one place in \cite[Section ~3]{ATW-relative}). This forces one to start with a large enough monomial blowing up of the base and results in a monomializing blowing up which can be non-trivial (though monomial) outside of $\overline{f(V(\cI))}$. We do not know if this can be improved. As a consequence, the same limitation holds in the formulation of the main theorems on relative principalization and desingularization. In fact, even if one starts principalization with a sharp morphism $f$, typical admissible blowings up of $X$ will no longer be sharp over the base, and each time the algorithm will have to modify the base it will also modify the monomial locus over which the morphism is not sharp.
\end{rem}

\subsubsection{Canonicity of the base change}
The monomialization theorem proved in \cite{ATW-relative} is existential. It is a natural question if the modification $B'\to B$ can be found in a canonical (or functorial) way. The answer seems to be positive, but working out details turned out to be really heavy and we have not brought it to satisfactory form. In addition, the resulting method, although canonical, may be difficult for a practical implementation. The main reason for this is that the problem is not local -- the situation at $b\in B$ depends on the whole fiber of $X$ over $b$, which does not even have to be connected. Such a method cannot be analogous to usual embedded principalization.


\section{The dream algorithms}\label{dreamsec}
The Kummer centers used in the logarithmic method are of the form $$(t_1\..t_r,u^{p_1/d}\..u^{p_s/d}),$$ where $u^{p_i}$ are monomials. This can also be viewed as a stack-theoretic refinement of the weighted blowing up with weights $(1\..1,d\..d)$ or a usual blowing along $(t_1^d\..t_r^d,u^{p_1}\..u^{p_s})$, so the following question is natural: once the stack-theoretic methods are at our disposal, can one use arbitrary weights? Can one blow up centers like $(t_1^{d_1}\..t_r^{d_r})$? In fact, we already saw in \cite{stacksnotes} that indeed, there is a natural stack-theoretic definition of such weighted blowings up of manifolds which outputs an orbifold (and will briefly recall some details below).

The next question is if including arbitrary weighted blowings up in the basic framework leads to a new algorithm. This was precisely the question we studied after discovering the logarithmic principalization, and the result was somewhat unexpected. In a sense it was too good: weighted principalization (and resolution) does not need logarithmic structure (or boundary) at all, and it works by reducing a natural simple invariant by each weighted blowing up. In other words, one obtains an ``ideal'' algorithm without inner structure/memory -- it just iteratively finds an admissible center with largest possible invariant and blows it up. The famous example of Whitney umbrella shows that such an algorithm does not exist within the classical framework, but we saw in \cite[\S3.4]{weightedexpo} how a single weighted blowing up at the pinch point improves the singularity.

In addition, one can also consider the general weighted centers in the logarithmic setting and this leads to a logarithmic dream algorithm, which was developed by Quek in \cite{Quek} in the absolute case.

\subsection{Weighted blowings up}\label{wblowsec}
The main new ingredient in weighted algorithms is provided by weighted blowings up. So, we start with a brief review of their definition, historical context, and modern formalism. The principalization algorithm only blows up smooth varieties $X$, but we will try to formulate our definitions and statements in the maximal generality, when they apply without significant changes -- usually this will be the generality of reduced or normal varieties.

\subsubsection{Weighted blowings up}
Let $X=\Spec(A)$ be a smooth affine variety, let $t_1\..t_d\in A$ be functions on $X$ that form a partial family of regular parameters at any point of $V=V(t_1\..t_d)$ (in other words, $V$ is smooth of pure codimension $d$), and let $w_1\..w_d\in\bbN_{\ge 1}$. The classical weighted blowing up $X'\to X$ associated with $t_1\..t_d$ and weights $w_1\..w_d$ is glued from the charts $X'_i=\Spec(A_i)$, where $A_i$ is the normalization of the $A$-subalgebra of $A_{t_i}$ generated by the fractions $(\prod_{j=1}^dt_j^{n_j})/t_i^n$ such that $\sum_{j=1}^dn_jw_j\ge nw_i$.

\begin{exer}
(i) Check that $X'=Bl_{\left(t_1^{l_1}\..t_d^{l_d}\right)}(X)^\nor$, where $l_i\ge 1$ are chosen so that $l_iw_i$ is the same number $N$ for any $i$. In other words, $$(w_1:w_2:\ldots:w_d)=(l_1:l_2:\ldots :l_d)^{-1}.$$

(ii) Check that blowing up of $\Spec(k[x,y])$ along $(x,y)$ with weights $1,n$ creates a singularity of type $A_{n-1}$ (\'etale locally looking as $uv=w^n$).
\end{exer}

\subsubsection{A stack theoretic refinement}
In order to freely use weighted blowings up in principalization one should modify their definition so that the outcome is smooth. As in the logarithmic algorithm this can be achieved by a stack-theoretic refinement of the classical definition. The most naive way is analogous to the first approach mentioned in \S\ref{twoapproaches} -- one (locally) considers the Galois cover $Y\to X$ generated by $s_i=t_i^{1/w_i}$ with $1\le i\le d$ and the blowing up $Y'=Bl_{(s_1\..s_d)}$ and then defines the weighted blowing up as the stack-theoretic quotient $X'=[X'/G]$, where $G=\prod_{i=1}^d\mu_{w_i}$.

The disadvantage of this definition is that, similarly to the classical definition of weighted blowings up, it is very ad hoc and coordinate dependent. It is even not so easy to globalize it, not to mention such notions as $\cI$-admissibility, etc. This happens because intuitively the weighted blowing up center is something of the form $(t_1^{1/w_1}\.. t_d^{1/w_d})$, but the  notation needs to be formalized. So the solution should start with introducing a notion of generalized ideals, where such beasts can live.

\subsubsection{Valuative ideals}
Similarly to Kummer ideals one can try to consider an appropriate topology $\tau$ on $X$ where covers generated by extracting roots of parameters are open covers and consider ideals in $\cO_{X_\tau}$. Since such covers do not even have to be flat, it seems most natural to consider the $h$-topology or the equivalent topology generated by open covers, finite covers and modifications. It turns out that points of these topologies are easily described -- they are valuations on $k(X)^a$ with center on $X$, so one can use the much more down-to-earth definition of valuative ideals that we are going to recall. We start with a bit more particular case, where roots are not extracted.

\begin{exer}\label{val1}
For a reduced variety $X$ consider the {\em Riemann-Zariski space} $$\RZ(X)=\gtX=\lim_{X_i\to X}X_i,$$ where the limit is taken over all modifications $X_i\to X$ in the category of locally ringed spaces. In particular, $|\gtX|=\lim_i|X_i|$ and $\cO_\gtX$ is the colimit of pullbacks of $\cO_{X_i}$ to $\gtX$.

(0) Show that $\RZ(X)=\coprod_{i=1}^n\RZ(X_i)$, where $X_1\..X_n$ are the irreducible components of $X$.

(i) Prove the following results from \cite[\S3.2]{Temkin-stable}: if $X$ is integral, then for each $x\in\gtX$ the ring $\cO_{\gtX,x}$ is a valuation ring of $k(X)$ with center on $X$. Deduce that this provides a bijective correspondence between the points of $\gtX$ and valuation rings of $k(X)$ with center on $X$.

(ii) Show that ideals in the topology generated by modifications and Zariski covers correspond bijectively to ideals of $\cO_\gtX$. In particular, a finitely generated ideal $\cI\subseteq\cO_\gtX$ is generated by an ideal $\cI_i$ on some modification $X_i$ of $X$. Moreover, $\cI$ is invertible and $\cI_i$ can be chosen invertible. (Hint: one can either use that a finitely generated ideal in a valuation ring is principal or that any ideal becomes invertible after blowing up.)

(iii) Assume that $X$ is normal. Show that two ideals $\cI,\cJ$ induce the same ideal on $\gtX$ if and only if $\cI^\nor=\cJ^\nor$. (Hint: if $\cI^\nor=\cJ^\nor$, then already their pullback to $Bl_\cI(X)^\nor$ coincide. In the opposite direction prove that $\cI^\nor=\cap_{x\in\gtX}\cI\cO_{\gtX,x}$ analogously to the fact that the integral closure of a domain $A$ coincides with the intersection of all valuation rings of $\Frac(A)$ containing $A$.)
\end{exer}

Furthermore, a finitely generated ideal in $\cO_\gtX$ can be described by a section of the {\em sheaf of values} $\Gamma=\Gamma_\gtX=k(X)^\times/\cO_\gtX^\times$ or even of its positive part $\Gamma_+=(\cO_\gtX\setminus\{0\})/\cO_\gtX^\times$. This provides the most elementary way to define such ideals, and it is easy to translate various operations on ideals into the language of sections:

\begin{exer}\label{val2}
(i) Show that the stalk $\Gamma_x$ at $x\in\gtX$ is the group of values of the valuation ring $\cO_{\gtX,x}$, and $\Gamma_{+,x}$ is the submonoid of non-negative elements.

(ii) Show that the homomorphism $(\cO_\gtX\setminus\{0\})\onto\Gamma_+$ (resp. $k(X)^\times\onto\Gamma$) induces a bijective correspondence between finitely generated ideals of $\cO_\gtX$ and invertible ideals of the monoid $\Gamma_+$ and the latter are in a bijective correspondence with the generators $s\in\Gamma_+(\gtX)$, which are global sections of $\Gamma_+$. Prove that in the same way sections of $\Gamma$ correspond to fractional ideals of $\cO_\gtX$.

(iii) Let $I=(f_1\..f_n)$ be an ideal on $X=\Spec(A)$ and $\cI$ the induced ideal on $\gtX$ with the associated section $v_I\in\Gamma_+(X)$. Show that $v_I=\min_i v_i$, where $v_i$ is the section corresponding to $(f_i)$, that is, $v_i(x)=v_x(f_i)$ where $v_x\:k(X)^\times\to\Gamma_x$ is the valuation of $x\in\gtX$. Also show that for any $1\le i\le n$ the locus given by $v_I=v_i$ consists of the valuations centered on the $f_i$-chart of $Bl_{I}(X).$

(iv) Show that $v_{I+J}=\min(v_I,v_J)$ and $v_{IJ}=v_I+v_J$. In particular, $v_{I^n}=nv_I$. Also show that $v_I\le v_J$ if and only if $J\subseteq I^\nor$. (Hint: by Exercise~\ref{val1}(iii) $v_I=v_J$ if and only if $I^\nor=J^\nor$.)

(v) Let $X=\Spec(k[x,y])$. Find ideals representing $s_1=\min(v_x,v_y)$ and $s_2=\max(v_x,v_y)$. (Hint: since $s_1+s_2=v_x+v_y$, the section $s_2$ should be represented by something like $(xy)I^{-1}$, where $I=(x,y)$. The latter can not be represented by an ideal on $X$, but makes perfect sense on the modification $X'=Bl_I(X)$.)
\end{exer}

\begin{rem}
The elements of $\Gamma_+(\gtX)$ are called {\em valuative ideals}. We have just seen that they really encode finitely generated ideals and the formalism of valuative ideals is extremely simple.
\end{rem}

\subsubsection{Valuative $\bbQ$-ideals}
Since valuative ideals are very simple objects one can easily extract roots from them: a {\em valuative $\bbQ$-ideal} is a section of the sheaf $\Gamma_{\bbQ,+}$, which is the saturation of $\Gamma_+$ in $\Gamma_\bbQ=\Gamma\otimes\bbQ$. One can use the same formalism to operate with valuative $\bbQ$-ideals as in the previous section -- the basic operations are summation, minimum and multiplication by a positive rational number.

We will not need the following remark about other interpretations of valuative $\bbQ$-ideals.

\begin{rem}
(i) A valuative $\bbQ$-ideal is an effective $\bbQ$-Cartier divisor on a fine enough modification.

(ii) In characteristic zero one can view a valuative $\bbQ$-ideals as an $h$-ideal.
\end{rem}

\begin{exam}
(i) If $f_1\..f_n\in\cO_X(X)$ are global functions and $q_1\..q_n\in\bfQ_{>0}$ then we will use the suggestive notation $(f_1^{q_1}\..f_n^{q_n})=\sum_{i=1}^n(f_i)^{q_i}$ to denote the valuative ideal $\gamma=\min_{1\le i\le n}q_iv_{f_i}$.

(ii) Assume now that $X$ is smooth. If $\gamma$ is such that locally on $X$ there exists a presentation $\gamma=(t_1^{q_1}\..t_n^{q_n})$ such that the {\em support} $V(t_1\..t_n)$ of $\gamma$ is smooth of codimension $n$, then $\gamma$ is called a {\em $\bbQ$-regular center} on $X$. Such  a center is called a {\em smooth weighted center} if one can choose a presentation with $q_i=1/w_i$ for each $i$, and a smooth weighted center is {\em reduced} if in addition $(w_1\.. w_n)=1$. The tuple $\bfw=(w_1\..w_n)$ is called the {\em tuple of weights}.
\end{exam}

If the support of a $\bbQ$-regular center is of codimension 2 and higher, then there are many different ways to choose the regular parameters. For example, $V(x,y^q)=V(x+y^n,y^q)$ for any natural $n\ge q$. However the weights are well defined.

\begin{exer}
Show that the multiplicities $q_1\..q_n\in\bbQ_{>0}$ are uniquely determined by a $\bbQ$-regular center $\gamma=(t_1^{q_1}\..t_n^{q_n})$.
\end{exer}

\subsubsection{Blowings up of valuative $\bbQ$-ideals}
To any valuative $\bbQ$-ideal $\gamma\in\Gamma_{\bbQ,+}(\gtX)$ on a normal variety $X$ one associates the graded $\cO_X$-algebra $\cR_\gamma=\oplus_{d\ge 0}\cR_d$ called the {\em Rees algebra} of $\gamma$ and defined by $\cR_d(U)=\{f\in\cO_X(U)|\ v_f\ge d\gamma|_\gtU\}$, where $\gtU=\RZ(U)\subseteq\gtX$. Then the blowing up of $X$ along $\gamma$ is the stack-theoretic Proj of the Rees algebra: $Bl_\gamma(X)=\cProj_X(\cR_\gamma)$.

\begin{rem}
Each $\cR_d$ is an ideal on $X$, which is the pushforward of the $d$-th power of the ideal $(\gamma\cO_\gtX)^d$. So our definition of the Rees algebra and its blowing up is a precise analogue of the definition of Kummer blowings up in \ref{kumblowsec}.
\end{rem}

\begin{exer}
Let $X$ be a normal variety with a valuative $\QQ$-ideal $\gamma$ of the special form $\gamma=(f_1^{1/d_1}\..f_n^{1/d_n})$ (these are $\QQ$-ideal or idealistic exponents from \S\ref{idealexp} below).

(i) The algebra $\cR_\gamma$ is integrally closed and finitely generated over $\cO_X$. 

(ii) If $\gamma=v_\cI$ for a usual ideal $\cI\subseteq\cO_X$, then $\cR_\gamma=\oplus_{d\in\bbN}\cI^{(d)}=(\oplus_{d\in\bbN}\cI^d)^\nor$ is the integral closure of the usual Rees algebra of $\cI$. In particular, $Bl_\gamma(X)=Bl_\cI(X)^\nor$ and it is singular even for the valuative center associated with $\cI=(t_1,t_2^2)$.

(iii) $\gamma$ corresponds to a usual invertible ideal on $X'=Bl_\gamma(X)$.
\end{exer}

\subsubsection{Smooth weighted blowings up}
By a {\em smooth weighted blowing up} of $X$ we mean blowing up of a smooth weighted center. Such blowings up output smooth stacks, unlike the blowings up along an arbitrary $\bbQ$-reduced center. The following result is established by a direct chart computation (e.g. see \cite[\S3.6]{weightedexpo}).

\begin{exer}\label{smoothweightedex}
(i) Let $w=(w_1\..w_n)\in\bbN_{\ge 1}^n$ be a tuple of weights and $\gamma=(t_1^{1/w_1}\..t_n^{1/w_n})$ a smooth weighted center on $X$. Then $X'=Bl_\gamma(X)$ is a smooth DM stack whose stabilizers on the $i$-th chart are subgroups of $\mu_{w_i}$. The smooth weighted center $\gamma$ becomes a usual invertible ideal on $X'$ which will be denoted $\gamma\cO_{X'}$.

(ii) If $d\ge 1$ and $\gamma'=d^{-1}\gamma=(t_1^{1/dw_1}\..t_n^{1/dw_n})$, then $Bl_{d^{-1}\gamma}$ is the root stack obtained from $Bl_\gamma(X)$ by extracting the $d$-th root from $\gamma\cO_{X'}$.
\end{exer}

\subsubsection{Associated weighted blowings up}
The weighted algorithms use certain $\bbQ$-regular centers including all those with natural multiplicities, which correspond to usual ideals. Blowing up such a center can output a singular variety, hence we should use a stack-theoretic refinement instead. The trick is to blow up an appropriate root of the center.

\begin{defin}\label{Def:wBl}
Assume that $\gamma=(t_1^{q_1}\..t_n^{q_n})$ is a $\bbQ$-regular center. Chose the representation $q_i=a_i/b_i$ with $(a_i,b_i)=1$ and let $a=\lcm(a_1\..a_n)$. Then $a^{-1}\gamma=(t_1^{q_1/a}\..t_n^{q_n/a})$ is the {\em smooth weighted center associated} with $\gamma$ and $X'=Bl_{a^{-1}\gamma}(X)$ is the {\em associated weighted blowing up}. We will use the notation $X'=wBl_\gamma(X)$.
\end{defin}

\begin{rem}
The blowing up $Bl_\gamma(X)$ is a partial coarsening of the weighted blowing up $wBl_\gamma(X)$. The scaling factor $a=\lcm(a_1\..a_n)$ is chosen so that the stack-theoretic refinement $wBl_\gamma(X)$ is smooth and the stacky structure is increased as little as possible.
\end{rem}

\subsection{$\bbQ$-ideals and idealistic exponents}\label{idealexp}
In principle, the formalism of valuative $\bbQ$-ideals introduced in \cite{ATW-weighted} covers our needs. This section will not be used in the sequel, but it is quite enlightening to also study the smaller class of $\bbQ$-ideals which really play a role in this story, especially because they formalize the classical notions of idealistic exponent and marked ideals. These notions are well behaved only on normal schemes, so we restrict to this generality. We also refer to \cite[\S2.2]{Quek} and \cite[\S2.1]{weighted}.

\subsubsection{$\bbQ$-ideals}
A {\em $\bbQ$-ideal} on a normal scheme $X$ is a valuative $\bbQ$-ideal which locally on $X$ is of the form $\gamma=\min(d_i^{-1}v_{f_i})$ where $f_i$ are functions. In other notation, $\gamma=(f_1^{1/d_1}\..f_n^{1/d_n})$. In particular, a $\bbQ$-regular center is a $\bbQ$-ideal. In a sense, $\bbQ$-ideals generalize usual ideals in the same meaning as valuative $\bbQ$-ideals generalize valuative ideals, and the following exercise formalizes this point of view.

\begin{exer}
(i) Let $\gamma$ be a valuative $\bbQ$-ideal. Show that it is a usual ideal (i.e. $\gamma=v_\cI$ for an ideal $\cI$) if and only if $\gamma$ is both a valuative ideal and a $\bbQ$-ideal. In particular, $\gamma$ is a $\bbQ$-ideal if and only if $d\gamma$ is an ideal for some $d$.

(ii) Show that the multiplicative monoid of $\bbQ$-ideals is uniquely divisible and hence it is the divisible hull of the monoid of integrally closed ideals. In particular, $\bbQ$-ideals can be safely presented in the form $\cI^{1/d}$, where $\cI$ is an ideal (or its integral closure).

(iii) A marked ideal $(\cI,d)$ can be viewed as the $\bbQ$-ideal $\cI^{1/d}$: show that this correspondence agrees with the usual operations on marked ideals -- inclusion, multiplication and summation (homogenized or not).

(iv) Give an example of a valuative ideal, which is not a $\bbQ$-ideal. (Hint: one can take $X=\Spec(k[x,y])$ and $\gamma=\max(0,\min(v_x,v_y-v_x))$ -- the valuative ideal corresponding to the ideal $(x,y/x)$ on $Bl_{(x,y)}(X)$. Then the minimal $\bbQ$-ideal containing $\gamma$ is $\min(v_x,v_y)$.)
\end{exer}

\begin{rem}
(i) In fact, the notion of a $\bbQ$-ideal is nothing else but a formalization of Hironaka's notion of idealistic exponent. As we saw, this can be done in two ways -- either realize them as valuative $\bbQ$-ideals of a special form, or as roots of usual ideals considered up to integral closure.

(ii) The classical order reduction of a marked ideal $(\cI,d)$ can now be formalized as reducing the order of the $\QQ$-ideal $\cI^{1/d}$ below 1 by blowing up smooth centers $\cJ$ and factoring out their pullback. Our interpretation in \S\ref{classicalfirst} is that one reduces below $d$ the order of $\cI$ itself by blowing up centers $\cJ^d$. These interpretations just differ by normalization.
\end{rem}

\subsubsection{Normalized blowings up}
Normalizing blowings up of $\QQ$-ideals to stay in the category of normal schemes one obtains the following universal property, see \cite[Theorem~3.4.3]{Quek-Rydh}.

\begin{theor}\label{normth}
If $\cI^{1/d}$ is a $\QQ$-ideal on a normal variety $X$ and $X'=Bl_{\cI^{1/d}}(X)^\nor$ is the normalized blowing up, then $\cI^{1/d}\cO_{X'}$ is an invertible ideal and $f\:X'\to X$ is the universal morphism of normal schemes with this property: if $g\:T\to X$ is a morphism of normal schemes such that $g^{-1}(X\setminus V(\cI))$ is dense in $T$, then there exists at most one factorization of $g$ through $f$ and it exists if and only if $\cI^{1/d}\cO_T$ is an invertible ideal.
\end{theor}

In particular, we obtain a generalization of Exercise \ref{smoothweightedex}(ii).

\begin{rem}
This universal property immediately implies that the normalized blowing up along $\cI^{1/d}$ can be described using the usual normalized blowing up $Y=Bl_\cI(X)^\nor\to X$ and the normalized root stack construction $Bl_{(\cI\cO_Y)^{1/d}}(Y)\to Y$: the first makes the pullback of $\cI$ invertible and the second extracts the $d$-th root.
\end{rem}

\subsection{Rees algebras and Rees blowings up}
In order to describe non-embedded resolution in the most precise way, we should also consider non-normal varieties and the weighted blowings up induced on them from weighted blowings up of the ambient manifold. This theory was developed by Rydh in his work on Nagata compactification for stacks and later by Quek-Rydh in \cite{Quek-Rydh}. One has to switch to the language of arbitrary (non-normal) Rees algebras. The reader can safely skip (or just look through) this section; it will only be used to formulate the non-embedded weighted resolution in the most precise way.

\subsubsection{Rees algebras}
The $\bbQ$-ideals can also be interpreted in terms of another classical object -- Rees algebra. As in \cite{Quek-Rydh}, by a {\em Rees algebra} on a variety $X$ we mean a finitely generated graded $\cO_X$-algebra $$R=\oplus_{d\in\bbN} I_dt^d\subseteq\cO_X[t]$$ such that $I_0=\cO_X$ and the ideals $I_d$ form a decreasing sequence: $I_n\subseteq I_m$ for $m\le n$ (the latter condition is automatic when $R$ is normal). By the support of $R$ we mean $V(R)=V(I_1)$. For any $n\ge 1$ we have that $I_1^n\subseteq I_n\subseteq I_1$, hence $V(I_n)=V(R)$.

If $R$ is normal, then necessarily $X$ is normal and $I_1$ is integrally closed. Conversely, if $X$ is normal then the normalization $R^\nor$ of a Rees algebra is also a Rees algebra on $X$ (for general schemes one should also assume that $X$ is quasi-excellent, or, at least, Nagata). In particular, for any valuative $\QQ$-ideal $\gamma$ on a normal scheme $X$ we have that $\cR_\gamma$ is a normal Rees algebra. The natural construction in the opposite direction associates to a Rees algebra $R$ the valuative $\bbQ$-ideal $\gamma_R=\min_{d\ge 1}d^{-1}v_{I_d}$. In fact, it is a $\bbQ$-ideal which possesses a concrete description in term of generators: if $R$ is generated over $\cO_X$ by $f_1t^{d_1}\..f_nt^{d_n}$, then $\gamma_R=\min(d_1^{-1}v_{f_1}\..d_n^{-1}v_{f_n})$.

\begin{exer}
Show that for any Rees algebra $R$ on a normal variety $X$ one has that $\gamma_R$ is the minimal valuative $\bbQ$-ideal whose Rees algebra equals $R$. Deduce that the constructions $\gamma_R$ and $R_\gamma$ provide a one-to-one correspondence between normal Rees algebras and $\bbQ$-ideals on $X$. Also deduce that for any valuative $\bbQ$-ideal $\gamma$ with $R=\cR_\gamma$ one has that $\gamma_R$ is the maximal $\bbQ$-ideal such that $\gamma_R\le\gamma$.
\end{exer}

Recall that the blowing up of a $\bbQ$-valuative ideal $\gamma$ is determined by its Rees algebra $R=\cR_\gamma$, hence it coincides with the blowing up along the $\QQ$-ideal $\gamma_R$. Thus, the theory of weighted blowings up of normal varieties can be described entirely in terms of $\QQ$-ideals or Rees algebras. To extend this to non-normal varieties one should use non integrally closed Rees algebras.

\subsubsection{Pullbacks}
For any morphism of normal schemes $Y\to X$ and a $\bbQ$-ideal $I^{1/d}$ on $X$ one defines its pullback $(I\cO_Y)^{1/d}$ which we will also denote as $I^{1/d}\cO_Y$ for shortness. It is easy to see that this operation is well defined. Also, for any morphism of schemes $Y\to X$ and a Rees algebra $R=\oplus_d I_d$ on $X$ one defines the pullback $R\cO_Y=\oplus_d I_d\cO_Y$.

\subsubsection{Rees blowings up}
To any Rees algebra $R$ on a variety $X$ one associates the blowing up $X'=\cProj_X(R)$. It comes equipped with the exceptional divisor $E'$ whose invertible ideal will be denoted $I_{E'}$. Namely, for $a_dt^d\in R_d$ the restriction of $I_{E'}$ onto the corresponding chart $[\Spec(R[(a_dt^d)^{-1}])/\Gm]$ corresponds to the $\Gm$-equivariant ideal $(t^{-1})$, where $t^{-1}=a_dt^{d-1}/a_dt^d\in R[(a_dt^d)^{-1}]$.

\subsubsection{The universal property of weighted blowings up}
Weighted blowings up possess a certain universal property similar to the one satisfied by the usual blowings up but also more subtle. The proof is a direct computation, see \cite[Theorem~3.2.9]{Quek-Rydh}.

\begin{theor}\label{univth}
Let $X$ be a variety and $R$ a Rees algebra on $X$ with the weighted blowing up $f\:X'=Bl_\gamma(X)\to X$ and support $V=V(R)$. Then:

(0) $f$ is an isomorphism over $X\setminus V$.

(i) $I_d\cO_{X'}\subseteq I_{E'}^d$ with equality holding for any divisible enough $d$ (in fact, it suffices that $d$ is divisible by the weights of a set of generators).

(ii) If $h\:T\to X$ is a morphism such that $h^{-1}(X\setminus V)$ is schematically dense, then the category of factorizations of $h$ through $X'$ is equivalent to the set of invertible ideals $I_T$ such that $I_d\cO_{T}\subseteq I_{T}^d$ with equality holding for any divisible enough $d$.
\end{theor}

\subsubsection{Strict transforms}
If $Y\to X$ is a morphism of varieties and $R$ is a Rees algebra on $X$ with $V=V(R)$, then the {\em strict transform} of $Y$ with respect to the blowing up $X'=Bl_R(X)\to X$ is the schematic closure of the preimage of $X\setminus V$ in $Y\times_XX'$. As in the classical case, the universal property easily implies the following description of strict transforms, see \cite[Corollary~3.2.14]{Quek-Rydh}.

\begin{cor}\label{strictcor}
Let $Y'\to X'$ be the strict transform of a morphism $h\:Y\to X$ with respect to a Rees blowing up $X'=Bl_R(X)\to X$. Then $Y'=Bl_{R\cO_Y}(Y)$. In addition, $Y'=Y\times_XX'$ whenever $h$ is flat, and $h'\:Y'\to X'$ is a closed immersion whenever $h$ is a closed immersion.
\end{cor}

\subsubsection{$\bbQ$-regular centers}
Finally, we would like to define $\QQ$-regular centers in an arbitrary variety $Z$. For a tuple $\bff=(f_1\..f_n)$ of functions $f_i\in\Gamma(\cO_Z)$ and a tuple of positive rational numbers $\bfq=(q_1\..q_n)$ consider the Rees algebra $R_{\bff,\bfq}=\oplus_dI_d$ with each $I_d$ generated by all monomials $f_1^{l_1}\ldots f_n^{l_n}$ with $l_1/q_1+\ldots+ l_n/q_n\ge d$. By a {\em $\QQ$-regular center of multiplicity $\bfq=(q_1\..q_n)$} on $Z$ we mean a Rees algebra which is locally of the form $R_{\bff,\bfq}$ so that for any $z\in V(R)$ the images of $f_1\..f_n$ in $m_z/m_z^2$ are linearly independent and the image of the set $$\{f_1^{l_1}\ldots f_n^{l_n}|\ l_1/q_1+\ldots+l_n/q_n<1\}$$ in $\cO_z/I_1\cO_z$ is linearly independent.

By the {\em associated weighted blowing up} we mean the blowing up of the Rees algebra obtained by shifting the weights to $a^{-1}\bfq$ where $a\ge 1$ is the minimal natural number such that each $q_i/a$ is of the form $1/w_i$. Locally this algebra is generated by the elements $f_1t^{w_1}\..f_nt^{w_n}$.

These definitions are compatible with closed immersions into manifolds:

\begin{exer}\label{qregex}
Keep the above notation and fix a closed immersion $Z\into X$ with a smooth $X$.

(i) Show that $R$ extends to a $\QQ$-regular center on $X$ in the following sense: choose any family of parameters $t_1\..t_n$ on $X$ which restricts to $f_1\..f_n$ and consider the $\QQ$-regular center $\gamma=(t_1^{q_1}\..t_n^{q_n})$. Then $R=R_\gamma\cO_Z$. In particular, the Rees blowing up of $Z$ along $R$ is the strict transform of $Bl_\gamma(X)\to X$.

(ii) Moreover, show that geometrically speaking $V(\gamma)$ is contained in $Z$ in the sense that $\gamma\le v_\cI$ for $\cI=I_Z$.

(iii) Conversely, show that any $\QQ$-regular center $\gamma$ on $X$ such that $\gamma\le v_\cI$ restricts to a $\QQ$-regular center on $Z$.

(iv) Finally, show that if $Z$ is normal, then $R^\nor$ corresponds to the $\QQ$-ideal $(f_1^{q_1}\..f_n^{q_n})$.
\end{exer}

\subsection{Non-logarithmic weighted algorithms}
In this section we will outline the simplest dream principalization -- the one without boundary.

\subsubsection{Weighted framework}
The geometric objects are just DM stacks of finite type over $k$, and one uses the usual smoothness and sheaves of derivations $\cD_X=\Der_k(\cO_X,\cO_X)$. Admissible centers are $\bbQ$-regular centers, and such a center $\gamma$ is $\cI$-admissible if $\gamma\le v_I$ (which informally means that $\cI\subseteq\gamma$). Admissible blowings up are weighted blowings up associated with such centers.

\subsubsection{Weighted order}
To complete the framework it remains to classify the centers and introduce an invariant. For this we always order parameters so that the tuple of multiplicities is monotonically increasing: $q_1\le q_2\le\dots \le q_n$. Naturally, the invariant $\word(\gamma)$ of the center is defined to be the ordered tuple $\bfq=(q_1\..q_n)$, and we provide the set of invariants with the lexicographic order, where shorter sequences are larger, for example, $(2)>(2,2)$ (alternatively, one can finish each string with $\infty$, making the invariant more similar to the classical one).

\begin{exer}
Check that the invariant is monotonic: if $\gamma\le\gamma'$, then $\word(\gamma)\le\word(\gamma')$.
\end{exer}

For an arbitrary proper ideal $\cI$ we define the {\em weighted order} $$\word_X(\cI)=\max_{\gamma\le v_\cI}\word(\gamma),$$ where the maximum is over all $\cI$-admissible $\bbQ$-regular centers. The {\em weighted order at a point} is defined by taking the minimum over all neighborhoods: $$\word_x(\cI)=\min_{x\in U}\word_U(\cI|_U).$$ For completeness we also define the invariant in extreme cases: $\word(\cO_X)=(0)$ and $\word(\cI)=()$ if $V(\cI)$ contains a generic point of $X$.

\subsubsection{Weighted order reduction}
Now we can formulate the main properties of the weighted framework, which completely dictate what the algorithm is. In fact, it turns out that there is a unique candidate for the first weighted blowing up and already this blowing up reduces the weighted order.

\begin{theor}\label{wordth}
Let $X$ be a smooth DM stack of finite type over $k$, and let $\cI$ be an ideal on $X$.

(i) The weighted order $\word_X(\cI)=(q_1\..q_n)$ satisfies the following integrality condition: $q_1\in\bbN$ and $q_{i+1}\prod_{j=1}^i(q_j-1)!\in\bbN$ for $2\le i\le n$.

(ii) There exists a unique $\cI$-admissible center $\gamma$ such that $\word(\gamma)=\word_X(\cI)$ and $V(\gamma)$ contains all points $x\in X$ with $\word_x(\cI)=\word_X(\cI)$.

(iii) Let $\gamma$ be the center described by (ii), $X'=wBl_\gamma(X)$ the associated weighted blowing up (Definition \ref{Def:wBl}) and $\cI'=(\gamma\cO_{X'})^{-1}\cI\cO_{X'}$ the transform of $\cI$. Then $\word_{X'}(\cI')<\word_X(\cI)$.

(iv) The center depends smooth functorially on the pair $(X,\cI)$.
\end{theor}

\subsubsection{Weighted principalization}
Part (i) of the theorem implies that the set of invariants of $\word_X$ is a well-ordered subset of $\{0\}\cup(\cup_{n\in\bfN}\bbQ_{>0}^n)$, therefore iterating the weighted order reduction we arrive at the end to the only ideal with zero invariant -- the unit ideal. This yields the main result about weighted principalization.

\begin{theor}\label{dreamth}
There exists a principalization method $\cF$ which associates to any ideal $\cI$ on a smooth DM stack $X$ of finite type over $k$ a sequence of smooth weighted blowings up $\cF(\cI)\:X'=X_n\dashto X_0=X$ such that the following conditions are satisfied.

(i) The sequence is a principalization: $\cI_n=\cO_{X_n}$ and each $X_{i+1}\to X_i$ is $\cI_i$-admissible, where $\cI_0=\cI$ and $\cI_{i+1}$ is the transform of $\cI_i$ for $0\le i\le n-1$.

(ii) The sequence depends smooth functorially on the pair $(X,\cI)$: for any smooth morphism $Y\to X$ with $\cJ=\cI\cO_Y$ the sequence $\cF(\cJ)$ is obtained from the pullback sequence $\cF(\cI)\times_XY$ by omitting all empty weighted blowings up.

(iii) The method requires no history (I call this a dream method): each weighted blowing up $X_{i+1}\to X_i$ depends only on $(X_i,\cI_i)$, but not on $(X_j,\cI_j)$ with $j<i$.
\end{theor}

\subsubsection{Examples when the dream fails}\label{nodream}
In the classical setting no dream principalization algorithm exists. Usually one argues that blowing up a pinch point on the Whitney umbrella $Z=V(x^2-y^2z)\into\bfA^3=X$ creates another pinch point, hence without memory one goes into an endless loop. However, blowing up $X$ along $V(x,y)$ resolves the singularity, which makes the case not fully convincing. Here is W{\l}odarczyk's favorite example which avoids this. Let $X=\bfA^4=\Spec(k[x,y,z,t])$ and $Z=V(x^2-yzt)$, in particular, $G=S_3$ acts on it by permuting $y,z,t$. The singular locus of $Z$ is the union of the $y$, $z$ and $t$-axes, and the origin $O$ is the only $S_3$-invariant smooth center which contains $O$ and lies in the singular locus. As in the case of Whitney umbrella, blowing $O$ up creates another singularity of the same type (in fact, 3 singularities on different charts permuted by $S_3$). This shows that there is no memoryless smooth functorial principalization which only blows up smooth centers.

\begin{exer}
Show that the weighted order at $0$ is $(2,3,3,3)$, the dream algorithm blows up $(x^{1/3},y^{1/2},z^{1/2},t^{1/2})$ and the weighted order drops after this blowing up.
\end{exer}

\subsubsection{Weighted desingularization}
As in the classical case, if $Z\into X$ is a closed generically reduced substack of constant codimension, then the principalization of $\cI_Z\subset\cO_X$ blows up all components of the strict transform of $Z$ at the same blowing up $X_{n+1}\to X_n$. This implies that the strict transform $Z_n\into Z$ of $Z$ is smooth, hence $Z_n\to Z$ is a resolution. By the usual argument of uniqueness of minimal embeddings up to an \'etale correspondence this yields a smooth functorial desingularization method for locally-equidimensional and generically reduced stacks of finite type over $k$.

\subsubsection{Back to schemes}
The stacky structure can be removed by a destackification. As a result one obtains a new desingularization method, which is more efficient than the classical one, but does not possess new theoretical properties. However, we will now see that a minor modification leads to a so-called strong desingularization method.

\subsubsection{Strong weighted desingularization}
When $Z\into X$ is resolved by the classical principalization of $\cI_Z$ the $i$-th center $V_i$ lies in the preimage of $Z$, but it does not have to be contained in the strict transforms $Z_i\into X_i$ of $Z$, and the center $V_i|_{Z_i}=V_i\times_{X_i}Z_i$ of $Z_{i+1}\to Z_i$ is usually singular, see \cite[Example~8.2]{Bierstone-Milman-funct}. The same is true for the weighted principalization: it can happen that $V(\gamma_i)\nsubseteq Z_i$. However, if $X'=wBl_\gamma(X)\to X$ is the weighted order reduction of $\cI=I_Z$, $Z'\into X'$ is the strict transform of $Z$, and $\cI'$ is the transform of $\cI$, then $Z'$ is a closed subset of $V(\cI')$, and hence $\word_{X'}(I_{Z'})\le\word_{X'}(\cI')<\word_X(I_Z)$. Therefore, the more economical way to resolve $Z$ is to proceed at the second stage with $I_{Z'}$ instead of $\cI'$, etc. In such a way we do not achieve principalization of $I_Z$ but just successively modify the strict transform of $Z$. Moreover, not only the $i$-th center $\gamma_i$ is contained in the $i$-th strict transform $Z_i$ (see Exercise~\ref{qregex}), it depends only on the pair $(X_i,Z_i)$ and then the re-embedding principle implies that in fact it only depends on $Z_i$. This provides the following non-embedded dream algorithm.

\begin{theor}
There exists a desingularization method $\cR$ which associates to any locally equidimensional generically reduced DM stack $Z$ of finite type over $k$ a sequence of weighted blowings up with $\bbQ$-regular centers $\cR(Z)\:Z'=Z_n\dashto Z_0=Z$ such that the following conditions are satisfied.

(i) The sequence is a desingularization: $Z'$ is smooth.

(ii) The sequence depends smooth functorially on $Z$: for any smooth morphism $Y\to Z$ the sequence $\cR(Y)$ is obtained from the pullback sequence $\cR(Z)\times_ZY$ by omitting all trivial blowings up.

(iii) The method requires no history: each weighted blowing up $Z_{i+1}\to Z_i$ depends only on $Z_i$.
\end{theor}

\subsubsection{The geometric interpretation}
The algorithm iteratively repeats the same base step: chose a canonical $\QQ$-regular center $R_Z$ depending only on $Z$ and blow it up. This $R_Z$ is in fact the unique $\QQ$-regular center whose multiplicity $(q_1\..q_n)$ is maximal possible and which contains all points of $Z$ with this invariant. Furthermore, this center witnesses the non-smoothness of $Z$ by the fact that it cannot be extended to a thicker center with larger multiplicities. Blowing up this obstacle we kill it, and it turns out that the blown up variety has a smaller maximal non-smoothness obstacle.

\begin{exam}
Take $Z=\Spec(k[t_1\..t_n]/(t_1^{q_1}+\ldots+t_n^{q_n}))$ with $2\le q_1\le q_2\le \dots \le q_n$. The maximal center in this case is induced by the usual ideal $I=(t_1^{q_1}\..t_n^{q_n})$ and its normalized Rees algebra $R=\oplus_d I_d$, where $I_d$ is generated by monomials $t_1^{l_1}\dots t_n^{l_n}$ with $\sum_j l_j/q_j\ge d$ and $I_1=I^\nor$. Clearly, $V(I_1)\into Z$ and the associated weighted blowing up is a stack theoretic refinement of the blowing up along $I_1$. Let also $I_{>1}$ be the ideal generated by the monomials with $\sum_j l_j/q_j>1$. Then the non-smoothness of $Z$ is detected by the fact that the closed immersion $V(I_1)\into Z$ does not extend to a square-zero thickening $V(I_{>1})\into Z$ because $t_1^{q_1}+\ldots+t_n^{q_n}\notin I_{>1}$.
\end{exam}

\subsubsection{Relation to the theory of maximal contact}
Now let us briefly discuss the justification. As we saw, everything follows easily from Theorem~\ref{wordth}. At first glance one might expect that such a result, once correct, should be provable in a few ways, including rather direct ones. The following example by W{\l}odarczyk shows that one should be more careful because such a $\bbQ$-regular center does not exist in positive characteristic, hence an argument should be subtle enough and use the zero characteristic assumption.

\begin{exam}
The example is again... a Whitney umbrella $\cI=(x^2+y^2z)$ on $X=\Spec(k[x,y,z])$, but this time when $k$ is perfect of characteristic 2. As in characteristic zero, the weighted order at the origin is $(2,3,3)$, but this time the whole singular locus $C=V(x,y)$ consists of pinch points because there is an automorphism $(x,y,z)\mapsto (x+ty,y,z+t^2)$ which translates $C$. Clearly, the local centers exist in this case at any point of $C$, but they do not glue to a global one. Moreover, at the generic point $\eta\in C$ the invariant is $(2,2)$ and the trouble hides in the fact that the unit $z\in\cO_\eta$ is not a square. Note also that the invariant at the geometric generic point $\bar \eta$, viewed as a point of $X_{\bar{k(z)}}$, is again $(2,3,3)$. All in all, we see that no nice theory of weighted centers seems to be possible, and, maybe, imperfect residue fields and units whose value at a point is not a $p$-th power should somehow be taken into account. My personal expectation is that the formalism of characteristic exponents might be useful but insufficient.
\end{exam}

Given our current knowledge, when working on \cite{ATW-weighted} and \cite{Quek} the fastest way was to use the classical maximal contact theory both to construct such a center and prove its uniqueness. We do expect that a more direct argument should exist, and looking for it is one of the future projects.

\subsubsection{Construction of the center}\label{centersec}
By the theory of maximal contact, working locally at a point $x\in X$ we define an iterated sequence of maximal contacts to iterated coefficient ideals. Formally speaking, we obtain a neighborhood $H_0$ of $x$, a partial sequence of regular parameters $t_1\..t_n$ and a sequence of ideals $\cI_i$ on $H_i=V(t_1\..t_i)$ of orders $d_i=\ord_x(\cI_i)$ such that $\cI_0=\cI|_{H_0}$, each $t_i|_{H_i}$ is a maximal contact to $\cI_i$, $d_i=\ord_{H_i}(\cI_i)$ for any $0\le i\le n$, in the notation of \S\ref{coeffsec} we have that $$\cI_{i+1}=\sum_{a=0}^{d_i-1}(\cD_X^{\le a}(\cI_i),d_i-a)\mid_{H_{i+1}}$$ is the restriction of the coefficient ideal of $\cI_i$ onto $H_{i+1}$ for any $0\le i\le n-1$, and $\cI_n=0$. In particular, $H_0$ is small enough so that each order accepts its maximum at $x$.

Once the above choices are done, setting $q_i=d_{i-1}/\prod_{j=0}^{i-2}(d_j-1)!$ (in particular, $q_1=d_0$) one obtains that $$\word_{H_0}(\cI)=\word_x(\cI)=(q_1\..q_n)$$ and the center on $H_0$ is $\gamma=(t_1^{q_1}\..t_n^{q_n})$. Note that this datum is very close to the one produced by the classical algorithm: the first blowing up of the non-weighted principalization is along the center $(t_1\..t_n)$ and $(q_1,0;\dots ;q_n,0;\infty)$ is the invariant of Bierstone-Milman for this blowing up.

Certainly, the above sketch only provides a construction and explains the integrality condition and smooth functoriality, but one has to prove that, indeed, $\gamma$ is $\cI$-admissible, of maximal weighted order, and unique. All this is done using standard results from the theory of maximal contact (e.g. \cite[Lemma~4.4.1]{ATW-weighted}) and homogenization (needed to pass from one maximal contact to another). We refer the reader to the proof of \cite[Theorem~5.1.1]{ATW-weighted} for details. These ideas are also recalled in chapter \cite{weightedexpo}, and in chapter \cite{weighted} the weighted principalization is constructed with all proofs and details in the context of cobordant blowings up, when instead of stack-theoretic blowings up one considers the canonical presentation by a torus quotient. Essentially this means that the orbifolds we consider are replaced by their torus equivariant representable covers, hence, being smooth local, the arguments and proofs are essentially the same.

\subsection{Logarithmic weighted algorithms}
The simplest dream algorithm does not use any log structures. In particular, one obtains a principalization of an ideal, which is just an invertible ideal, with an arbitrarily singular support. In addition, one does not obtain resolution of divisors by snc ones. There is a logarithmic refinement of the dream algorithm which addresses these issues. It should certainly exist in the relative setting, but so far the theory was only developed by Quek in \cite{Quek} in the absolute case -- the case of log varieties. This method carefully combines the logarithmic and weighted settings and many things have to be checked, but the main line is to imitate the dream algorithm when working with log varieties and using arbitrary weighted submonomial centers. The reader already saw these ideas, so we will only explain the main novelty.

\subsubsection{The framework}
Naturally, one works with logarithmic DM stacks of finite type over a field, logarithmic derivations and logarithmic smoothness. Admissible centers on a log orbifold $X$ are valuative $\bbQ$-ideals which are locally at $x\in X$ of the form $\gamma=(t_1^{q_1}\..t_n^{q_n},u^{p_1}\..u^{p_r})$, where $t_1\..t_m$ is a regular family of parameters at $x$, and $p_1\..p_r\in(\oM_x)_\bbQ$ are rational monomials. Admissible blowings up are weighted blowings up of such centers as we define below. We will use the shorter notation $\gamma=(t_1^{q_1}\..t_n^{q_n},\cN_\gamma)$, where $\cN_\gamma=(u^{p_1}\..u^{p_r})$ is a Kummer monomial ideal called the {\em monomial type} of $\gamma$. As in the non-logarithmic case, one has a certain freedom in choosing the regular parameters $t_1\..t_n$ and the other data is fixed:

\begin{exer}
Show that both the {\em multiplicity} $(q_1\..q_n)$ of $\gamma$ and the monomial type are uniquely determined by $\gamma$.
\end{exer}

\subsubsection{Weighted submonomial blowings up}
As in the non-logarithmic case, by a {\em weighted submonomial center} we mean an admissible center with $q_i=1/w_i$ for natural weights $w_i\in\bbN$. For an admissible center $\gamma=(t_1^{q_1}\..t_n^{q_n},\cN_\gamma)$ with $q_i=a_i/b_i$ and $a=\lcm(a_1\..a_n)$ the center $a^{-1}\gamma$ is weighted and we define the {\em weighted blowing up along $\gamma$} to be $X'=wBl_\gamma(X)=Bl_{a^{-1}\gamma}(X)$. For any ideal $\cI$ on $X$ such that $\gamma$ is $\cI$-admissible the transform $\cI'=(\gamma\cO_{X'})^{-1}\cI\cO_{X'}$ is defined. Again, a direct chart computation shows that a weighted submonomial blowing up of a log orbifold $X$ outputs a log orbifold $X'$, see \cite[Lemmas~4.1 and 4.2]{Quek}.

\subsubsection{Weighted log order and monomial type}
For an admissible center $\gamma=(t_1^{q_1}\..t_n^{q_n},\cN_\gamma)$ set $\wlogord_X(\gamma)=(q_1\..q_n)$ if $\cN_\gamma=\emptyset$ (i.e. there is no monomial part) and $\wlogord_X(\gamma)=(q_1\..q_n,\infty)$ if $\cN_\gamma$ is non-empty. This invariant and the monomial type provide the following monotonicity: if $\gamma<\gamma'$, then $\wlogord_X(\gamma)\le\wlogord_X(\gamma')$ and in the case of equality one has that $\cN_\gamma\supsetneq\cN_{\gamma'}$

For an arbitrary ideal $\cI$ on $X$ we define the {\em weighted log order} by $$\wlogord_X(\cI)=\max_{\gamma\le v_\cI}\wlogord(\gamma).$$ In addition, $\wlogord(\cO_X)=(0)$ and $\wlogord=()$ if $V(\cI)$ contains a generic point. Furthermore, we define the {\em monomial type} $\cN_\cI$ of $\cI$ to be empty if $\wlogord(\cO_X)$ does not end with $\infty$ and to be the minimal (with respect to inclusion) monomial type of an $\cI$-admissible $\gamma$ such that $\wlogord_X(\cI)=\wlogord(\gamma)$.

\subsubsection{Weighted log order reduction}
Now we can formulate the main theorem about weighted submonomial centers. The only new tool when comparing it to the non-logarithmic analogue is that one should also pay attention to the monomial type of an ideal.

\begin{theor}\label{wlogordth}
Let $X$ be a log smooth DM stack of finite type over $k$, and let $\cI$ be an ideal on $X$.

(i) The weighted log order $\wlogord_X(\cI)=(q_1\..q_n,*)$ and the monomial type $\cN_\cI$ are well defined and the following integrality condition is satisfied: $q_1\in\bbN$ and $q_{i+1}\prod_{j=1}^i(q_j-1)!\in\bbN$ for $2\le i\le n$, and $*$ is either $\emptyset$ or $\infty$.

(ii) There exists a unique $\cI$-admissible center $\gamma$ such that for each point $x\in X$ with $\wlogord_x(\cI)=\wlogord_X(\cI)$ one has that $\wlogord_x(\gamma)=\wlogord_x(\cI)$ and $\cN_x(\gamma)=\cN_x(\cI)$.

(iii) Let $\gamma$ be the center described by (ii), $X'=wBl_\gamma(X)$ the associated weighted blowing up and $\cI'=(\gamma\cO_{X'})^{-1}\cI\cO_{X'}$ the transform of $\cI$. Then $\wlogord_{X'}(\cI')<\wlogord_X(\cI)$.

(iv) The center depends log smooth functorially on the pair $(X,\cI)$.
\end{theor}

\subsubsection{Weighted log principalization}
As in the non-logarithmic case iterating the weighted log order reductions one obtains a dream algorithm for log principalization.

\begin{theor}\label{dreamlogth}
There exists a log principalization method $\cF$ which associates to any ideal $\cI$ on a log smooth DM stack $X$ of finite type over $k$ a sequence of weighted blowings up $\cF(\cI)\:X'=X_n\dashto X_0=X$ such that the following conditions are satisfied.

(i) The sequence is a log principalization: $\cI_n=\cO_{X_n}$ and each $X_{i+1}\to X_i$ is $\cI_i$-admissible, where $\cI_0=\cI$ and $\cI_{i+1}$ is the transform of $\cI_i$ for $0\le i\le n-1$.

(ii) The sequence depends log smooth functorially on the pair $(X,\cI)$: for any log smooth morphism $Y\to X$ with $\cJ=\cI\cO_Y$ the sequence $\cF(\cJ)$ is obtained from the pullback sequence $\cF(\cI)\times_XY$ by omitting all trivial weighted blowings up.

(iii) The method requires no history: each weighted blowing up $X_{i+1}\to X_i$ depends only on $(X_i,\cI_i)$.
\end{theor}

\subsubsection{Strong logarithmic resolution}
The same argument as in the non-logarithmic case shows that applying weighted log order reduction to the strict transforms of $Z\into X$ one obtains the following non-embedded dream log desingularization algorithm

\begin{theor}\label{logdesingth}
There exists a desingularization method $\cR$ which associates to any locally equidimensional and generically log smooth log DM stack $Z$ of finite type over $k$ a sequence of weighted blowings up $\cR(Z)\:Z'=Z_n\dashto Z_0=Z$ such that the following conditions are satisfied.

(i) The sequence is a log desingularization: $Z'$ is log smooth.

(ii) The sequence depends log smooth functorially on $Z$: for any smooth morphism $Y\to Z$ the sequence $\cR(Y)$ is obtained from the pullback sequence $\cR(Z)\times_ZY$ by omitting all trivial blowings up.

(iii) The method requires no history: each weighted blowing up $Z_{i+1}\to Z_i$ depends only on $Z_i$.
\end{theor}

\subsubsection{Justification}
As in the non-logarithmic case described in \S\ref{centersec}, one uses the logarithmic theory of maximal contact (and the sheaf of logarithmic derivations $\cD_X$). Locally at a point $x\in X$ one defines an iterated sequence of logarithmic maximal contacts to the iterated coefficient ideals, obtaining a neighborhood $H_0$ of $x$, a partial sequence of regular parameters $t_1\..t_n$ and a sequence of ideals $\cI_i$ on $H_i=V(t_1\..t_i)$ of orders $d_i=\ord_x(\cI_i)$ such that $\cI_0=\cI|_H$, each $t_i|_{H_i}$ is a logarithmic maximal contact to $\cI_i$, $d_i=\ord_{H_i}(\cI_i)$ for any $0\le i\le n-1$, we have that $$\cI_{i+1}=\sum_{a=0}^{d_i-1}(\cD_X^{\le a}(\cI_i),d_i-a)|_{H_{i+1}}$$ is the restriction of the coefficient ideal of $\cI_i$ onto $H_{i+1}$ for any $0\le i\le n-1$, and $\logord(\cI_n)=\infty$. Then the monomial type $\cN_{\cI_0}$ is defined as the monomial ideal generated by the monomials from the monomial hull $\cM(\cI_n)$ of $\cI_n$.

Once these choices are done, the invariant and the center are read off as follows: $$\wlogord_{H_0}(\cI)=\word_x(\cI)=(q_1\..q_n, *),$$ where $q_i=d_{i-1}/\prod_{j=0}^{i-2}(d_j-1)!$ (in particular, $q_1=d_0$), $*=\emptyset$ if $\cN_{\cI_0}=\emptyset$ and $*=\infty$ otherwise, and the center on $H_0$ is $\gamma=(t_1^{q_1}\..t_n^{q_n},\cN_{\cI_0})$. The justification is analogous, but it works with the logarithmic theory of maximal contact and also addresses the monomial type in the end. Once again we can relate this datum to the first blowing up of the non-weighted logarithmic principalization: the latter blows up the center $(t_1\..t_n,\cN_{\cI_0}^{1/q_1})$ and the invariant at this step is $(q_1\..q_n,*)$, see Remark~\ref{invrem}.

\section{Resolution for quasi-excellent schemes and other categories}\label{othersec}
Throughout the notes we only worked with schemes and stacks of finite type over $k$. In this section we will discuss resolution in the wider context of quasi-excellent schemes and other categories, such as formal schemes, and complex or non-archimedean analytic spaces. We only aim to provide a very brief survey of some tools and literature. In addition, we will discuss directions which were not fully verified so far, but we expect the results to be true. The two main methods we will discuss are via extending the framework, and via black box reduction to (appropriate) quasi-excellent schemes. Note that a rather detailed survey on the second approach can be found in \cite{Temkin-survey} and we will refer to it from time to time. Unfortunately \cite{Temkin-survey} reflects our knowledge from more than ten years ago, and hence considers only the questions of extending the classical methods to wider settings.

\subsection{Reduction to quasi-excellent schemes}
In this subsection we will show that desingularization of objects that look completely non-algebraizable nevertheless follows from desingularization of quasi-excellent schemes. Thus, quite surprisingly, resolution of singularities seems to be a purely algebraic phenomenon. The key observation is that the latter desingularization should be functorial with respect to all {\em regular} morphisms. Also we will see in \S\ref{qesec} that within the class of qe schemes there exist various bootstraps which allow to essentially reduce the problem to resolving algebraic varieties (these methods were used in \cite{Temkin} and \cite{Temkin-qe}). We start with specifying the classes of geometric objects whose desingularization we will discuss, and then we will discuss the reduction to qe schemes.

\subsubsection{Analytic spaces}
Resolution of singularities makes sense in various categories of geometric objects, where appropriate notions of smoothness and blowings up make sense. Probably, the most natural case is that of complex analytic spaces, and Hironaka himself outlined how his proof should be modified in that case. In \cite{Bierstone-Milman} Bierstone-Milman tried to axiomatize various situations in which their desingularization algorithm applies. In particular, they claimed that it applies to algebraic spaces of finite type over a field, analytic spaces -- complex, real and non-archimedean, and certain quasi-analytic objects. This is definitely true with our current knowledge, though I cannot track the original argument in some cases (one cannot use Zariski topology when working with algebraic spaces; also in Berkovich geometry sheaves of differentials are defined with respect to the $G$-topology rather than the usual one).

\subsubsection{Quasi-excellent schemes}
Next, let us discuss the case of schemes not of finite type over a field. Hironaka used formal completion in his original work \cite{Hironaka}, so he had to prove his results for the much wider class of schemes of finite type over local rings $A$ such that the homomorphism $A\to\hatA$ is regular. Recall that a morphism $Y\to X$ of noetherian schemes is {\em regular} if it is flat and has geometrically regular fibers. For morphisms of finite type this notion is equivalent to smoothness, so regularity is the natural generalization of smoothness in the case of large morphisms. Soon after this Grothendieck defined a more general class of {\em quasi-excellent} or {\em qe} rings\footnote{The terminology was introduced later. Grothendieck only used the notion of excellent schemes to denote universally catenary quasi-excellent ones.} by two conditions: $X$ is qe if the following two conditions are satisfied:
\begin{itemize}
\item[(N)] After Nagata: for any integral $Y$ of finite type over $X$ the singular locus of $Y$ is open.
\item[(G)] After Grothendieck (though we saw, that historically it should have been named after Hironaka): for any $x\in X$ the completion homomorphism $\cO_{X,x}\to\hatcO_{X,x}$ is regular.
\end{itemize}
Some discussion and examples related to these properties and their relation to resolution can be found in \cite[\S2.3]{Temkin-survey}.

The main observation of Grothendieck was \cite[${\rm IV}_2$, Proposition~7.9.5]{ega} claiming that even a weakest consistent resolution theory is possible only for qe schemes: if any integral $X$-scheme of finite type possesses a desingularization, then $X$ is qe. Already Grothendieck expressed a hope that any integral qe scheme possesses a resolution and this is widely believed to be true. In characteristic zero, this was proved in \cite{Temkin}, and stronger versions are in \cite{Temkin-qe} and \cite{Temkin-embedded}.\footnote{In general only resolution of qe threefolds is known, see \cite{CP}.} Thus, when working with schemes we will always restrict to the quasi-excellent schemes of characteristic zero. By \cite[${\rm IV}_2$, Proposition~7.8.6(i)]{ega} if $X$ is excellent, then any $X$-scheme of finite type is excellent too, and the same argument applies to qe schemes as well.

\subsubsection{Stacks}
A stack is called {\em qe} if it has a smooth cover by a qe scheme. We tacitly used in these notes that any smooth-functorial algorithm automatically extends from schemes of finite type over a field to stacks of finite type over a field. The same argument applies to qe schemes and stacks. Also, any method which involves stack-theoretic blowings up (e.g. logarithmic or weighted principalization) should be developed in the generality of qe stacks right ahead.

\subsubsection{Formal schemes}\label{formalsec}
We say that a noetherian formal scheme $X=(X,\cO_X)$ is a {\em formal variety} if its closed fiber $X_s=(X,\cO_X/I)$, where $I$ is the maximal ideal of definition, is an algebraic variety. Naturally, when working with more general formal schemes we will have to impose a quasi-excellence restriction. A formal scheme $X$ is {\em qe} if its closed fiber is a qe scheme. This definition is really useful because of the following theorem of Gabber, see \cite{formalqe}: an $I$-adic notherian ring $A$ is qe if and only if $A/I$ is qe. This is a deep fact whose proof uses a weak local resolution of singularities. In particular, it implies that quasi-excellence is preserved by formal completions, and hence also by formal localizations, passing from $A$ to $A\llbracket t\rrbracket$, etc. Part of the difficulty stems from the fact that the G-property along is not satisfied even by passing from $A$ to $A\llbracket t\rrbracket$.

Gabber's theorem implies that if $X$ is qe and $Y$ is of topologically finite type over $Y$, then $Y$ is qe. Before the theorem was available one had to work with clumsier and ad hoc definitions, e.g. see \cite[\S3.1]{Temkin}. An important property of qe formal schemes is that formal localization homomorphisms $A\to A_{\{f\}}$ are regular. The latter condition is extremely important, since it allows to extend to qe formal schemes notions from the theory of schemes which are local with respect to the topology of regular covers. In particular, one can define the notions of singular loci, see \cite[\S3.1]{Temkin}, and regular morphisms. In particular, $Y\to X$ is regular if it is covered by morphisms $\Spf(B_{ij})\to\Spf(A_i)$ with regular homomorphisms $A_i\to B_{ij}$. The general principle is very simple -- use the usual definiiton which works with rings in the affine case and use compatibility of the notion with regular morphisms to globalize.

Blowings up of formal schemes are defined in a similar fashion, see \cite[\S2.1]{Temkin}. If $X=\Spf(A)$ with an $I$-adic $A$, then one simply blows up $\Spec(A)$ and $I$-adically completes, and in general one glues the local construction using that blowings up are compatible with flat morphisms. In particular, this construction is well defined for arbitrary noetherian formal schemes.

\subsubsection{Geometric spaces}
For concreteness we will work with one of the following spaces:
\begin{itemize}
\item[(1)] Quasi-excellent schemes.
\item[(2)] Quasi-excellent formal schemes.
\item[(3)] Complex analytic spaces.
\item[(4)] Non-archimedean analytic spaces introduced by Berkovich, see \cite{berihes}.
\end{itemize}

Certainly there are other possibilities - rigid or adic analytic spaces, Nash spaces, etc.

\subsubsection{Affinoid spaces}
In \S\ref{formalsec} we saw how constructions from the scheme theory can be transferred to formal schemes. It turns out that one can study analytic spaces similarly. This is not surprising in the non-archimedean setting, since Berkovich spaces are pasted from affinoid spaces $\cM(\cA)$, which are spectra of $k$-affinoid algebras. Affinoid algebras are excellent and for any subdomain embedding $\cM(\cB)\into\cM(\cA)$ the homomorphism $\cA\to\cB$ is regular by results of Ducros, see \cite{Ducros-excellent} or \cite[\S\S6.2--6.3]{factorqe}.

Nevertheless, an analogous theory also exists for complex analytic spaces. Each such space $X$ is covered by so-called semi-algebraic Stein compacts $X_i$ (for example, closed subdomains of polydiscs), the ring $A_i=\cO_X(X_i)$ of overconvergent functions on $X_i$ is excellent and controls $X_i$ well enough. Again, if $X_j\subseteq X_i$, then the homomorphism $A_j\to A_i$ is regular. In detail, these claims are checked in \cite[Appendix~B]{factorqe}, and a relative GAGA over Stein compacts is established in \cite[Appendix~C]{factorqe}. In particular, GAGA implies that, as in the formal case, complex blowings up are obtained by analytifying the algebraic ones.

\subsubsection{Reduction to qe schemes}
The general principle is very simple: any desingularization, principalization, etc. method $\calF$, which is functorial with respect to arbitrary regular morphisms, automatically extends to analytic spaces and qe formal schemes. The construction is as follows:
\begin{itemize}
\item[(0)] Cover the space $X$ with affinoid/affine subspaces $X_i$ corresponding to  qe rings $A_i$. If $X$ comes equipped with an ideal $\cI\subseteq\cO_X$ one also obtains an ideal $I_i\subseteq A_i$, etc. Also, cover each $X_i\cap X_j$ with affinoid/affine subspaces $X_{ijk}$.
\item[(1)] Consider the blowing up tower $\cF(\Spec(A_i))$ of $\Spec(A_i)$ and apply the analytification/completion functor to produce a blowing up tower of $X_i$. This step uses the relative GAGA as we need to analytify ideals on blowing up in the tower.
\item[(2)] Show that the obtained towers $\cF(X_i)$ are compatible on intersections because $\cF$ respects the regular morphisms $\Spec(X_{ijk})\to\Spec(X_i)$.
\end{itemize}

In each case some minor details should be spelled out. In particular, this reduction scheme was used in \cite{Temkin-qe}, \cite{Temkin-embedded}, and \cite[\S6]{factorqe}, but it is the latter reference where all details were carefully spelled out (in the case of the factorization functor), a relative GAGA for complex spaces was constructed, etc.

\subsubsection{Non-compact objects}
Finally, we note that functorial methods can be also extended to non-compact spaces, including locally noetherian qe schemes. Clearly, the only way to do this is to cover $X$ by affine subspaces $X_i$ and glue the blowing up sequences for different $X_i$. The technical complication is that the resulting sequence can be infinite; for example, this happens when $X=\coprod_{i=1}^\infty X_i$ and the resolution of $X_i$ takes $i$ steps. However, one can naturally define infinite ordered sequences of blowing up (called hypersequences in \cite[\S5.3]{Temkin-qe}) with the following local finiteness condition: over any compact subspace of $X$ the sequence contains only finitely many non-trivial blowings up. For such a sequence a composition is well defined, and gluing local resolutions of the subspaces $X_i$ one obtains a hypersequence of blowings up of $X$ in general, see \cite[\S5.3]{Temkin-qe}.

\subsection{Extending the framework}\label{extsec}
The second method to extend algorithms developed for algebraic varieties is to directly adopt it to another category. For example, this is what Bierstone and Milman did for the classical method and the categories of analytic spaces. This line of research is very natural, as it just explores in which generality the methods work. Nevertheless, it seems that it is almost not presented in the literature, so I will just express my expectations. In brief, absolutely in line with the philosophy of frameworks I expect that once the framework of the method extends to a wider setting, the method extends as well. So, one should construct appropriate blowings up, generalized ideals and derivation theory.
In case of non-embedded resolution one should also worry that embedding into a smooth space exists. Note also that the theories of orbifolds and logarithmic analytic spaces are folklore and were partially developed in the literature.  Now, let us discuss case by case.

\subsubsection{Analytic spaces}
I expect that, excluding resolution of morphisms, all methods we discussed in these notes extend to complex and non-archimedean analytic spaces. One should use the usual analytic differentials and derivations. Perhaps the shortest way to introduce weighted blowings up is via $\QQ$-ideals and Rees algebras. I do not expect any serious complication.

As for principalization on relative orbifolds and resolution of morphisms of analytic spaces, there is one major obstacle -- the monomialization theorem, see \S\ref{monomsec}. For example, let $B=\bfP^2_\CC$, $B_0=\bfA^2_\CC$ and $Z=V(x-e^y)$ is a curve in $B_0$ which cannot be extended to $B$. Let $X=\bfA^1_{B_0}=\bfA^3_\CC$ with coordinates $x,y,z$ and $\cI$ the ideal on $X$ given by $(x-e^y,z)$. Then principalization of $\cI$ over $B_0$ goes by blowing up $V$ in $B_0$, which increases the log structure by $u^p=x-e^y$ and then blowing up the ideal $(x,u^p)$. However, once trying to principalize $\cI$ on $X\to B$ one is stuck.

I expect that except the monomialization the whole algorithm works fine. As for monomialization there are two hopes/questions which, at the very least, do not contradict any example we know: Does the monomialization theorem holds true when $X\to B$ is proper? Does the monomialization theorem holds for an arbitrary smooth morphism $X\to B$ of smooth analytic spaces if one allows base changes $B'\to B$ of the following more general form: $B'\to B$ is a cover for the topology generated by modifications and open covers?

\subsubsection{Schemes with enough derivations}
Arbitrary qe schemes $X$ may have too nasty absolute derivation theory. Not only, the sheaf $\cD_X=\Der_\QQ(\cO_X,\cO_X)$ does not have to be quasi-coherent, already for $X=\Spec(R)$ with $R$ a DVR it can happen that the sheaf $\cD_X$ has a zero stalk at the closed point, see \cite[Example~2.3.5(ii)]{Temkin-survey}. Certainly, all methods of these notes cannot apply to schemes with such derivation theory, so one should consider only qe schemes whose derivation theory is reach enough.

The following definition was suggested already in \cite[Remark~1.3.1(iii)]{Temkin-qe}: a scheme $X$ {\em has enough derivations} if for any point $x\in X$ all elements of the cotangent space $m_x/m_x^2$ are distinguished by elements of $\cD_{X,x}$. In other words, the homomorphism $\cD_{X,x}\to(m_x/m_x^2)^*$ to the tangent space is onto. Note that since $\cD_X$ is not quasi-coherent, $\cD_{X,x}$ can be strictly smaller than $\Der(\cO_{X,x},\cO_{X,x})$ so this condition asserts that there exist enough derivations in a neighborhood of $x$. It can be shown that schemes with enough derivations are qe and their class is closed under passage to schemes of finite type. I expect that non-logarithmic principalization methods -- classical and weighted, work for general schemes with enough derivations (in the classical case this was conjectured in \cite{Temkin-qe}) and are functorial with respect to arbitrary regular morphisms. Similarly, I expect that the non-embedded methods work for schemes which are \'etale-locally embeddable into regular schemes with enough derivations.

I expect that the same holds true for logarithmic methods and log schemes with enough log derivations, where the latter means that log derivations distinguish both regular and monoidal parameters. In the non-weighted case this was proved in \cite{ATW-relative} as the particular case of \cite[Theorem~1.2.6]{ATW-relative} when the target $B$ is just $\Spec(\QQ)$. In fact, \cite{ATW-relative} is the only paper I am aware of, where the notion of enough derivations was seriously explored. It even studied resolution of morphisms $X\to B$ with enough relative derivations, and showed that it works whenever $\dim(B)\le 1$. In the case of a higher dimensional $B$ a much stronger restriction on derivations is needed -- the so-called abundance of derivations, which also takes into account derivations in ``constant directions'' which distinguish elements of a transcendence basis of $k(x)/k(b)$. This condition cannot be borrowed to analytic spaces, and this explains why our monomialization theorem does not extend to analytic spaces when $\dim(B)>1$.

Finally, we note that, beyond algebraic varieties, the most important case of schemes with enough derivations are schemes of finite type over noetherian complete local rings. We will later see, why this class (or its certain subclasses) are so important.

\subsubsection{Formal schemes}
The treatment of formal schemes should be analogous to schemes, namely, I expect that the methods extend to formal schemes with enough (log) derivations. The most important case is that of formal varieties.

\subsection{Desingularization for quasi-excellent schemes}\label{qesec}
Finally, let us describe the situation with arbitrary qe (formal) schemes. The tools we have described so far do not suffice to establish resolution for them, and it seems that the only way is to use a certain descent from the formal completion via the (G) property. Thus, at first stage one should construct a desingularization or principalization method $\cF_\cC$ on the class $\cC$ of (formal) schemes of finite type over complete local rings or its suitable subclasses (for example, by the method of \S\ref{extsec}), and at the second step one applies descent. The classical solution uses descent of open ideals, and necessarily constructs a new method $\cF$, which is more complicated even on the schemes from $\cC$. This method is called localization or induction on codimension, it goes back to Hironaka's original paper, and (due to my ignorance) it was re-invented in \cite{Temkin}. A natural alternative would be to try to descent more general ideals, since this would just extend $\cF_\cC$ from $\cC$ to the class of all qe schemes. The only known result in this direction is McQuillan's proof that weighted centers indeed satisfy this descent, and hence the weighted resolution and principalization algorithms extend to arbitrary qe schemes, see \cite[Section~VII]{McQuillan}.

\subsubsection{Induction on codimension}
The method is very robust and applies to all algorithms, see \cite[Chapter~IV, \S1]{Hironaka} and \cite{Temkin-survey} (see also \cite[\S4]{Temkin-qe} and \cite[\S3.4]{Temkin-embedded}). In addition, this method is also used in the proof of the monomialization theorem. However, for the sake of concreteness we will only illustrate the idea on the case of principalization.

Assume first that $X=\Spec(A)$ is a local regular qe scheme and $I$ is supported at the closed point $s\in X$. Then the formal completion $\hatX=\Spec(\hatA)$ at the maximal ideal $m$ is a regular scheme and $\hatI=I\hatA$ is supported at the closed point $\hats\in\hatX$. The principalization $\hatX_n\dashto\hatX_0=\hatX$ of $\hatI$, which exists by step one, only blows up centers contained in the fiber over $\hats$. They correspond to ideals $\hatJ_i\subseteq\cO_{\hatX_i}$ open in the $\hatm$-adic topology, hence all these ideals algebraize and the tower is obtained by $m$-adic completion of a tower $X_n\dashto X_0=X$ of blowings up with centers $J_i\subseteq\cO_{X_i}$. It then easily follows by descent that the latter sequence principalizes $I$.

In the general case one proceeds by induction on codimension of the image of non-principalized locus in $X$. First, one considers the finite set of points $x_1\..x_n\in T_0=V(I)$ of minimal codimension, principalizes the pullback of $I$ to $\coprod_{i=1}^n\Spec(\cO_{X,x_i})$ by the method from the previous paragraph and then just blows up the schematic closure of the centers, obtaining a blowing up sequence $X'\dashto X$ which principalizes $I$ over $x_1\..x_n$. The centers do not have to be smooth over specializations of $x_1\..x_n$, but by induction assumption we can assume that principalization (and hence also resolution) has been already constructed for them. Thus, one simply resolves each center before blowing it up. This inserts intermediate blowings up in the sequence $X'\dashto X$ and the algorithm becomes more complicated, but this does not affect the situation over $x_1\..x_n$. At the next step one considers the image $T_1\subset X$ of $V(I')$, where $I'\subseteq\cO_{X'}$ is the transform of $I$. It is a closed subset of $T_0\setminus\{x_1\..x_n\}$. Choose points $y_1\..y_m\in T_1$ of minimal codimension, find a sequence $X''\dashto X'$ which principalizes $I'$ over the preimages of $y_1\..y_m$, etc.

\begin{rem}
(i) As an input for the induction on codimension scheme it suffices to take a principalization method $\cF_\cC$ which applies to regular schemes $X$ of finite type over the spectrum $\Spec(R)$ of a complete local ring and ideals $I\subset\cO_X$ such that $V(I)$ is contained in the preimage of the closed point. This is a serious restriction, which allows to obtain $\cF_\cC$ via a black box construction from principalization of varieties. Indeed, such a scheme can locally be realized as a completion of a variety and $I$ algebraizes to an ideal on the variety, so one can just pullback principalization on the algebraization. The subtle point in this method is independence of algebraization, which is based on a theorem of Elkik, see \cite[\S4]{Temkin-survey}. This method was implemented in \cite[\S3]{Temkin-qe}, though nowadays I think that extending the framework to schemes with enough derivations would be a better solution since it is more robust and easy to generalize to other situations.

(ii) In order to prove that the obtained method is functorial with respect to all regular morphisms one should prove that this is true for the input algorithm which operates with varieties. This turns out to be a bit subtle because there are regular morphisms between varieties defined over different fields $k$, and the classical method uses only $k$-derivations. One solution is to extend the framework by working with absolute derivations over $\QQ$. However, there is also a black box argument which proves that the classical method (and all methods we constructed in these notes) are compatible with arbitrary regular morphisms between varieties. This was worked out in \cite{BMT}.
\end{rem}

\subsubsection{Direct descent}
Finally, let us discuss a hypothetical direct descent. Assume, again that $X$ is a local regular qe scheme and $\hatX$ is its completion, but this time consider an arbitrary ideal $I$ on $X$ with completion $\hatI\subseteq\cO_{\hatX}$. The fact that $\hatI$ comes from $X$ makes it reasonable to hope that any regular functorial principalization of $\hatI$ should descent to $X$. We illustrate this with a wrong argument: if $Y=\hatX\times_X\hatX$ would be noetherian and qe, then the two pullbacks of the principalization to the completion $\hatY$ would coincide with the principalization of $I\cO_\hatY$ (the projections  $\hatY\to\hatX$ are regular in such case), and hence the principalization would descend to $X$ by the flat descent. Unfortunately, $Y$ is usually very far from being noetherian so a completely different argument is needed, but the question if some other form of descent to $X$ is possible seems very natural. I expect that if such descent works in the local case, then one should be able to patch the local solutions on an arbitrary qe $X$ using the N-property.

\appendix

\section{}

\subsection{Integral closure}\label{intclos}
Recall that the integral closure $I^\nor$ of an ideal $I\subseteq A$ consists of all elements $t\in A$ satisfying a monic equation $t^n+\sum_{i=1}^{n}a_it^{n-i}=0$ with $a_i\in I^i$. Any ideal $J$ with $I\subseteq J\subseteq I^\nor$ is called {\em integral } over $I$. Both notions are compatible with localizations and hence extend to sheaves of ideals on schemes. If $\cJ$ is integral over an ideal $\cI$ on a scheme $X$ and $Y\to X$ is a morphism, then $\cJ\cO_Y$ is integral over $\cI\cO_Y$.

\begin{lem}\label{intlem}
Let $X$ be a normal scheme with an invertible ideal $\cI$. Then for any ideal $\cJ$ with $\cI^\nor=\cJ^\nor$, one has that $\cI=\cJ$. In other words, $\cI$ is integrally closed and it is the integral closure only of itself.
\end{lem}
\begin{proof}
The claim can be checked locally, so assume that $A$ is a local normal domain, $I=(t)$ is principal and $J$ is an ideal with $I^\nor=J^\nor$. If $s$ is integral over $I$, then it satisfies an equation $s^n+\sum_{i=1}^n a_it^is^{n-i}=0$ and hence $s/t$ satisfies a monic equation over $A$. By the normality of $A$ we have that $s/t\in A$, that is $s\in I$. Thus, $I=I^\nor$.

Assume that $J\neq I^\nor=I$. Then $t\notin J$, but it is integral over $J$ and we take a monic equation $t^n+\sum_{i=1}^nb_it^{n-i}=0$ with $b_i\in J^i\subset I^i=(t^i)$ and minimal possible $n$. By the minimality of $n$ each element $c_i=b_i/t^i$ is not a unit, as otherwise $t^i-b_i/c_i=0$ would be an equation of smaller degree. Hence we obtain that $t^nu=0$, where $u=1+\sum_{i=1}^n c_i$ is a unit, and this yields a contradiction.
\end{proof}

The normality assumption is necessary in the lemma. For example, in $A=\Spec(k[x^2,x^3])$ the principal ideal $(x^2)$ has a non-invertible integral closure $(x^2,x^3)$.

\begin{cor}\label{intcor}
If $\cI$ and $\cJ$ are ideals on a normal scheme $X$ such that $\cJ^\nor=(\cI^d)^\nor$, then $(\Bl_\cI(X))^\nor=(\Bl_\cJ(X))^\nor$.
\end{cor}
\begin{proof}
The claim easily follows from the two particular cases: when $\cJ=\cI^d$, and when $\cJ$ is integral over $\cI$. In the first case, for any modification $Y\to X$ we have that $\cJ\cO_Y=(\cI\cO_Y)^d$ is invertible if and only if $\cI\cO_Y$ is invertible. So, by the universal property of blowings up $\Bl_\cI(X)=\Bl_{\cI^d}(X)$. If $\cJ$ is integral over $\cI$, then $\cJ\cO_Y$ is integral over $\cI\cO_Y$, so if $Y$ is normal and one of these ideals is principle, then the other is (and both coincide) by Lemma~\ref{intlem}. The second case follows.
\end{proof}

\bibliographystyle{amsalpha}

\bibliography{lognotes}

\end{document}